\documentclass[11pt]{amsart}
\usepackage{amscd, amsfonts, amssymb}
\usepackage{graphicx}
\usepackage{epstopdf}
\usepackage{xfrac}
\usepackage[breaklinks=true]{hyperref}
\usepackage{pinlabel}
\newtheorem{theorem}{Theorem}

\newtheorem{proposition}[theorem]{Proposition}

\newtheorem{corollary}[theorem]{Corollary}
\newtheorem{cor}[theorem]{Corollary}
\newtheorem{lemma}[theorem]{Lemma}
\newtheorem{lma}[theorem]{Lemma}

\theoremstyle{definition}
\newtheorem{definition}[theorem]{Definition}

\theoremstyle{remark}
\newtheorem{remark}[theorem]{Remark}
\newtheorem{rmk}[theorem]{Remark}

\numberwithin{theorem}{section}

\numberwithin{equation}{section}

\newenvironment{pf}{\begin{proof}}{\end{proof}}

\def\A{\mathcal{A}}
\def\d{\partial}
\def\R{\mathbb{R}}
\def\Z{\mathbb{Z}}
\def\F{\mathcal{F}}

\def\sgn{\operatorname{sgn}}

\def\op{\textrm{op}}

\newcommand{\C}{{\mathbb{C}}}
\newcommand{\Q}{{\mathbb{Q}}}

\newcommand{\Nn}{{\mathcal{N}}}

\newcommand{\act}{{\mathfrak{a}}}

\newcommand{\Ordo}{{\mathbf{O}}}

\newcommand{\M}{{\mathcal{M}}}
\newcommand{\CZ}{\operatorname{CZ}}

\newcommand{\inr}{\operatorname{int}}

\newcommand{\Lag}{\operatorname{Lag}}

\newcommand{\la}{\langle}
\newcommand{\ra}{\rangle}
\newcommand{\pa}{\partial}
\newcommand{\id}{\operatorname{id}}

\newcommand{\krn}{\operatorname{ker}}
\newcommand{\cokrn}{\operatorname{coker}}

\newcommand{\ol}{\overline}

\newcommand{\half}{\text{\rm \sfrac{1}{2}}}

\begin{document}

\title[Subcritical Legendrian Contact Homology]{Legendrian Contact Homology
in the Boundary of a Subcritical Weinstein $4$-Manifold}

\author{Tobias Ekholm}
\address{Uppsala University, Box 480, 751 06 Uppsala, Sweden\newline
\indent Insitute Mittag-Leffler, Aurav 17, 182 60 Djursholm, Sweden}
\email{tobias@math.uu.se}
\author{Lenhard Ng}
\address{Mathematics Department, Duke University, Durham, NC 27708-0320, USA}
\email{ng@math.duke.edu}

\begin{abstract}
We give a combinatorial description of the Legendrian contact homology
algebra associated to a Legendrian link in $S^1\times S^2$ or any
connected sum $\#^k(S^1\times S^2)$, viewed as the contact boundary of
the Weinstein manifold obtained by attaching 1-handles to the
4-ball. In view of the surgery formula for symplectic homology
\cite{bib:BEE}, this gives a combinatorial description of the
symplectic homology of any 4-dimensional Weinstein manifold (and of
the linearized contact homology of its boundary). We also study
examples and discuss the invariance of the Legendrian homology algebra
under deformations, from both the combinatorial and the analytical
perspectives.
\end{abstract}

\maketitle

\tableofcontents

%*********************************************************************
%*********************************************************************
\section{Introduction}\label{Sec:intro}
Legendrian contact homology is a part of Symplectic Field Theory,
which is a generalization of Gromov--Witten theory to a certain class
of noncompact symplectic manifolds including symplectizations of
contact manifolds. SFT contains holomorphic curve theories for contact
geometry, where Legendrian contact homology in a sense is the most
elementary building block. Although Legendrian contact homology is a
holomorphic curve theory, it is often computable as the homology of a
differential graded algebra (DGA) that can be described more simply
and combinatorially. For example, the DGA of a Legendrian $(n-1)$-link
in the contact $(2n-1)$-sphere at the boundary of a symplectic
$2n$-ball
can be computed in terms of Morse flow trees, see \cite{E-Morse}.

In a different direction, the computation of Legendrian contact
homology for 1-dimensional links in $\R^{3}$ was famously reduced by
Chekanov \cite{bib:Chekanov} to combinatorics of polygons determined
by the knot diagram of the link. The main goal of the current paper is
to generalize this combinatorial description to Legendrian links in
general boundaries of subcritical Weinstein $4$-manifolds. These
boundaries are topologically connected sums $\#^{k}(S^{1}\times
S^{2})$ of $k$ copies of $S^{1}\times S^{2}$ where $k$ is the rank of
the first homology of the subcritical Weinstein $4$-manifold; we
discuss Weinstein manifolds in more detail later in the introduction.

In Section \ref{ssec:dga}, we present a combinatorial model for the
Legendrian contact homology DGA for an arbitrary Legendrian link in
$\#^k(S^1\times S^2)$. This description follows Chekanov's, but
several new features are needed due to the presence of
$1$-handles. Perhaps most notably, the DGA is generated by a countably
infinite, rather than finite, set of generators (Reeb chords); cf.\
\cite{bib:Sabloffcirclebundle}, where infinitely many Reeb chords also
appear but in a different context. We accordingly present a
generalization of the usual notion due to Chekanov of equivalence of
DGAs, ``stable tame isomorphism'', to the infinite setting. Briefly,
Chekanov's stable tame isomorphisms involve finitely many
stabilizations and finitely many elementary automorphisms; here we
allow compositions of infinitely many of both, as long as they behave
well with respect to a filtration on the algebra.

One can view the combinatorial DGA abstractly as an invariant of
Legendrian links in $\#^{k}(S^{1}\times S^{2})$, and indeed one can
give a direct but somewhat involved algebraic proof that the DGA is
invariant under Legendrian isotopy, without reference to holomorphic
curves and contact homology. This is the content of
Theorem~\ref{thm:invariance} below. One can use the DGA, much as for
Legendrian links in $\R^3$, to extract geometric information: e.g.,
the DGA provides an obstruction to Legendrian links in
$\#^{k}(S^{1}\times S^{2})$ being destabilizable.

We then join the combinatorial and geometric sides of the story in our
main result, Theorem~\ref{thm:main}, which states that the
combinatorial DGA coincides with the DGA for Legendrian contact
homology, defined via holomorphic disks. This result has interesting
consequences for general 4-dimensional Weinstein manifolds that we
discuss next.

One of the main motivations for the study undertaken in this paper is
the surgery formula that expresses the symplectic homology of a
Weinstein manifold in terms of the Legendrian contact homology of the
attaching sphere of its critical handles, see \cite{bib:BEE}. Here, a
Weinstein manifold is a $2n$-dimensional symplectic manifold $X$ which
outside a compact subset agrees with $Y\times[0,\infty)$ for some
contact $(2n-1)$-manifold $Y$ and which has the following
properties. The symplectic form $\omega$ on $X$ is exact,
$\omega=d\lambda$, and agrees with the standard symplectization form
in the end $Y\times[0,\infty)$, $\lambda= e^{t}\alpha$, where
$t\in[0,\infty)$ and $\alpha$ is a contact form on $Y$; and the
Liouville vector field $Z$ $\omega$-dual to $\lambda$,
$\omega(Z,\cdot)=\lambda$, is gradient-like for some Morse function
$H\colon X\to\R$ with $H(y,t)=t$, $(y,t)\in Y\times[0,\infty)\subset
X$. The zeros of $Z$ are then exactly the critical points of $H$ and
the flow of $Z$ gives a finite handle decomposition for
$X$. Furthermore, since $Z$ is a Liouville vector field, the unstable
manifold of any zero of $Z$ is isotropic and hence the handles of $X$
have dimension at most $n$. The isotropic handles of dimension $<n$
are called subcritical and the Lagrangian handles of dimension $n$
critical.

A Weinstein manifold is called subcritical if all its handles are subcritical. The
symplectic topology of subcritical manifolds is rather easy to control. More precisely, any subcritical Weinstein $2n$-manifold $X$ is symplectomorphic to a product $X'\times \R^{2}$, where $X'$ is a Weinstein $(2n-2)$-manifold, see \cite[Section 14.4]{bib:CE}. Furthermore, any symplectic tangential
homotopy equivalence between two subcritical Weinstein manifolds is
homotopic to a symplectomorphism, see \cite[Sections
14.2--3]{bib:CE}. As a consequence of these results, the nontrivial
part of the symplectic topology of a Weinstein manifold is
concentrated in its critical handles. More precisely, a Weinstein
manifold $X$ is obtained from a subcritical Weinstein manifold $X_0$
by attaching critical handles along a collection of Legendrian
attaching spheres $\Lambda\colon \bigsqcup_{j=1}^{m}S^{n-1}\to Y_{0}$,
where $Y_0$ is the ideal contact boundary manifold of $X_0$. In
particular, the Legendrian isotopy type of the link $\Lambda$ in $Y_0$
thus determines $X$ up to symplectomorphism.

An important invariant of a Weinstein manifold $X$ is its symplectic
homology $SH(X)$, which is a certain limit of Hamiltonian Floer
homologies for Hamiltonians with prescribed behavior at
infinity. Symplectic homology and Legendrian contact homology are
connected: \cite[Corollary 5.7]{bib:BEE} expresses $SH(X)$ as the
Hochschild homology of the Legendrian homology DGA
$(\A(\Lambda),\d(\Lambda))$ of the Legendrian attaching link $\Lambda$
of its critical handles. Similarly, the linearized contact homology of
the ideal contact boundary of $X$ is expressed as the corresponding
cyclic homology, see \cite[Theorem 5.2]{bib:BEE}.

In view of the above discussion, Theorem~\ref{thm:main} then leads to
a combinatorial formulation for the symplectic homology of any
Weinstein 4-manifold (as well as the linearized contact homology of
its ideal boundary). As one consequence, we deduce a new proof of a
result of McLean \cite{bib:McLean} that states that there are exotic
Stein structures on $\R^8$. We note that our construction of exotic
Stein $\R^8$'s (and corresponding exotic contact $S^{7}$'s) is
somewhat different from McLean's.

Here is an outline of the paper.
Our combinatorial setup and computation of the DGA are presented in Section \ref{sec:comb} and sample calculations and applications are given in Section~\ref{sec:calc}.
In Section~\ref{sec:geom}, we set up the contact topology needed to define Legendrian contact homology in our context. This leads to the proof in Section~\ref{Sec:openclosed}
that the combinatorial formula indeed agrees with the holomorphic curve count in the definition of the DGA. In the Appendices, we demonstrate invariance of the DGA under Legendrian isotopy in two ways: from the analytical perspective in Appendix~\ref{app:aninv}, and from the combinatorial perspective in Appendix~\ref{app:combinv}, which also includes a couple of deferred proofs of results from Section~\ref{sec:comb}. It should be mentioned that the analytical invariance proof depends on a perturbation scheme for so-called M-polyfolds (the most basic level of polyfolds), the details of which are not yet worked out.

%*********************************************************************
\subsection*{Acknowledgments}

We thank Mohammed Abouzaid, Paul Seidel, and Ivan Smith for many
helpful discussions. TE was partially supported by the Knut and Alice
Wallenberg Foundation as a Wallenberg scholar. LN thanks Uppsala
University for its hospitality during visits in 2009, when this
project began, and 2010. LN was partially supported by NSF grants
DMS-0706777 and DMS-0846346.

%*********************************************************************
%*********************************************************************
\section{Combinatorial definition of the invariant}
\label{sec:comb}

In this section, we present a combinatorial definition of the DGA for
the Legendrian contact homology of a Legendrian link in $\#^k(S^1\times S^2)$
with the usual
Stein-fillable contact structure. (For the purposes of this paper, ``link'' means ``knot or link''.) We first
define a ``normal form'' in Section~\ref{ssec:stdform} for
presenting Legendrian links in
$\#^k(S^1\times S^2)$, and describe an easy algorithm in
Section~\ref{ssec:resolution} for deducing a
normal form from the front of a Legendrian link, akin to the
resolution procedure from \cite{bib:NgCLI}. We then define the DGA
associated to a Legendrian link in normal form, in two parts: in
Section~\ref{ssec:intdga} we
present a differential subalgebra, the ``internal DGA'', which is
associated to the
portion of the link inside the $1$-handles and depends only on the
number of strands of the link passing through each $1$-handle; then we
extend this in Section~\ref{ssec:dga} to a DGA that takes account of
the rest of the Legendrian link, with an example in Section~\ref{ssec:ex}.
Finally, in Section~\ref{ssec:sti} we present a version of stable
tame isomorphism, an equivalence relation on DGAs, which allows us to
state the algebraic invariance result for the DGA.

%*********************************************************************
\subsection{Normal form for the $xy$ projection of a Legendrian
  link}
\label{ssec:stdform}

As is the case in $\R^3$, it is most convenient to define the DGA for
a Legendrian link in $\#^k(S^1\times S^2)$ in terms of the projection
of the link (or the portion outside of the $1$-handles) in the $xy$
plane. Let $A,M>0$.

\begin{definition}
A tangle in $[0,A]\times [-M,M] \times [-M,M]$ is \textit{Legendrian}
if it is everywhere tangent to the standard contact structure
$dz-y\,dx$, where $x,y,z$ are the usual coordinates.
\label{def:stdform}
A Legendrian
tangle $\Lambda \subset [0,A]\times [-M,M] \times [-M,M]$ is
\textit{in normal form} if there exist integers $n_1,\ldots,n_k \geq
0$ such that
\begin{align*}
\Lambda \cap \{x=0\} &= \{(0,y_i^\ell,z_i^\ell)\,|\,
1\leq \ell\leq k,~ 1\leq i\leq n_\ell\}, \\
\Lambda \cap \{x=A\} &= \{(A,\tilde{y}_i^\ell,\tilde{z}_i^\ell)\,|\,
1\leq \ell\leq k,~ 1\leq i\leq n_\ell\},
\end{align*}
with $y_i^\ell,z_i^\ell,\tilde{y}_i^\ell,\tilde{z}_i^\ell \in [-M,M]$
satisfying the following conditions:
\begin{itemize}
\item
For some (arbitrarily small) $\epsilon>0$, for any fixed $\ell$,
each of the following sets lies in an interval of length less than
$\epsilon$: $\{y_1^\ell,\ldots,y_{n_\ell}^\ell\}$,
$\{z_1^\ell,\ldots,z_{n_\ell}^\ell\}$,
$\{\tilde{y}_1^\ell,\ldots,\tilde{y}_{n_\ell}^\ell\}$, and
$\{\tilde{z}_1^\ell,\ldots,\tilde{z}_{n_\ell}^\ell\}$;
\item
if $\ell_1<\ell_2$ then for all $i,j$,
\begin{align*}
y_i^{\ell_1} &> y_j^{\ell_2}, &
z_i^{\ell_1} &> z_j^{\ell_2}, \\
\tilde{y}_i^{\ell_1} &> \tilde{y}_j^{\ell_2}, &
\tilde{z}_i^{\ell_1} &> \tilde{z}_j^{\ell_2};
\end{align*}
\item
for $1\leq i<j\leq n_\ell$,
\begin{align*}
y_i^\ell &> y_j^\ell, &
z_i^\ell &> z_j^\ell, \\
\tilde{y}_i^\ell &< \tilde{y}_j^\ell, &
\tilde{z}_i^\ell &> \tilde{z}_j^\ell.
\end{align*}
\end{itemize}
\end{definition}

Less formally, $\Lambda$ meets $x=0$ and $x=A$ in $k$ groups of strands,
with groups of size $n_1,\ldots,n_k$. The groups are arranged from top to bottom in
both the $xy$ and the $xz$ projections. Within the $\ell^{\rm th}$ group, the
strands can be labeled by $1,\ldots,n_\ell$ in such a way that the
strands appear in increasing order from top to bottom in both the $xy$
and $xz$ projections at $x=0$ and in the $xz$ projection at $x=A$, and
from bottom to top in the $xy$ projection at $x=A$.

Any Legendrian tangle in normal form corresponds to a Legendrian
link in $\#^k(S^1\times S^2)$ by attaching $k$ $1$-handles joining the
portions of the $xz$ projection of the tangle at $x=0$ to the portions
at $x=A$. The $\ell^{\rm th}$ $1$-handle joins the $\ell^{\rm th}$ group at $x=0$
to the $\ell^{\rm th}$ group at $x=A$, and within this group, the strands
with the same label at $x=0$ and $x=A$ are connected through the
$1$-handle.

See Figure~\ref{fig:stdform} for an illustration of the $xy$ and $xz$
projections of a Legendrian tangle in normal form. Note that the
$xy$ projection can be deduced from the $xz$ projection as usual by
setting $y=dz/dx$.
%but the converse does not hold.

\begin{figure}
\centerline{
\includegraphics[width=\textwidth]{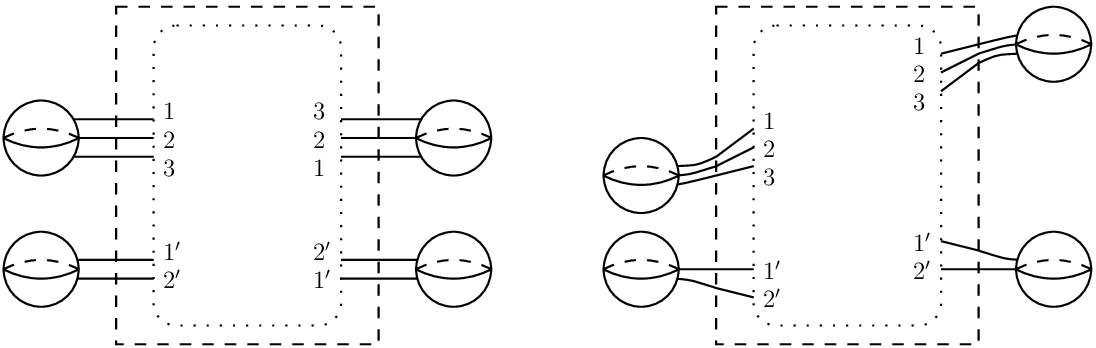}
}
\caption{
The $xy$ (left) and $xz$ (right) projections of a Legendrian tangle in
normal form. The dashed boxes are $[0,A] \times [-M,M]$ in $(x,y)$ or
$(x,z)$ coordinates, and the tangle continues into the dotted boxes. In
the terminology of Definition~\ref{def:stdform},
the strands labeled $1$, $2$, $3$ intersect the left side of the
dashed boxes at $(0,y_1^1,z_1^1)$, $(0,y_2^1,z_2^1)$,
$(0,y_3^1,z_3^1)$ and the right side at
$(A,\tilde{y}_1^1,\tilde{z}_1^1)$,
$(A,\tilde{y}_2^1,\tilde{z}_2^1)$,
$(A,\tilde{y}_3^1,\tilde{z}_3^1)$,
respectively; strands $1'$, $2'$ intersect the left side at
$(0,y_1^2,z_1^2)$, $(0,y_2^2,z_2^2)$ and the right side at
$(A,\tilde{y}_1^2,\tilde{z}_1^2)$,
$(A,\tilde{y}_2^2,\tilde{z}_2^2)$.
}
\label{fig:stdform}
\end{figure}

\begin{definition}
A tangle diagram is in \textit{$xy$-normal form} if it is the $xy$
projection of a Legendrian tangle in normal form.
\label{def:xystdform}
\end{definition}

In Section~\ref{ssec:dga}, we will associate a differential graded algebra to a
tangle diagram in $xy$-normal form.

%*********************************************************************
\subsection{Resolution}
\label{ssec:resolution}

As in the case of $\R^3$ \cite{bib:Chekanov}, it is not necessarily
easy to tell whether a tangle diagram is (planar isotopic to a
diagram) in $xy$-normal form. However, in practice a Legendrian link in
$\#^k(S^1\times S^2)$ is typically presented as a front diagram,
following Gompf \cite{bib:Gompf}. In this subsection, we
describe a procedure called \textit{resolution} that inputs a front
diagram for a Legendrian link in $\#^k(S^1\times S^2)$, and outputs a
tangle diagram in $xy$-normal form that represents a Legendrian-isotopic link.

Gompf represents a Legendrian link in $\#^k(S^1\times S^2)$ by a front
diagram in a box $[0,A] \times [-M,M]$ that is nearly identical in form to
the $xz$ projections of our normal-form tangles from
Definition~\ref{def:stdform}, but with the intersections with $x=0$
and $x=A$ aligned horizontally; see Figure~\ref{fig:gompf} for an
illustration. We say that such a front is in \textit{Gompf standard
  form}.

\begin{figure}
\labellist
\small\hair 2pt
\pinlabel {tangle} at 127 64
\pinlabel {tangle} at 415 64
\endlabellist
\centerline{\includegraphics[width=\textwidth]{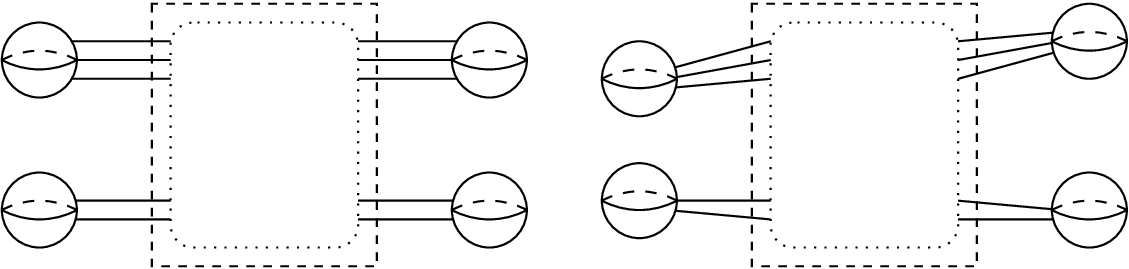}}
\caption{
A Legendrian front diagram in $\#^2(S^1\times S^2)$ in Gompf standard form
(left), and perturbed to be the $xz$ projection of a tangle in
normal form in the sense of Definition~\ref{def:stdform} (right).
}
\label{fig:gompf}
\end{figure}

Any front in Gompf standard form can be perturbed to be the $xz$
projection of a tangle in normal form. This merely involves perturbing the
portions of the front near $x=0$ and $x=A$ so that, rather than being
horizontal, they are nearly horizontal but with slopes increasing from
bottom to top along $x=0$, and increasing from bottom to top along
$x=A$ except decreasing within each group of strands corresponding to
a $1$-handle. The resulting front then gives a Legendrian tangle whose
$xy$ projection satisfies the ordering condition of
Definition~\ref{def:stdform}. See Figure~\ref{fig:gompf}.

Although one can deduce the $xy$ projection of a Legendrian tangle from its $xz$ projection by using $y = dz/dx$, this can be somewhat difficult to effect in practice. However, as in $\R^3$ \cite{bib:NgCLI}, if we allow the tangle to vary by Legendrian isotopy (in fact, planar isotopy in the $xz$ plane), then it is possible to obtain a front whose corresponding $xy$ projection is easy to describe.

\begin{figure}
\centerline{\includegraphics[width=\textwidth]{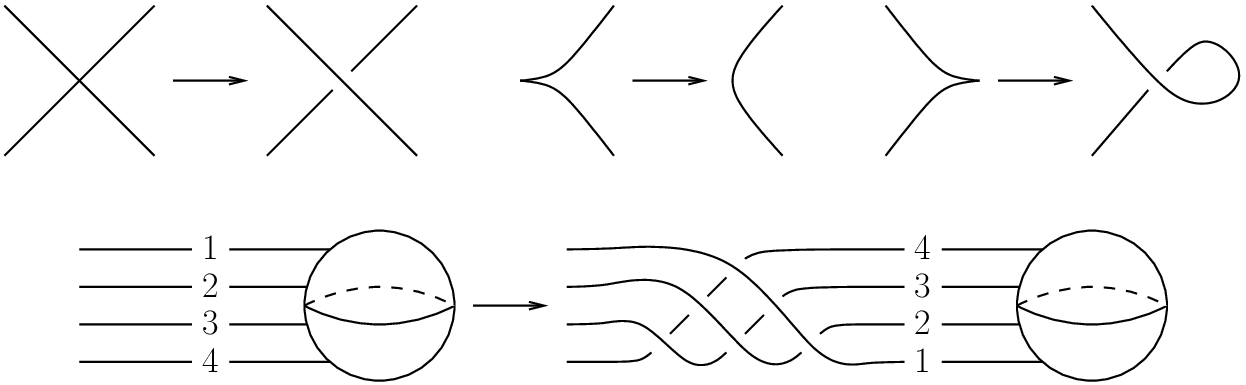}}
\caption{
Resolving a front in Gompf standard form.
}
\label{fig:htwist}
\end{figure}

\begin{definition}
The \textit{resolution} of a front in Gompf standard form is the tangle diagram obtained by resolving the singularities of the front as shown in Figure~\ref{fig:htwist} and, for each $1$-handle, adding a half-twist to the strands that pass through that $1$-handle at the $x=A$ end of the tangle.
\end{definition}

\begin{figure}
\centerline{\includegraphics[width=\textwidth]{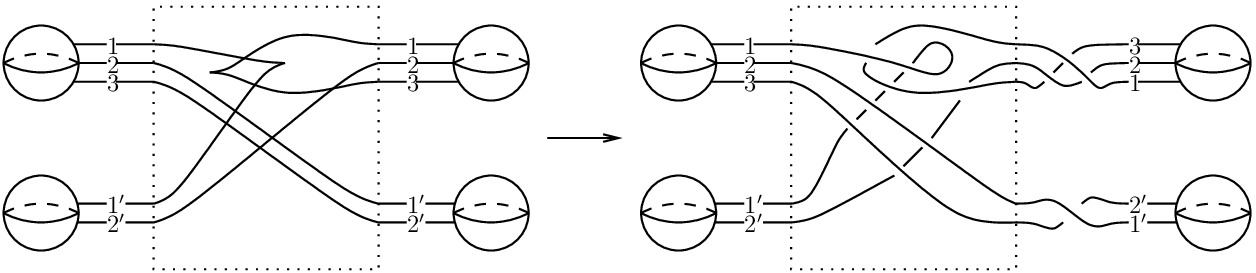}}
\caption{
Example of front resolution for a Legendrian link in $\#^2(S^1\times S^2)$.
}
\label{fig:resolve-ex}
\end{figure}

Note that the half-twist has the effect of reversing the order of the strands entering the $1$-handles at $x=A$. In Gompf standard form, the strands entering a $1$-handle from the left and the right are identified with each other in the obvious way, by identifying $x=0$ and $x=A$; in the resolution of such a front, the top strand entering a particular $1$-handle from the left is identified with the bottom strand entering the $1$-handle from the right, and so forth. See Figure~\ref{fig:resolve-ex} for an example of a resolution.

\begin{proposition}%[cf.\ \cite{bib:NgCLI}]
Let $\Lambda$ be a Legendrian link in $\#^k(S^1\times S^2)$
represented by a front in Gompf standard form.
\label{prop:resolution}
Then the resolution of
the front is (up to planar isotopy) the $xy$ projection of a
Legendrian link in normal form that is Legendrian isotopic to $\Lambda$.
\end{proposition}

See Appendix~\ref{app:d2proof} for the proof of Proposition~\ref{prop:resolution}, which is very similar to the analogous proof for resolutions in $\R^3$ from \cite{bib:NgCLI}.

In practice, to compute the DGA associated to a Legendrian link in
$\#^k(S^1\times S^2)$, we begin with a Gompf standard form for the
link, resolve it as above, and then apply the combinatorial formula
for the DGA to be described in Section~\ref{ssec:dga} below.

%*********************************************************************
\subsection{Internal differential graded algebra}
\label{ssec:intdga}

Here we present the subalgebra of the contact homology differential
graded algebra determined
by the Reeb chords in each $1$-handle, and holomorphic disks with positive
puncture at one of these Reeb chords. Let $n\geq 1$ be an integer,
representing the number of strands of the Legendrian link $\Lambda=\Lambda_1\cup\dots\cup\Lambda_s$ passing
through the $1$-handle, where the $\Lambda_j$ denote the knot components of $\Lambda$. To define a grading on the subalgebra, we need
two auxiliary pieces of data: an $s$-tuple of integers $(r_1,\ldots,r_s)$ associated to the components $\Lambda_j$, $j=1,\dots,s$, and an $n$-tuple of integers
$(m(1),\ldots,m(n))$. These represent the rotation numbers of the
Legendrian link (which only appear here in the grading of the
homology variables $t_j$) and a choice of Maslov potential for
each of the strands passing through the $1$-handle; see also
Section~\ref{ssec:dga}.

Given $(r_1,\dots, r_s)$ and $(m(1),\ldots,m(n))$, let $(\A_n,\d_n)$ denote the differential graded algebra
given as follows. As an algebra, $\A_n$ is the tensor algebra over the coefficient ring $\Z[H_1(\Lambda)]$, that is,
\[
\Z[\mathbf{t},\mathbf{t}^{-1}]:=\Z[t_1,t_1^{-1},\dots, t_s,t_s^{-1}],
\]
freely
generated by generators $c_{ij}^0$ for $1\leq i<j\leq n$ and
$c_{ij}^p$ for $1\leq i,j\leq n$ and $p \geq 1$. (See Remark~\ref{rmk:coeffs} for a discussion of the geometric significance of the coefficient ring.)
This algebra is
graded by setting $|t_j|=-2r_j$, $|t_j^{-1}|=2r_j$, and
\[
|c_{ij}^p| = 2p-1+m(i)-m(j)
\]
for all $i,j,p$. The differential
$\d_n$ is defined on generators by
\begin{align*}
\d_n(c_{ij}^0) &= \sum_{m=1}^n \sigma_i\sigma_m c_{im}^0 c_{mj}^0 \\
\d_n(c_{ij}^1) &= \delta_{ij} + \sum_{m=1}^n \sigma_i\sigma_m c_{im}^0 c_{mj}^1 + \sum_{m=1}^n
\sigma_i\sigma_m c_{im}^1 c_{mj}^0 \\
\d_n(c_{ij}^p) &= \sum_{\ell=0}^p \sum_{m=1}^n \sigma_i\sigma_m c_{im}^\ell c_{mj}^{p-\ell}
\end{align*}
where $p \geq 2$, $\sigma_i = (-1)^{m(i)}$ for all $i$, $\delta_{ij}$ is the Kronecker delta, and
we set $c_{ij}^0 = 0$ for $i \geq j$. Extend $\d_n$ to all of $\A_n$ in the usual way by the Leibniz rule
\[
\d_n(xy) = (\d_nx)y + (-1)^{|x|} x (\d_ny).
\]

It is clear that $\d_n$ has degree $-1$, and easy to check that $\d_n^2=0$. Note that $\A_n$ is
infinitely generated as an algebra but has a natural increasing
filtration given by the superscripts, with respect to which $\d_n$ is a filtered
differential.

Given a Legendrian link $\Lambda \subset \#^k(S^1\times S^2)$, we can associate a DGA $(\A_{n_i},\d_{n_i})$ as above to each of the $k$ $1$-handles; see Section~\ref{ssec:dga} below. We refer to the DGA whose generators are the collection of generators of $\A_{n_i}$, $i=1,\ldots,k$, and whose differential is induced from $\d_{n_i}$, as the \textit{internal DGA} of $\Lambda$.

%*********************************************************************
\subsection{The DGA for an $xy$ projection in normal form}
\label{ssec:dga}

We can now define the DGA associated to a Legendrian link in
$\#^k(S^1\times S^2)$, or more precisely to a tangle in $xy$-normal
form in the terminology of Section~\ref{ssec:stdform}, with one base point for each link component.

Suppose that we have a Legendrian link $\Lambda=\Lambda_1\cup\dots\cup\Lambda_s \subset \#^k(S^1\times
S^2)$ in normal form; then its projection $\pi_{xy}(\Lambda)$ to the
$xy$ plane is a tangle diagram in $xy$-normal form. Let $a_1,\dots, a_n$ denote the crossings of the tangle diagram. Label the $k$
$1$-handles appearing in the diagram by $1,\ldots,k$ from top to
bottom; let $n_\ell$ denote the number of strands of the tangle
passing through handle $\ell$. For each $\ell$, label the strands
running into the $1$-handle on the left side of the diagram by
$1,\ldots,n_\ell$ from top to bottom, and label the strands running
into the $1$-handle on the right side by $1,\ldots,n_\ell$
\textit{from bottom to top}. Also choose base points $\ast_j$, $j=1,\dots,s$ in the tangle diagram such that $\ast_j$ lies on component $\Lambda_j$ for all $j=1,\ldots,s$, and no $\ast_j$ lies at any of the crossings (or in the $1$-handles).

We can now define the DGA. Our definition involves three parts: the
algebra, the grading, and the differential.

\subsubsection{The algebra}
Let
$\A$ be the tensor algebra over
$\Z[\mathbf{t},\mathbf{t}^{-1}]=\Z[t_1,t_1^{-1},\dots,t_s,t_s^{-1}]$ freely generated by:
\begin{itemize}
\item
$a_1,\ldots,a_n$;
%\item
%$d_{ij;\ell}$ for $1\leq \ell\leq k$ and $1\leq i<j\leq n_\ell$;
\item
$c_{ij;\ell}^0$ for $1\leq \ell\leq k$ and $1\leq i<j\leq n_\ell$;
\item
$c_{ij;\ell}^p$ for $1\leq \ell\leq k$, $p>0$, and $1\leq i,j\leq
n_\ell$.
\end{itemize}
(We will drop the index $j$ in $t_j$ when the Legendrian $\Lambda$ is a single-component knot, and the index $\ell$ in $c_{ij;\ell}^p$ when there is only one $1$-handle.)
Note that $\A$ contains $k$ subalgebras
$\A_{n_1}^1,\ldots,\A_{n_k}^k$, where $\A_{n_\ell}^\ell$ is the tensor
algebra over $\Z[\mathbf{t},\mathbf{t}^{-1}]$ freely generated by $c_{ij;\ell}^0$ for
$1\leq i<j\leq n_\ell$ and $c_{ij;\ell}^p$ for $p>0$ and $1\leq i,j\leq
n_\ell$. Each of these subalgebras $\A_{n_\ell}^\ell$ should be
thought of as the internal DGA corresponding to the $\ell^{\rm th}$ handle,
and the grading and differential on this subalgebra will be defined
accordingly. Together, the $\A_{n_\ell}^\ell$ generate a differential subalgebra of $\A$, which we call ``the'' internal DGA of $\Lambda$.

It should be noted that we have chosen our formulation of the algebra
in such a way that all of the $t^{\pm 1}_j$ commute with each other and with the generators $a_i$ and
$c_{ij;\ell}^p$. This suffices for our purposes but is not strictly
necessary. It is possible to elaborate on the construction and consider
an algebra over $\Z$ generated by $a_i$, $c_{ij;\ell}^p$, and $t^{\pm
  1}_j$, modulo only the obvious relations $t_j\cdot t_j^{-1} = t^{-1}_j\cdot t_j
= 1$. See, e.g., \cite[Section~2.3.2]{bib:EENS} or \cite[Remark~2.2]{bib:Ngsurvey} for more discussion.

\subsubsection{Grading}
The grading on $\A$ is determined by stipulating a grading on $t_j$ and
on each generator $a_i$ and $c_{ij;\ell}^p$ of $\A$. We will do each of these in turn.

We begin with some preliminary definitions. A \textit{path} in
$\pi_{xy}(\Lambda)$ is a path that traverses some amount of
$\pi_{xy}(\Lambda)$, connected except for points where it enters a
$1$-handle (i.e., approaches $x=0$ or $x=A$ along a labeled strand)
and exits the $1$-handle
along the corresponding strand (i.e., departs $x=A$ or $x=0$ along the
strand with the same label). In particular, the tangent vector in
$\R^2$ to a path varies continuously as we traverse the path (note that the
strands entering or exiting a $1$-handle are horizontal). The
\textit{rotation number} $r(\gamma)$ of a path $\gamma$ consisting of unit vectors in $\R^2$ (i.e., points in $S^1$) is the number of counterclockwise revolutions made by $\gamma(t)$ around $S^1$ as we traverse the path (i.e., the total curvature $\int_\gamma\kappa\,ds$ divided by $2\pi$); note that this is generally a real number, and is an integer if and only if $\gamma$ is closed. By slight abuse of notation, we will often speak of the rotation number of a path $\gamma$ in $\pi_{xy}(\Lambda)$ to mean the rotation number of its unit tangent vector $\gamma'(t)/|\gamma'(t)|$.

In this terminology, the rotation number $r_j=r(\Lambda_j)$ is the rotation
number of the path in $\pi_{xy}(\Lambda)$ that begins and ends at the
base point $\ast_j$ on the $j^{\rm th}$ component $\Lambda_j$ and traverses the diagram once in the direction of
the orientation of $\Lambda_j$. We define
\[
|t_j| = -2r(\Lambda_j).
\]

To define the remainder of the grading on $\A$, we need to make some auxiliary choices joining tangent directions to the various base points $\ast_1,\ldots,\ast_s$, although the grading only depends on these choices in the case of multi-component links ($s\geq 2$). For $i=1,\ldots,s$, let $v_i \in S^1$ denote the unit tangent vector to the oriented curve $\pi_{xy}(\Lambda_i)$ at the base point $\ast_i$. Now for $i=2,\ldots,s$, pick a path $\xi_i$ in $S^1$ from $v_1$ to $v_i$; then for $1\leq i,j\leq s$, let $\xi_{ij}$ denote the path in $S^1$ from $v_i$ to $v_j$ given by the orientation reverse of $\xi_i$ followed by $\xi_j$, considered up to homotopy (so we can choose $\xi_{ii}$ to be constant). Note that there is a $\Z^{s-1}$ worth of possible choices for the paths $\xi_{ij}$.

With $\xi_{ij}$ chosen, we next define the grading of the $a_i$ generators. Let $a_i^+$ and $a_i^-$ denote the preimages $a_i^{+}$ and $a_i^{-}$ in $\Lambda$ of the crossing point $a_i$ in the upper and lower strands of the crossing, respectively, and suppose that $a_i^+,a_i^-$ belong to components $\Lambda_{i^+},\Lambda_{i^-}$, respectively. There are unique paths $\gamma^{\pm}$ in the $xy$ projection of the component $\Lambda_{i^{\pm}}$ connecting $a_i^{\pm}$ to the base point $\ast_{i^{\pm}}$ and following the orientation of $\Lambda_{i^{\pm}}$, such that the lifts of $\gamma^{\pm}$ to $\Lambda$ are embedded.
Let $\sigma_{i}$ be the path of unit tangent vectors along $\gamma^+$, followed by $\xi_{j^{+}j^{-}}$, and followed by the path of unit tangent vectors along $\gamma^-$ traversed backwards. Assume that the crossing at $a_i$ is
transverse (else perturb the diagram); then $r(\sigma_i)$ is neither
an integer nor a half-integer, and we define
\[
|a_i| = \lfloor 2r(\sigma_i) \rfloor,
\]
where $\lfloor x \rfloor$ denotes the largest integer smaller than $x$.

It remains to define the grading of the $c_{ij;\ell}^p$ generators. This can be done by adding dips and treating $c_{ij;\ell}^0$ generators as crossings in a dipped diagram; cf.\ the proof of Proposition~\ref{prop:d2} in Appendix~\ref{app:d2proof}. We however use a slightly different approach here.
Choose a \textit{Maslov potential} $m$ that associates an integer to each strand passing through each $1$-handle, in such a way that the following conditions hold:
\begin{itemize}
\item
if $S_L,S_R$ are strands on the left and right of $\pi_{xy}(\Lambda)$ that correspond to the ends of a strand of $\Lambda$ passing through a $1$-handle, then $m(S_L) = m(S_R)$, and these Maslov potentials are even if $\Lambda$ is oriented left to right (i.e., it passes through the $1$-handle from $x=A$ to $x=0$), and odd if $\Lambda$ is oriented right to left;
\item
if $S,S'$ are endpoints of strands through $1$-handles with $S \in \Lambda_i$ and $S'\in\Lambda_j$, such that $\Lambda_i,\Lambda_j$ are oriented from $S,S'$ to $\ast_i,\ast_j$ respectively, then
\[
m(S')-m(S) = -2r(\sigma)
\]
where $\sigma$ is the path of unit tangent vectors along $\Lambda_i$ from $S$ to $\ast_i$, followed by $\xi_{ij}$, followed by the path of unit tangent vectors along $\Lambda_j$ from $\ast_j$ to $S'$; note that this last path is traversed opposite to the orientation on $\Lambda_j$, and that $r(\sigma) \in \frac{1}{2}\Z$ since strands are horizontal as they pass through $1$-handles.
\end{itemize}
It is easy to check that the Maslov potential is well-defined (given choices for $\xi_{ij}$) up to an overall shift by an even integer.

As suggested by Section~\ref{ssec:intdga}, we now grade the ``internal
generators'' $c_{ij;\ell}^p$ of $\A$ as follows:
\[
|c_{ij;\ell}^p| = 2p-1+m(S_{i;\ell})-m(S_{j;\ell}),
\]
where $S_{i;\ell},S_{j;\ell}$
are the strands running through handle $\ell$ labeled by $i$ and $j$. This completes the definition of the grading on $\A$.

\begin{remark}
Note that the grading on $\A$ is independent of the choice of Maslov potential. Different choices of $\gamma_{ij}$ do however lead to different gradings for $s\geq 2$. As mentioned previously, there is a $(2\Z)^{s-1}$ worth of choices for $\xi_{ij}$. Given the grading $|\cdot|$ on $\A$ resulting from such a choice, and any $(2n_1,\ldots,2n_s) \in (2\Z)^s$, one can obtain another grading $|\cdot|'$ on $\A$ by defining
\[
|a_i|' = |a_i| + 2(n_{i^+}-n_{i^-})
\]
and similarly for the other generators $c_{ij;\ell}^p$. (The grading on the homology generators $t_1,\ldots,t_s$ is unchanged.) Any such grading comes from a different choice of $\xi_{ij}$, and conversely.
\end{remark}

\begin{remark}
If $\pi_{xy}(\Lambda)$ is the resolution of a front diagram of an $s$-component link, then we
can calculate the grading on generators of $\A$ directly from the
front diagram, as follows. We can associate a Maslov potential to
connected components of a
front diagram minus cusps and the base points $\ast_j$, $1,\dots,s$, in such a way
that the following conditions hold:
\begin{itemize}
\item
the same Maslov potential is assigned to the left and right sides of
the same strand (connected through a $1$-handle), and this potential is even if the strand is oriented left to right (from $x=A$ to $x=0$) and odd otherwise;
\item
at a cusp, the upper component (in the $z$ direction) has Maslov
potential one more than the lower component.
\end{itemize}
As before, we set $|t_j| = -2r(\Lambda_j)$ and $|c_{ij;\ell}^p| =
2p-1+m(S_{i;\ell})-m(S_{j;\ell})$. The other generators of $\A$ are in
one-to-one correspondence to: right cusps in
the front; crossings in the front; and pairs of strands entering the
same $1$-handle at $x=A$
(corresponding to the half-twists in the resolution).
Let $a$ be one of these generators. If $a$ is a right cusp, define
$|a|=1$ (this assumes that no base point $\ast_j$ is in the portion
of the resolution given by the loop at $a$). If $a$ is a crossing,
then $|a| = m(S_o)-m(S_u)$, where $S_u$ is the undercrossing strand at
$a$ (in the front projection, i.e., the strand with more positive
slope) and $S_o$ is the overcrossing strand at $a$. Finally, if $a$ is
a crossing in the half-twist near a handle and involving strands labeled $i$ and $j$ with $i<j$, then
\[
|a| = m(S_i)-m(S_j)
\]
where $S_i$ and $S_j$ are the strands labeled $i$ and $j$ at the
$1$-handle.
\end{remark}

\subsubsection{Differential}
Finally, we define the differential $\partial$ on $\A$. It suffices to
define the differential on generators of $\A$, and then impose the
Leibniz rule. We set $\partial(t_j) = \partial(t^{-1}_j) = 0$. On each
$\A_{n_\ell}^\ell$, we then define the differential by $\partial
= \partial_{n_\ell}$, as defined in Section~\ref{ssec:intdga}.

It remains to define the differential for crossings $a_i$. To do this,
decorate
the quadrants of the crossings in $\pi_{xy}(\Lambda)$ by
the ``Reeb signs'' shown in the left diagram in Figure~\ref{fig:signs},
as in Chekanov \cite{bib:Chekanov}.

\begin{figure}
\centerline{
\includegraphics[width=3.5in]{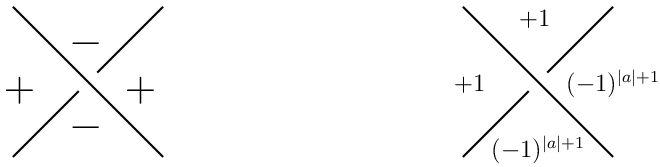}
}
\caption{Reeb signs on the left; orientation signs at a crossing $a$ on the right.}
\label{fig:signs}
\end{figure}

For $r\geq 0$, let $a_i,b_1,\ldots,b_r$ be a collection of (not
necessarily distinct) generators
of $\A$ such that $a_i$ is a crossing in
$\pi_{xy}(\Lambda)$, and each of $b_1,\ldots,b_r$ is either a crossing
$a_j$ or a generator of the form $c_{j_1j_2;\ell}^0$ for $j_1<j_2$. Define
$\Delta(a_i;b_1,\ldots,b_r)$ to be the set of immersed disks with convex corners (up to
parametrization) with boundary on $\pi_{xy}(\Lambda)$, such that the
corners of the disk are, in order as we traverse the
boundary counterclockwise, a ``positive corner'' at $a_i$ and
``negative corners'' at each of $b_1,\ldots,b_r$. Here positive and
negative corners are as depicted in Figure~\ref{fig:corners}.
Disks are not allowed to pass through a $1$-handle, but they can have
a negative corner $c_{j_1j_2;\ell}^0$ on either side of the $1$-handle.

\begin{figure}
\centerline{
\includegraphics[width=\textwidth]{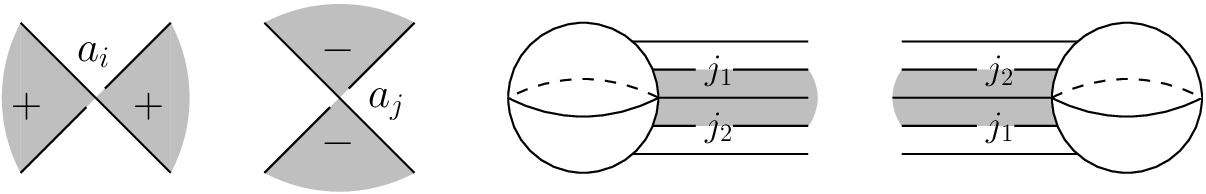}
}
\caption{
Positive and negative corners for a disk (locally depicted by one of
the shaded regions). From left to right: two possible positive corners at a crossing
$a_i$; two possible negative corners at $a_j$; and two possible negative corners at
$c_{j_1j_2;\ell}^0$, where $j_1<j_2$ and $\ell$ is the label of the
depicted handle.
}
\label{fig:corners}
\end{figure}

We now set
\[
\d(a_i) = \sum_{r \geq 0} ~\sum_{b_1,\ldots,b_r} ~\sum_{\Delta\in
\Delta(a_i;b_1,\ldots,b_r)} \sgn(\Delta)\, t_1^{-n_1(\Delta)}\cdots t_s^{-n_s(\Delta)}\, b_1\cdots b_r,
\]
where $n_j(\Delta)$ is the signed number of times that the boundary of
$\Delta$ passes through $\ast_j$, and $\sgn(\Delta)$ is a sign to be
defined below. Extend $\d$ to $\A$ via the Leibniz rule. For any
crossing $a_i$, the set of all possible immersed disks with $+$ corner
at $a_i$ and any number of $-$ corners is finite, by the usual area
argument (or see the proof of Proposition~\ref{prop:d2} below), and so the sum in $\d(a_i)$ is finite.

To define the sign associated to an immersed disk, we assign
``orientation signs'' (entirely distinct from Reeb signs) at every
corner of the disk, as follows. For corners at a $c_{j_1j_2;\ell}^0$, we
associate the orientation sign:
\begin{itemize}
\item
$+1$ for a corner reaching the handle from the
right (at $x=0$; see the second diagram from the right in
Figure~\ref{fig:corners});
\item
 $(-1)^{|c_{j_1j_2;\ell}^0|+1} = (-1)^{m(j_1)-m(j_2)}$ for
  a corner reaching the handle from the left (at $x=A$; see the
  rightmost diagram in Figure~\ref{fig:corners}), where
  $m(j_1),m(j_2)$ are the Maslov potentials associated to strands
  $j_1,j_2$.
\end{itemize}

Next we consider corners at crossings of $\pi_{xy}(\Lambda)$.
At a crossing of odd degree, all orientation signs are $+1$. At a crossing of even degree, there are two possible choices for assigning two $+1$ and two $-1$ signs to the corners (corresponding to rotating the diagram in Figure~\ref{fig:signs} by $180^\circ$), with only the stipulation that adjacent corners on the same side of the understrand have the same sign; either choice will do, and the two choices are related by an algebra automorphism.
For the sake of definiteness, in computations involving the resolution
of a front, we will take the $-1$ corners to be the south and east
corners at every corner of even degree.

Finally, for an immersed disk $\Delta$ with corners, we set
$\sgn(\Delta)$ to be the product of the orientation signs at all
corners of $\Delta$. This completes the definition of the differential $\partial$.

\begin{remark}\label{rem:signconv}
Our sign convention agrees with the convention in
\cite{bib:ENS}, up to an algebra automorphism that multiplies some
even crossings by $-1$. For disks that do not pass through the
$1$-handles and do not involve the $c$ generators or half-twist crossings, this agrees precisely with the convention in \cite{bib:NgCLI}.

Furthermore, our orientation scheme is induced from the non-null-cobordant spin structures on the circle components of the link. For calculations related to symplectic homology it is important to use the null-cobordant spin structure since the link components are boundaries of the core disks of the handle and we need to orient moduli spaces of holomorphic disks with boundaries on these in a consistent way. From an algebraic point of view the change in our formulas are minor: changing the spin structure on the component $\Lambda_j$ corresponds to substituting $t_j^{\pm}$ in the formulas above by $-t_j^{\pm}$. We refer to \cite[Section 4.4s]{EES-orientation} for a detailed discussion.
\end{remark}

With the definition of $(\A,\d)$ in hand, we conclude this subsection by stating the usual basic facts about the differential.
%
%We next verify combinatorially that the map $\partial$ is actually a differential. (An alternative proof of that fact relating the differential to moduli spaces of holomorphic disks is presented in Section \ref{}.)

\begin{proposition}
The map $\partial$ has degree $-1$ and is a differential, $\partial^2=0$.
\label{prop:d2}
\end{proposition}

Proposition~\ref{prop:d2} can be proven either combinatorially or geometrically. The combinatorial proof is based on the proof of the analogous result in $\R^3$ and is deferred to Section~\ref{app:d2proof}.
The geometric proof relates the differential to moduli spaces of holomorphic disks, in the usual Floer-theoretic way, see Remark \ref{rmk:d^2=0}.

%*********************************************************************
\subsection{An example}
\label{ssec:ex}

To illustrate the definition in Section \ref{ssec:dga},
we describe the differential graded algebra associated to the Legendrian knot
in Figure~\ref{fig:torus-example2}. Note that this example appears in
\cite[Figure~36]{bib:Gompf}, in the context of constructing a
Stein structure on $T^2\times D^2$.

\begin{figure}
\centerline{
\includegraphics[width=\textwidth]{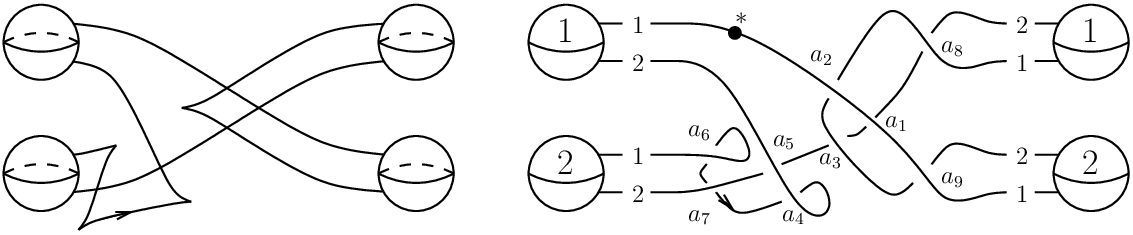}
}
\caption{An oriented Legendrian knot in $\#^2 (S^1\times S^2)$: front
  in Gompf standard form (left), and resolution with base point $\ast$ (right).}
\label{fig:torus-example2}
\end{figure}

The knot has $tb=1$ and $r=0$ (note that $tb$ is well-defined since
the knot is null-homologous). The differential graded algebra $\A$
associated to the knot has generators
$a_1,\dots,a_9,c_{12}^0,\tilde{c}_{12}^0$, and
$c_{ij}^p,\tilde{c}_{ij}^p$ for $1\leq i,j\leq 2$ and $p\geq 1$, with
grading
\begin{gather*}
|a_1|=|a_4|=|a_6|=|a_8|=|a_9|=1, \\
|a_2|=|a_3|=|a_5|=|a_7|=|c_{12}^0|=|\tilde{c}_{12}^0|=0, \\
|c_{11}^p|=|\tilde{c}_{11}^p|=|c_{22}^p|=|\tilde{c}_{22}^p|=2p-1,\\
|c_{12}^p|= |\tilde{c}_{12}^p|
= 2p,\\
|c_{21}^p| = |\tilde{c}_{21}^p| = 2p-2.
\end{gather*}
(Here for notational simplicity we have dropped the second subscript
on $c_{ij;\ell}^p$ and instead write $c_{ij}^p := c_{ij;1}^p$,
$\tilde{c}_{ij}^p := c_{ij;2}^p$.)

The differential on $\A$ is given by
\begin{align*}
\d(a_1) &= -a_2a_3+t^{-1}c_{12}^0 a_5 &
\d(c_{12}^0) &= \d(\tilde{c}_{12}^0) = 0 \\
\d(a_4) &= 1+a_5a_7 & \d(c_{21}^1) &= \d(\tilde{c}_{21}^1) = 0 \\
\d(a_6) &= 1+\tilde{c}_{12}^0 a_7 & \d(c_{11}^1) &= 1-c_{12}^0c_{21}^1 \\
\d(a_8) &= -c_{12}^0+a_3 &
\d(\tilde{c}_{11}^1) &= 1-\tilde{c}_{12}^0\tilde{c}_{21}^1 \\
\d(a_9) &= -\tilde{c}_{12}^0-a_2 & \d(c_{22}^1) &= 1-c_{21}^1c_{12}^0 \\
\d(a_i) &= 0,~~~~i\neq 1,4,6,8,9  &
\d(\tilde{c}_{22}^1) &= 1-\tilde{c}_{21}^1\tilde{c}_{12}^0 \\
&& \d(c_{12}^1) &= -c_{12}^0c_{22}^1 + c_{11}^1c_{12}^0 \\
&& \d(\tilde{c}_{12}^1) &= -\tilde{c}_{12}^0\tilde{c}_{22}^1 +
\tilde{c}_{11}^1\tilde{c}_{12}^0
\end{align*}
and so forth for the differentials of $c_{ij}^p$ and
$\tilde{c}_{ij}^p$, $p \geq 2$.

We remark that $(\A,\d)$ has a graded augmentation over $\Z/2$ (and indeed over $\Z$ if we set $t=-1$), given by the
graded algebra map $\epsilon :\thinspace \A \to \Z/2$ determined by
$\epsilon(a_2)=\epsilon(a_3)=\epsilon(a_5)=\epsilon(a_7)=\epsilon(c_{12}^0)
= \epsilon(\tilde{c}_{12}^0) = \epsilon(c_{21}^1) =
\epsilon(\tilde{c}_{21}^1) = 1$ and $\epsilon=0$ for all other
generators of $\A$. It follows from the results of the subsequent
section (see Corollary~\ref{cor:augmentation}) that the Legendrian
knot pictured in
Figure~\ref{fig:torus-example2} is not destabilizable.

%*********************************************************************
\subsection{Stable tame isomorphism for countably generated DGAs}
\label{ssec:sti}

In this section, we discuss the notion of equivalence of DGAs that we
need in order to state the invariance result for the DGAs described in
Section~\ref{ssec:dga}. For finitely generated semifree DGAs, this equivalence
was first described by Chekanov \cite{bib:Chekanov}, who called it
``stable tame isomorphism''. We extend his notion here to countably
generated semifree DGAs.

\begin{definition}
Let $I$ be a countable index set, either $\{1,\ldots,n\}$ for some $n$
or $\mathbb{N} = \Z_{>0}$, and let $R$ be a commutative ring with unit. A
\textit{semifree algebra} over $R$ is an algebra
$\A$ over $R$, along with a distinguished set of generators $\{a_i
\,|\,
i\in I\} \subset \A$, such that $\A$ is the unital tensor algebra over
$R$ freely generated by the $\{a_i\}$:
\[
\A = R\langle a_1,a_2,\ldots \rangle.
\]
Thus $\A$ is freely generated as an $R$-module by finite-length words in the
$a_i$, including the empty word. A \textit{semifree differential
  graded algebra} $(\A,\d)$ over $R$ is a semifree algebra $\A$ over
$R$, equipped with a grading (additive over products, with $R$ in
grading $0$) and a degree
$-1$ differential $\d$ satisfying the signed Leibniz rule:
$\d(ab) = (\d a)b+(-1)^{|a|}a(\d b)$.
\end{definition}

Note that the differential $\d$ on a semifree DGA is determined by its
values on the generators $a_i$, $i\in I$. In practice, $R$ will be
either $\Z[\mathbf{t},\mathbf{t}^{-1}]$ or a quotient such as $\Z/2$.

We next define two classes of automorphisms of a semifree DGA, the
elementary and the tame automorphisms. These do not involve the differential.

\begin{definition}
An \textit{ordering} of a semifree algebra $\A$ over $R$ is a bijection $\sigma
\colon I\to I$, which we picture as giving an increasing total
order of the generators of $\A$ by setting
$a_{\sigma(1)} < a_{\sigma(2)} < a_{\sigma(3)} < \cdots$. Any ordering
produces a filtration on $\A$,
\[
R = \F^0\A \subset \F^1\A \subset \F^2\A \subset \cdots \subset \A,
\]
where $\F^k\A = R\langle
a_{\sigma(1)},a_{\sigma(2)},\ldots,a_{\sigma(k)}\rangle$.
\end{definition}

\begin{definition}
Let $\A$ be a semifree algebra over $R$. An \textit{elementary
  automorphism} of $\A$ is a grading-preserving algebra map $\phi
:\thinspace \A\to\A$
such that there exists an ordering $\sigma$ of $\A$ for which for all
$k\in I$,
\[
\phi(a_{\sigma(k)}) = u_k a_{\sigma(k)} + v_k,
\]
where $u_k$ is a unit in $R$ and $v_k \in \F^{k-1}\A$.
\label{def:elementary}
\end{definition}

\noindent
Informally, an elementary automorphism is a map that sends each
generator $a_k$ to itself plus terms that are strictly lower in the
ordering than $a_k$. Note that any elementary automorphism preserves
the corresponding filtration $\{F^k\A\}$ of $\A$.

\begin{remark}
The more familiar notion of an
elementary automorphism in the
sense of Chekanov \cite{bib:Chekanov}
(which is defined when the index set $I$ is finite, but the notion can
be extended to any index set)
is also an elementary
automorphism in the sense of Definition~\ref{def:elementary}. For
Chekanov (see also \cite{bib:ENS}), an algebra map $\phi :\thinspace
\A\to\A$ is elementary automorphism if
there exists an $i$ such that $\phi(a_i) = ua_i+v$ for $u$ a unit and
$v\in\A$ not involving $a_i$, and $\phi(a_j) = a_j$ for all $j\neq
i$. Given such a $\phi$, suppose the generators of $\A$ appearing in
$v$ are $a_{j_1},\ldots,a_{j_\ell}$ where $j_1,\ldots,j_\ell \neq
i$. Then any ordering $\sigma$ satisfying
$\sigma(1) = j_1$, \ldots, $\sigma(\ell) = j_\ell$, $\sigma(\ell+1)=i$
fulfills the condition of Definition~\ref{def:elementary}.

Conversely, an elementary automorphism $\phi$ as given in
Definition~\ref{def:elementary} is a composition of Chekanov's
elementary automorphisms: for $k\in I$, define $\phi_k$ by
$\phi_k(a_{\sigma(k)}) = u_k a_{\sigma(k)} + v_k$ and $\phi_k(a_j) =
a_j$ for all $j\neq \sigma(k)$; then
\[
\phi = \cdots \circ \phi_3 \circ \phi_2 \circ \phi_1.
\]
Note that this composition is infinite if $I$ is infinite, but
converges when applied to any element of $\A$.
\end{remark}

From now on, the term ``elementary automorphism'' will be in the sense
of Definition~\ref{def:elementary}.

\begin{proposition}
Any elementary automorphism of a semifree algebra is invertible, and
its inverse is also an
elementary automorphism.
\label{prop:elemaut}
\end{proposition}

\begin{proof}
Suppose that $\phi$ is an elementary automorphism of $\A$ with
ordering $\sigma$, units $u_k$, and algebra elements $v_k$ as in
Definition~\ref{def:elementary}. Note that $v_1 \in \F^0\A =
R$. Construct an algebra map $\psi \colon \A\to\A$ as follows.
We define $\psi(a_{\sigma(k)})$ inductively on $k$ by:
\begin{align*}
\psi(a_{\sigma(1)}) &= u_1^{-1}(a_{\sigma(1)}-v_1), \\
\psi(a_{\sigma(k)}) &= u_k^{-1}(a_{\sigma(k)}-\psi(v_k)).
\end{align*}
Note that this is constructed so that $\psi(v_k) \in \F^{k-1}\A$ and
$\psi(a_{\sigma(k)}) \in \F^k\A$ for all $k$, as is clear by
induction. (In particular, $\psi(v_k)$ is determined by
$\psi(a_{\sigma(1)}),\ldots,\psi(a_{\sigma(k-1)})$.) It is
straightforward to check by induction on $k$ that
$\psi(\phi(a_{\sigma(k)})) = \phi(\psi(a_{\sigma(k)})) =
a_{\sigma(k)}$, and so $\psi = \phi^{-1}$.
\end{proof}

\begin{definition}
A \textit{tame automorphism} of a semifree algebra is a composition of
finitely many elementary automorphisms.
\end{definition}

\noindent
Note crucially that different elementary automorphisms may have
different orderings associated to
them: it is not the case that a tame automorphism must preserve one
particular filtration of the algebra.

It follows from Proposition~\ref{prop:elemaut} that every tame
automorphism is invertible, with another tame automorphism as its
inverse. Thus the set of tame automorphisms forms a group, and the
following relation is an equivalence relation.

\begin{definition}
A \textit{tame isomorphism} between
two semifree differential graded algebras $(\A,\d)$ and $(\A',\d')$,
with generators $\{a_i \,|\, i\in I\}$ and $\{a_i' \,|\, i\in I\}$
respectively for a common index set $I$, is a graded algebra map $\phi
\colon \A\to\A'$ with
\[
\phi \circ \d = \d' \circ \phi,
\]
such that we can write $\phi = \phi_2 \circ \phi_1$, where
$\phi_1 \colon \A\to\A$ is a tame automorphism and $\phi_2$ is
the algebra map sending $a_i$ to $a_{\sigma(i)}'$ for all $i\in I$,
where $\sigma \colon I\to I$ is any bijection such that $|a_i| =
|a_{\sigma(i)}'|$ for all $i$. If there is a tame isomorphism between
$(\A,\d)$ and $(\A',\d')$, then the DGAs are \textit{tamely isomorphic}.
\end{definition}

The final ingredient in stable tame isomorphism is the notion of an
algebraic stabilization of a DGA. Our definition of stabilization
extends the corresponding definition in \cite{bib:Chekanov} by
allowing countably many generators to be added simultaneously.

\begin{definition}
Let $(\A,\d)$ be a semifree DGA over $R$ generated by $\{a_i \,|\, i\in
I\}$. A \textit{stabilization} of $(\A,\d)$ is a semifree DGA
$(S(\A),\d)$ constructed as follows.
Let $J$ be a countable (possibly finite) index set. Then $S(\A)$ is
the tensor algebra over $R$ generated by $\{a_i \,|\, i\in
I\} \cup \{e_j \,|\, j\in J\} \cup \{f_j \,|\, j\in J\}$, graded in
such a way that the grading on the $a_i$ is inherited from $\A$, and
$|e_j| = |f_j|+1$ for all $j\in J$. The differential on $S(\A)$
agrees on $\A\subset S(\A)$ with the original differential $\d$, and
is defined on the $e_j$ and $f_j$ by
\[
\d(e_j) = f_j, ~~~~~~~~~~
\d(f_j) = 0
\]
for all $j\in J$; extend to all of $S(\A)$ by the Leibniz rule as usual.
\end{definition}

Now we can define our notion of equivalence for countably generated DGAs.

\begin{definition}
Two semifree DGAs $(\A,\d)$ and $(\A',\d')$ are \textit{stable tame
  isomorphic} if some stabilization of $(\A,\d)$ is tamely isomorphic
to some stabilization of $(\A',\d')$.
\end{definition}

\noindent
Note that stable tame isomorphism is an equivalence relation.

With this in hand, we can state the main algebraic invariance result
for the DGA associated to a Legendrian link in $\#^k(S^1\times S^2)$.

\begin{theorem}
Let $\Lambda$ and $\Lambda'$ be Legendrian links in $\#^k(S^1\times
S^2)$ in normal form, and suppose that $\Lambda$ and $\Lambda'$ are
Legendrian isotopic. Let $(\A,\d)$ and $(\A',\d')$ be the semifree DGAs over
$\Z[\mathbf{t},\mathbf{t}^{-1}]$ associated to the diagrams $\pi_{xy}(\Lambda)$ and
$\pi_{xy}(\Lambda')$, which are in $xy$-normal form. Then $(\A,\d)$
and $(\A',\d')$ are stable tame isomorphic.
\label{thm:invariance}
\end{theorem}

Theorem~\ref{thm:invariance} will be proven in Section~\ref{sec:combpf}. In
practice, given a front projection for a Legendrian link in
$\#^k(S^1\times S^2)$ in Gompf standard form, one resolves it
following the procedure in Section~\ref{ssec:resolution} and then computes the DGA
associated to the resolved diagram; up to stable tame isomorphism,
this DGA is an invariant of the original Legendrian link.

We conclude this section with some general algebraic remarks about
stable tame isomorphism.
First, just as for finitely generated DGAs, stable tame isomorphism is a
special case of quasi-isomorphism.

\begin{proposition}
If $(\A,\d)$ and $(\A',\d')$ are stable tame isomorphic, then
$H_*(\A,\d) \cong H_*(\A',\d')$.
\end{proposition}

\begin{proof}
This is essentially the same as the corresponding proof in
\cite{bib:Chekanov}, see also \cite[Cor.~3.11]{bib:ENS}. If $(\A,\d)$
and $(\A',\d')$ are tamely isomorphic, then they are chain isomorphic
and the result follows. It suffices to check that if $(S(\A),\d)$ is a
stabilization of $(\A,\d)$, then the homologies are isomorphic. Let
$\iota :\thinspace \A\to S(\A)$ denote inclusion, and let $\pi
:\thinspace S(\A)\to\A$ denote the projection that sends any term
involving an $e_j$ or $f_j$ to $0$. Then $\pi\circ\iota = Id_{\A}$. If
we define $H :\thinspace S(\A)\to S(\A)$ by
\begin{align*}
H(v) &= 0 && v\in\A \\
H(ve_jw) &= 0 && v\in\A,~w\in S(\A) \\
H(vf_jw) &= (-1)^{|v|+1} ve_jw && v\in\A,~w\in S(\A),
\end{align*}
then it is straightforward to check that on $S(\A)$,
\[
H\circ\d + \d\circ H = \iota\circ\pi - Id_{S(\A)}.
\]
The result follows.
\end{proof}

Second, we can apply the usual machinery (augmentations,
linearizations, the
characteristic algebra \cite{bib:NgCLI}, etc.) to semifree DGAs up to
stable tame isomorphism. For now, we consider augmentations.

\begin{definition}
A \textit{graded augmentation} (over $\Z/2$) of a semifree DGA $(\A,\d)$ is
a graded algebra map $\epsilon :\thinspace \A \to \Z/2$, where $\Z/2$
lies in degree $0$, for which $\epsilon \circ \d = 0$.
\end{definition}

\begin{corollary}
The existence or nonexistence of a graded augmentation of $(\A,\d)$ is
invariant under Legendrian isotopy.
\label{cor:augmentation}
\end{corollary}

As in \cite{bib:Chekanov}, if $\Lambda$ is a (geometric) stabilization
of another Legendrian link $\Lambda'$, then the differential graded
algebra for $\Lambda$ is trivial up to stable tame isomorphism, and in
particular has no graded augmentations.

\begin{corollary}
If $(\A,\d)$ has a graded
augmentation, then $\Lambda$ is not destabilizable.
\end{corollary}

\begin{figure}
\centerline{
\includegraphics[width=\textwidth]{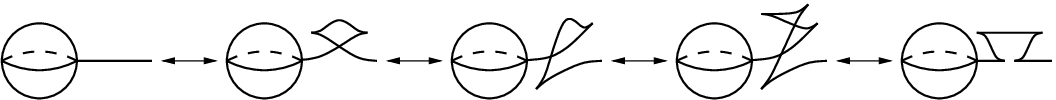}
}
\caption{A Legendrian link passing exactly once through a $1$-handle
  is Legendrian isotopic to its double stabilization. The second
  move is Gompf move 6 (see Figure~\ref{fig:gompfmoves}); the
  others are planar Legendrian isotopies.}
\label{fig:stab}
\end{figure}

\begin{rmk}\label{rmk:loose}
Suppose that the Legendrian link $\Lambda$ passes through any of the
$1$-handles exactly once ($n_j=1$
for some $j$). Then it is easy to check that $(\A,\d)$ is trivial,
because $\partial c_{11;j}^1 = 1$.
Indeed, it is the case that any
such $\Lambda$ is Legendrian isotopic to its own double stabilization
(i.e., the result of stabilizing $\Lambda$ once positively and once
negatively); see Figure~\ref{fig:stab}.

We can repeat this argument to
conclude that $\Lambda$ is Legendrian isotopic to arbitrarily high
double stabilizations of itself. It follows that the Legendrian
isotopy class of such a link is determined by its formal Legendrian
isotopy class (topological class and rotation number).

A Legendrian link in $\#^k(S^1\times S^2)$ passing through some handle
once could be considered an
imprecise analogue of a loose Legendrian knot in an overtwisted
contact $3$-manifold, i.e., a Legendrian knot whose complement is
overtwisted (see \cite{bib:EF}), in the sense that both of these are
infinitely destabilizable. It is also reminiscent of a loose
Legendrian embedding in higher dimensions \cite{bib:Murphy}.
\end{rmk}

%*********************************************************************
%*********************************************************************
\section{Calculations and applications}
\label{sec:calc}

In this section, we present a number of calculations of Legendrian contact
homology in connected sums of $S^1\times S^2$, as well as some applications, notably a new proof of the existence of exotic Stein structures on $\R^8$.

%*********************************************************************
\subsection{The cotangent bundle of $T^2$}
\label{ssec:T2ex}

Consider the Legendrian knot $\Lambda_1 \subset \#^2(S^1\times S^2)$ from
Section~\ref{ssec:ex}. As shown in \cite{bib:Gompf}, handle
attachment along this knot yields a Stein structure on the
$D^2$-bundle over $T^2$ with Euler number $0$, that is, $DT^*T^2$;
see also Proposition~\ref{prop:cotangent}. We
will explicitly calculate the Legendrian contact homology in this
case.

The DGA for $\Lambda_1$ was computed in
Section~\ref{ssec:ex}. Recall from there that it has a differential subalgebra, the internal DGA, generated by
internal Reeb
chords $c_{ij}^p,\tilde{c}_{ij}^p$.

\begin{proposition}
The differential graded algebra for $\Lambda_1$ is stable tame
isomorphic to the internal DGA along with one additional generator,
\label{prop:DGAT2}
$a$, of degree $1$, whose differential is
\[
\d(a) = c_{12}^0 \tilde{c}_{12}^0 + t\tilde{c}_{12}^0 c_{12}^0.
\]
\end{proposition}

\begin{proof}
Beginning with the differential graded algebra for $\Lambda_1$ from
Section~\ref{ssec:ex}, we successively apply the tame automorphisms
$a_2 \mapsto a_2 - \tilde{c}_{12}^0$, $a_3 \mapsto a_3 + c_{12}^0$,
$a_1 \mapsto
a_1+a_9a_3+\tilde{c}_{12}^0a_8+a_9c_{12}^0$; the
resulting differential has
\begin{align*}
\d(a_1) &= \tilde{c}_{12}^0 c_{12}^0 + t^{-1}c_{12}^0 a_5 \\
\d(a_4) &= 1+a_5a_7 \\
\d(a_6) &= 1+\tilde{c}_{12}^0 a_7 \\
\d(a_8) &= a_3 \\
\d(a_9) &= -a_2.
\end{align*}
Destabilize to eliminate the generators
$a_2,a_3,a_8,a_9$, and then successively apply the
tame automorphisms $a_4 \mapsto a_4 +
a_5(\tilde{c}_{22}^1a_7+\tilde{c}_{21}^1a_6)$,
$a_1 \mapsto t^{-1}a_1 + t^{-1}c_{12}^0 (-a_4\tilde{c}_{12}^0+a_5 \tilde{c}_{22}^1)$,
$a_4 \mapsto a_4+\tilde{c}_{11}^1$, $a_5 \mapsto a_5+\tilde{c}_{12}^0$, $a_6 \mapsto
a_6+\tilde{c}_{11}^1$, $a_7 \mapsto a_7-\tilde{c}_{21}^1$. This gives
\begin{align*}
\d(a_1) &= c_{12}^0 \tilde{c}_{12}^0 + t\tilde{c}_{12}^0 c_{12}^0\\
\d(a_4) &= -a_5 \tilde{c}_{21}^1 \\
\d(a_6) &= \tilde{c}_{12}^0 a_7.
\end{align*}
To complete the proof, we need to eliminate $a_4,a_5,a_6,a_7$; this is
done in the lemma that follows, with $(v,w,x,y) = (\tilde{c}_{11}^1,\tilde{c}_{22}^1,\tilde{c}_{12}^0,\tilde{c}_{21}^1)$ allowing us to eliminate $(a,b) = (a_4,a_5)$ and $(a,b) = (a_6,a_7)$ in turn.
\end{proof}

\begin{lemma}
Let $(\A,\d)$ be a differential graded algebra whose generators
include $v,w,x,y$ with $|v|=|w|=1$, $|x|=|y|=0$,
\[
\d(x)=\d(y)=0, ~~~~~
\d(v)=1-xy, ~~~~~\d(w)=1-yx.
\]
Let $(\A',\d)$ be the differential graded
algebra given by appending two generators $a,b$ to the generators of
$\A$, with $|a|=|b|+1$ and differential given by the differential on
$\A$, along with $\d(b)=0$ and $\d(a)$ equal to one of the following:
\[
\d(a) = \pm xb, \pm bx, \pm yb, \pm by.
\]
Then $(\A',\d)$ is stable tame isomorphic to $(\A,\d)$.
\label{lem:eliminate}
\end{lemma}

\begin{proof}
We will prove the lemma when $\d(a)=xb$; the other cases are clearly similar. Stabilize $(\A',\d)$ once by adding $e_{11},e_{12}$ with
$|e_{11}|=|e_{12}|+1=|a|+1$ and $\d(e_{11})=e_{12}, \d(e_{12})=0$. Applying the
successive elementary automorphisms $e_{12} \mapsto e_{12}-ya-wb$, $e_{11}
\mapsto e_{11} + we_{12}$, $a \mapsto a+xe_{12}$ yields
\[
\d(e_{11}) = -ya,~
\d(e_{12})=b,~
\d(a)=\d(b)=0.
\]
Destabilize once by removing $e_{12},b$, and stabilize once by adding
$e_{21},e_{22}$ with $|e_{21}|=|e_{22}|+1=|a|+2$ and
$\d(e_{21})=e_{22}, \d(e_{22})=0$. Applying the successive elementary
automorphisms
$e_{22} \mapsto e_{22}+xe_{11}-va$, $e_{21}\mapsto e_{21}+ve_{22}$,
$e_{11} \mapsto e_{11}-ye_{22}$ yields
\[
\d(e_{21}) = xe_{11},~
\d(e_{22}) = a,~
\d(a)=\d(e_{11}) = 0.
\]
Finally, destabilize once by removing $e_{22},a$ to obtain an algebra
generated by the generators of $\A$ along with $e_{21},e_{11}$ with
$|e_{21}|=|e_{11}|+1=|a|+2$ and $\d(e_{21})=xe_{11},\d(e_{11})=0$.

This procedure shows that $\A'$ is stable tame isomorphic to the same algebra but with the gradings of $a,b$ both increased by $2$, and if we omit $a,b$, then the resulting algebra is $\A$. We can then iterate the procedure, adding generators to $\A'$ in successively higher grading, to conclude the following. Let $(S(\A'),\d)$ be the stabilization of $(\A',\d)$ obtained by adding $e_{i1},e_{i2}$ for all $i\geq 1$, with $|e_{i1}|=|e_{i2}|+1=|a|+i$ and $\d(e_{i1}) = e_{i2}$, $\d(e_{i2}) = 0$. Then $(S(\A'),\d)$ is tamely isomorphic to $(S(\A'),\d')$, where $\d'$ is the same differential as $\d$ except
\begin{align*}
\d'(e_{12}) &= b, & \d'(b) &= 0, &&\\
\d'(e_{22}) &= a, & \d'(a) &= 0, &&\\
\d'(e_{i+2,2}) &= e_{i1}, & \d'(e_{i1}) &= 0, & i&\geq 1.
\end{align*}
But $(S(\A'),\d')$ is a stabilization of $(\A,\d)$, and the lemma is proven.
\end{proof}

From Proposition~\ref{prop:DGAT2}, we can calculate the Legendrian contact homology of $\Lambda_1$ in degree $0$, which is $H_0(\A,\d)$ where $(\A,\d)$ is the DGA for $\Lambda_1$.
In particular, we have the following result.

\begin{proposition}
If we set $t=-1$, then
\label{prop:LCHT2}
the Legendrian contact homology of $\Lambda_1$ in degree $0$ is
\[
\Z[x_1^{\pm 1},x_2^{\pm 1}] \cong \Z[\pi_1(T^2)].
\]
\end{proposition}

\begin{proof}
By Proposition~\ref{prop:DGAT2}, we want to calculate $H_0(\A,\d)$ where $(\A,\d)$ is the internal DGA along with $a$. Since $\A$ is supported in nonnegative degree, the subalgebra in degree $0$, which is generated by $c_{12}^0,c_{21}^1,\tilde{c}_{12}^0,\tilde{c}_{21}^1$, consists entirely of cycles. The boundaries in degree $0$ are generated by the differentials of generators of degree $1$:
$\d(a) = c_{12}^0 \tilde{c}_{12}^0 - \tilde{c}_{12}^0 c_{12}^0$, $\d(c_{11}^1) = 1-c_{12}^0c_{21}^1$, $\d(c_{22}^1) = 1-c_{21}^1c_{12}^0$, $\d(\tilde{c}_{11}^1) = 1-\tilde{c}_{12}^0\tilde{c}_{21}^1$, $\d(\tilde{c}_{22}^1) = 1-\tilde{c}_{21}^1\tilde{c}_{12}^0$. Thus in homology, $c_{21}^1$ and $\tilde{c}_{21}^1$ are the multiplicative inverses of $c_{12}^0$ and $\tilde{c}_{12}^0$ respectively. If we write $c_{12}^0 = x_1$ and $\tilde{c}_{12}^0 = x_2$, then in homology, $\d(a)$ causes $x_1$ and $x_2$ to commute, and thus $H_0(\A,\d) \cong \Z[x_1^{\pm 1},x_2^{\pm 1}]$, as desired.
%To be added later; the identification is $c_{12}^0 \mapsto x_2$,
%$c_{21}^1 \mapsto x_2^{-1}$, $\tilde{c}_{12}^0 \mapsto x_1$,
%$\tilde{c}_{21}^1 \mapsto x_1^{-1}$, and the differential $\d(a) =
%\tilde{c}_{12}^0 c_{12}^0 - c_{12}^0 \tilde{c}_{12}^0$ causes $x_1$
%and $x_2$ to commute in homology.
\end{proof}

\begin{remark}\label{rmk:wrapped}
It can be shown that the entire Legendrian contact homology of $\Lambda_1$, $H_*(\A,\d)$, is supported in degree $0$, but we omit the proof here.

From the point of view of the symplectic homology of $T^{\ast} T^{2}$, the Legendrian DGA is isomorphic to the (twisted) linearized contact homology of the co-core disk of the surgery, see \cite[Section 5.4]{bib:BEE}. This linearized contact homology is isomorphic to the wrapped Floer homology of the co-core disk (i.e.~the fiber in $T^{\ast} T^{2}$), see e.g.~\cite[Proof of Theorem 7.2]{EHK}. The wrapped homology of the fiber is in turn is isomorphic to the homology of the based loop space of $T^{2}$, see \cite{bib:APS, bib:A}, which is $\Z[\pi_1(T^{2})]$ in agreement with our calculation.

We also point out that the fact that we need to take $t=-1$ in Proposition \ref{prop:LCHT2} corresponds to changing  our choice of the Lie group spin structure on the knot in defining the signs to the bounding spin structure which extends over the core disk of the handle as is required in the construction of the surgery isomorphism.
\end{remark}

%*********************************************************************

\subsection{The cotangent bundle of $\Sigma_g$}

This is a generalization of the previous example. Let $\Lambda_1$ be
the Legendrian knot in $\#^2 (S^1\times S^2)$ given in the previous
section, and let $\Lambda_2$ be the knot in $\#^4 (S^1\times S^2)$ drawn
in Figure~\ref{fig:genus2b}. This construction generalizes
to $\Lambda_g \subset \#^{2g} (S^1\times S^2)$ for any $g \geq 1$.
It can readily be calculated that $\Lambda_g$ is nullhomologous and
has Thurston--Bennequin number $2g-1$ and rotation number $0$.

\begin{figure}
\centerline{
\includegraphics[width=\textwidth]{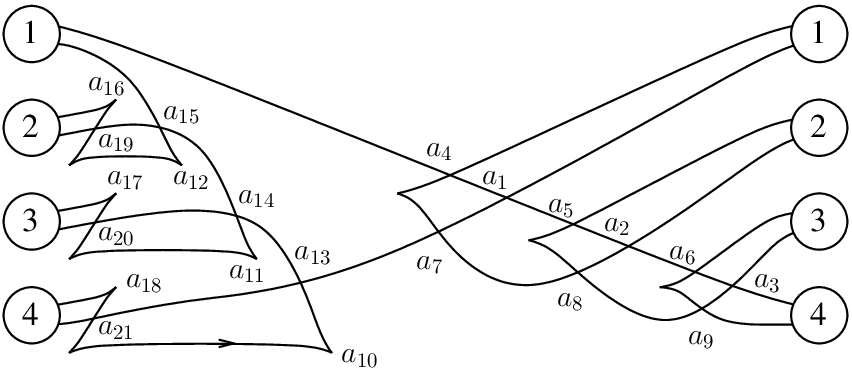}
}
\caption{The Legendrian knot $\Lambda_2 \subset \#^4 (S^1\times S^2)$.}
\label{fig:genus2b}
\end{figure}

The significance of $\Lambda_g$ is contained in the following result.

\begin{proposition}
Handle attachment along $\Lambda_g$ gives a Stein structure on
$DT^*\Sigma_g$, the disk cotangent bundle of the Riemann surface of
genus $g$.
\label{prop:cotangent}
\end{proposition}

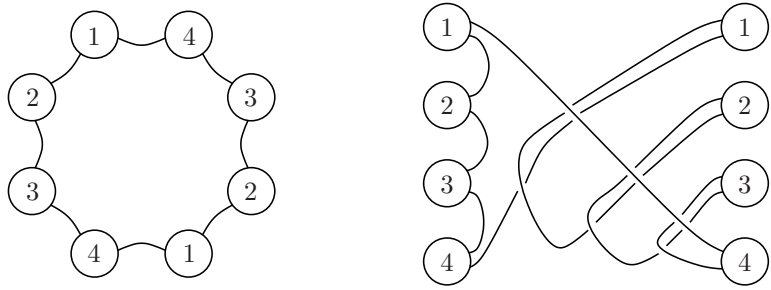
\begin{figure}
\labellist
\small\hair 2pt
\pinlabel {1} at 50 144
\pinlabel {2} at 15 109
\pinlabel {3} at 15 55
\pinlabel {4} at 50 19
\pinlabel {1} at 105 19
\pinlabel {2} at 140 55
\pinlabel {3} at 140 109
\pinlabel {4} at 105 144
\pinlabel {1} at 253 149
\pinlabel {2} at 253 104
\pinlabel {3} at 253 59
\pinlabel {4} at 253 14
\pinlabel {1} at 424 149
\pinlabel {2} at 424 104
\pinlabel {3} at 424 59
\pinlabel {4} at 424 14
\endlabellist
\centerline{
\includegraphics[width=4in]{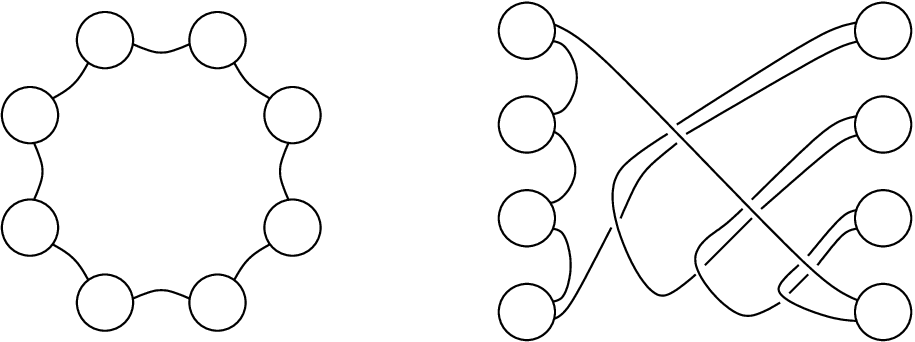}
}
\caption{Obtaining a disk bundle over $\Sigma_2$ by handle attachment.}
\label{fig:genus2a}
\end{figure}

\begin{proof}
We restrict ourselves to the case $g=2$ (the general case is similar), and
follow \cite{bib:Gompf}. A (slightly unorthodox)
handle decomposition of a disk bundle
over $\Sigma_2$ is given by the left diagram in
Figure~\ref{fig:genus2a}. Here the circles represent spheres,
identified in pairs via reflection, and the arcs represent a knot in
$\#^4(S^1\times S^2)$ along which a $2$-handle is attached. Via
isotopy, we may draw this handle decomposition in more standard form
as the right diagram in Figure~\ref{fig:genus2b}, where the spheres
are now identified pairwise through reflection in a vertical
plane. The Legendrian knot $\Lambda_2$ is simply a Legendrian form of
the knot in Figure~\ref{fig:genus2b}; note that it wraps around itself
near the left spheres labeled $2,3,4$, but this does not change the
isotopy class of the knot.

The particular disk bundle over $\Sigma_2$ determined by handle
attachment along $\Lambda_2$ has Euler number given by
$tb(\Lambda_2)-1 = 2g-2$. This agrees with the Euler number of
$DT^*\Sigma_2$, and the proposition follows.
\end{proof}

We now calculate the Legendrian contact homology of $\Lambda_g$. As in
the previous section, the
differential graded algebra for $\Lambda_g$ has an internal subalgebra
generated by $2g$ copies of $(\A_2,\d_2)$.

\begin{proposition}
The differential graded algebra for $\Lambda_g$ is stable tame
isomorphic to the internal subalgebra with one additional generator
$a$ of degree $1$ whose differential is
\[
\d(a) = c_{12;2g}^0c_{12;2g-1}^0\cdots c_{12;2}^0c_{12;1}^0 +t c_{12;1}^0c_{12;2}^0\cdots c_{12;2g-1}^0c_{12;2g}^0.
\]
\end{proposition}

\begin{proof}
We will assume $g=2$; the general case is similar.
Label crossings in the resolution of $\Lambda_2$ as follows:
the crossings corresponding to front crossings and right cusps are
labeled in Figure~\ref{fig:genus2b}; there are four additional
crossings $d_1,d_2,d_3,d_4$, corresponding to the half-twists at the
right of the diagram for handles $1$, $2$, $3$, $4$.
Pick a base point along the strand of $\Lambda_2$ connecting crossing $a_4$ leftwards to handle $1$. The
non-internal differential is given by
\begin{align*}
\d(a_1) &= -a_4a_7+t^{-1}c_{12;1}^0a_{15}a_{14}a_{13} &
\d(a_{16}) &= 1+c_{12;2}^0a_{19} \\
\d(a_2) &= -a_4-a_5a_8 & \d(a_{17}) &= 1+c_{12;3}^0a_{20} \\
\d(a_3) &= -a_5-a_6a_9 & \d(a_{18}) &= 1+c_{12;4}^0a_{21} \\
\d(a_{10}) &= 1+a_{13}a_{21} & \d(d_1) &= a_7-c_{12;1}^0 \\
\d(a_{11}) &= 1+a_{14}a_{20} & \d(d_2) &= a_8-c_{12;2}^0 \\
\d(a_{12}) &= 1+a_{15}a_{19} & \d(d_3) &= a_9-c_{12;3}^0 \\
&& \d(d_4) &= -a_6-c_{12;4}^0
\end{align*}
and the differential of all other non-internal generators is $0$.

%For the remainder of this proof, we will set $t=-1$ and work in $\Z/2$; it is an easy matter to reinsert signs and powers of $t$.
Applying the tame automorphism
\begin{align*}
a_1 &\mapsto a_1 + a_2a_7-a_3a_8a_7 + d_{4}a_9a_8a_7 + c_{12;4}^0d_{3}a_8a_7 +
c_{12;4}^0c_{12;3}^0d_{2}a_7 \\
& \qquad +
c_{12;4}^0c_{12;3}^0c_{12;2}^0d_1
\end{align*}
yields $\d(a_1) = c_{12;4}^0c_{12;3}^0c_{12;2}^0c_{12;1}^0 +
t^{-1}c_{12;1}^0a_{15}a_{14}a_{13}$.
Next define the auxiliary quantities $g_2,g_3,g_4,h_2,h_3,h_4$ by
\begin{align*}
g_2 &= c_{21;2}^1 a_{16} + c_{22;2}^1 a_{19} &
h_2 &= -a_{12}c_{12;2}^0+a_{15}g_2c_{12;2}^0+a_{15}c_{22;2}^1\\
g_3 &= c_{21;3}^1 a_{17} + c_{22;3}^1 a_{20} &
h_3 &= -a_{11}c_{12;3}^0+a_{14}g_3c_{12;3}^0+a_{14}c_{22;3}^1 \\
g_4 &= c_{21;4}^1 a_{18} + c_{22;4}^1 a_{21} &
h_4 &= -a_{10}c_{12;4}^0+a_{13}g_4c_{12;4}^0+a_{13}c_{22;4}^1,
\end{align*}
and note that $\d(g_2) = a_{19}+c_{21;2}^1$,
$\d(g_3) = a_{20}+c_{21;3}^1$, $\d(g_4) = a_{21}+c_{21;4}^1$, $\d(h_2) = a_{15}-c_{12;2}^0$, $\d(h_3) = a_{14}-c_{12;3}^0$, $\d(h_4) = a_{13}-c_{12;4}^0$. Thus applying the tame automorphism
\[
a_1 \mapsto t^{-1}(a_1 + c_{12;1}^0 h_2a_{14}a_{13} +
c_{12;1}^0c_{12;2}^0h_3a_{13}+
c_{12;1}^0c_{12;2}^0c_{12;3}^0h_4)
\]
now gives $\d(a_1) = tc_{12;4}^0c_{12;3}^0c_{12;2}^0c_{12;1}^0 +c_{12;1}^0c_{12;2}^0c_{12;3}^0c_{12;4}^0$.

Next, the succession of tame automorphisms $a_4 \mapsto -a_4-a_5a_8$, $a_5 \mapsto -a_5-a_6a_9$, $a_6\mapsto a_6-c_{12;4}^0$, $a_7\mapsto a_7+c_{12;1}^0$, $a_8\mapsto a_8+c_{12;2}^0$, $a_9 \mapsto a_9+c_{12;3}^0$ gives
\[
\d(a_2) = a_4, ~\d(a_3) = a_5, ~\d(d_{1}) = a_7, ~\d(d_{2}) = a_8, ~
\d(d_{3}) = a_9, ~ \d(d_{4}) = a_6,
\]
and we can destabilize to eliminate the generators $a_2,\ldots,a_9,d_{1},d_{2},d_{3},d_{4}$.
Finally, apply the successive tame automorphisms
$a_{10} \mapsto a_{10}+a_{13}g_4+c_{11;4}^1$,
$a_{11} \mapsto a_{11}+a_{14}g_3+c_{11;3}^1$,
$a_{12} \mapsto a_{12}+a_{15}g_2+c_{11;2}^1$,
$a_{16} \mapsto a_{16}+c_{11;2}^1$,
$a_{17} \mapsto a_{17}+c_{11;3}^1$,
$a_{18} \mapsto a_{18}+c_{11;4}^1$,
$a_{13} \mapsto a_{13}+c_{12;4}^0$,
$a_{14} \mapsto a_{14}+c_{12;3}^0$,
$a_{15} \mapsto a_{15}+c_{12;2}^0$,
$a_{19} \mapsto a_{19}-c_{21;2}^1$,
$a_{20} \mapsto a_{20}-c_{21;3}^1$,
$a_{21} \mapsto a_{21}-c_{21;4}^1$ to get
\begin{align*}
\d(a_{10}) &= a_{13}c_{21;4}^1 & \d(a_{16}) &= c_{12;2}^0 a_{19} \\
\d(a_{11}) &= a_{14}c_{21;3}^1 & \d(a_{17}) &= c_{12;3}^0 a_{20} \\
\d(a_{12}) &= a_{15}c_{21;2}^1 & \d(a_{18}) &= c_{12;4}^0 a_{21},
\end{align*}
and we can eliminate $a_{10},\ldots,a_{21}$ by Lemma~\ref{lem:eliminate}. What remains is the internal subalgebra and $a_1$.
\end{proof}

\begin{proposition}
If we set $t=-1$, then the Legendrian contact homology of $\Lambda_g$ in degree $0$ is
\[
\Z \langle x_1^{\pm 1},x_2^{\pm 1},\ldots,x_{2g}^{\pm 1}\rangle \, / \, (x_1x_2\cdots x_{2g}-x_{2g}\cdots x_2x_1) \cong  \Z[\pi_1(\Sigma_g)].
\]
\end{proposition}

\begin{proof}
Nearly identical to the proof of Proposition~\ref{prop:LCHT2}. In this case, the identification is given by $c_{12;\ell}^0 \mapsto x_\ell$,
$c_{21;\ell}^1 \mapsto x_\ell^{-1}$ for $1\leq \ell \leq 2g$, and the differential $\d(a)$ gives the desired relation $x_1x_2\cdots x_{2g}=x_{2g}\cdots x_2x_1$ in homology.
\end{proof}

\begin{remark}
As was the case for $\Lambda_1$,
it can be shown that the entire Legendrian contact homology of $\Lambda_g$ is supported in degree $0$. Also, as there our calculation can be interpreted as a calculation of the wrapped Floer homology of a fiber in $T^{\ast}\Sigma_{g}$ which is isomorphic to the homology of the based loop space, i.e.~$\Z[\pi_1(\Sigma_g)]$.
\end{remark}

%*********************************************************************

\subsection{Exotic Stein structure on $\R^8$}

Let $\Lambda$ be the Legendrian knot in $S^1\times S^2$ depicted in Figure~\ref{fig:exotic-example}. Note that $\Lambda$ generates $H_1(S^1\times S^2)$, i.e., it winds algebraically once around $S^1$. This means in particular that the Weinstein $4$-manifold $W$ that results from adding a $2$-handle along $\Lambda$ is contractible (although its fundamental group at infinity is nontrivial).

We will prove the following result.

\begin{figure}
\centerline{
\includegraphics[width=4in]{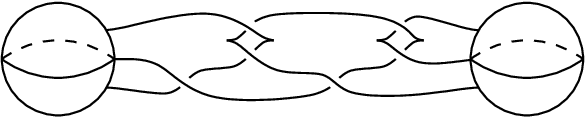}
}
\caption{The Legendrian knot $\Lambda \subset S^1\times S^2$.}
\label{fig:exotic-example}
\end{figure}

\begin{proposition}
The Legendrian contact homology of $\Lambda$, i.e., the homology of $(\A(\Lambda),\d(\Lambda))$, is nonzero.
\label{prop:nonzeroLCH}
\end{proposition}

Before proving Proposition~\ref{prop:nonzeroLCH}, we deduce a result about exotic Stein structures on $\R^8$. Let $W$ be the Weinstein $4$-manifold given by attaching a $2$-handle to $S^1\times D^3$ along $\Lambda$. Then the product $W\times W$ inherits a Stein structure from $W$.

\begin{proposition}
$W\times W$ is diffeomorphic to $\R^8$
\label{prop:exoticStein}
but the Stein structure on $W\times W$ is distinct from the standard Stein structure on $\R^8$.
\end{proposition}

\begin{proof}
First note that $W$ is contractible. Consider the boundary of $W\times W$ as the result of joining the bundles $W\times\pa W$ and $\pa W\times W$ along their common boundary $\partial W\times\partial W$. Note that any loop in the boundary of $W\times W$ is homotopic to a loop in $\pa W\times\pa W$, which is null-homotopic since $W$ is contractible. Thus we find that $\pa(W\times W)$ is a homotopy $7$-sphere that bounds the contractible manifold $W\times W$. It follows that $W\times W$ is in fact diffeomorphic to $\R^{8}$.

It remains to show that the Stein structure on $W\times W$ is not the standard one on $\R^8$. By Proposition~\ref{prop:nonzeroLCH}, for the Legendrian homology DGA $(\A(\Lambda),\d(\Lambda))$, the unit $1 \in \A(\Lambda)$ is not in the image of the differential $\d(\Lambda)$. It follows
it follows that the homology of the Hochschild complex $A^{\rm Ho}(\Lambda)$ is nonzero, and thus from \cite{bib:BEE} that the symplectic homology $SH(W)$ is nonzero. Now the K\"unneth formula in symplectic homology \cite{bib:Oancea_Kunneth} gives
\[
SH(W\times W)\cong SH(W)\otimes SH(W)\ne 0 =SH(\R^{8}),
\]
and we conclude that $W\times W$ is an exotic Stein $\R^{8}$.
\end{proof}

\begin{remark}
The same result holds for any Legendrian knot $\Lambda \subset S^1\times S^2$ that generates $H_1(S^1\times S^2)$ and has nonzero Legendrian contact homology. A ``simpler'' example of such a $\Lambda$ is given in Remark~\ref{rmk:simpleStein}; see Figure~\ref{fig:exotic-example1}.
\end{remark}

\begin{figure}
\centerline{
\includegraphics[width=\textwidth]{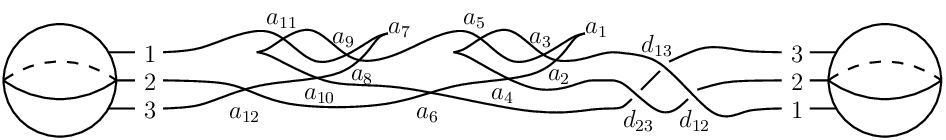}
}
\caption{A knot Legendrian isotopic to $\Lambda$, with half-twist
  added and $xy$-projection crossings labeled.}
\label{fig:exotic-label2}
\end{figure}

\begin{proof}[Proof of Proposition~\ref{prop:nonzeroLCH}]
For the purposes of this result, it suffices to work over the ring
$\Z/2$ instead of $\Z[t,t^{-1}]$ by setting $t=1$ and reducing mod $2$.
We will show that the homology for the DGA over this ring has a certain quotient that has (somewhat remarkably) previously appeared in the literature on Legendrian contact homology in a rather different context \cite{bib:SVV}.

For computational convenience, change $\Lambda$ by a Legendrian isotopy by pulling out two of its cusps, to give the knot in Figure~\ref{fig:exotic-label2}; this causes all holomorphic disks (besides those in the internal differential) to be embedded rather than just immersed.
The non-internal differential for the knot shown in
Figure~\ref{fig:exotic-label2} is
\begin{align*}
\d(a_1) &= 1+a_5a_4+(1+a_9a_{11})a_6+a_9c_{12}^0 \\
\d(a_2) &= (1+a_3a_5)a_4+a_3(1+a_9a_{11})a_6+a_3a_9c_{12}^0 \\
\d(a_7) &= 1+a_{11}a_{10}+c_{12}^0a_{12}+c_{13}^0 \\
\d(a_8) &= (1+a_9a_{11})a_{10}+a_9c_{12}^0a_{12}+a_9c_{13}^0 \\
\d(a_{12}) &= c_{23}^0 \\
\d(d_{12}) &= c_{12}^0+1+a_3a_5 \\
\d(d_{13}) &= c_{13}^0+d_{12}c_{23}^0+(1+a_3a_5)d_{23}+a_3(1+a_9a_{11}),
\end{align*}
with zero differential for all other non-internal generators.

Let $\A' = (\Z/2) \langle a,b,c,d \rangle$ be the tensor algebra on four generators $a,b,c,d$. Define an algebra map $\phi:\thinspace\A\to\A'$ by
\begin{gather*}
\phi(c_{12}^0) = a, ~ \phi(c_{13}^0) = d, ~
\phi(c_{21}^1) = c,  ~ \phi(c_{31}^1) = b,\\
\phi(a_3)=1, ~ \phi(a_4) = 1, ~ \phi(a_5) = a+1, ~ \phi(a_9) = 1, ~
\phi(a_{10}) = 1, ~ \phi(a_{11}) = d+1,
\end{gather*}
and $\phi =0$ for all other generators of $\A$. It is straightforward to check that $\phi\circ\d = 0$ on all non-internal generators of $\A$. On internal generators of $\A$, we have
\begin{align*}
(\phi\circ\d)(c_{32}^1) &= ba, \\
(\phi\circ\d)(c_{11}^1) &= 1+ac+db, \\
(\phi\circ\d)(c_{22}^1) &= 1+ca, \\
(\phi\circ\d)(c_{33}^1) &= 1+bd, \\
(\phi\circ\d)(c_{23}^1) &= cd,
\end{align*}
and $(\phi\circ\d)
(c_{ij}^p) = 0$ for all other internal generators.

It follows that $\phi$ induces a surjective algebra map from $\A$ to the quotient
\[
\A'' = (\Z/2) \langle a,b,c,d \rangle / (1+ac+db,1+ca,1+bd,ba,cd),
\]
and this is a chain map from $(\A,\d)$ to $(\A'',0)$. Thus $\phi$ descends to a surjection from $H(\A,\d)$ to $\A''$.
But it was proven in \cite{bib:SVV} that $\A''$ is nonzero.
\end{proof}

\begin{figure}
\centerline{
\includegraphics[width=.6\linewidth]{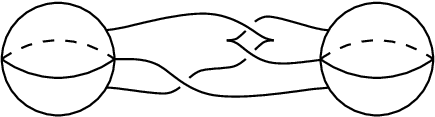}
}
\caption{Another Legendrian knot in $S^1\times S^2$.}
\label{fig:exotic-example1}
\end{figure}

\begin{remark}
It can also be shown that the Legendrian knot shown in Figure~\ref{fig:exotic-example1}, which is slightly simpler than $\Lambda$ but also winds homologically once around $S^1$, also has nonzero Legendrian contact homology.
\label{rmk:simpleStein}
Thus this knot also produces an exotic Stein structure on $\R^8$.
However, the proof that the homology is nonzero in this case appears
to be more complicated than the proof of
Proposition~\ref{prop:nonzeroLCH}; the proof known to us uses
Gr\"obner bases calculated via computational algebra software.
\end{remark}

%*********************************************************************
%*********************************************************************
\section[Geometric Constructions]{Geometric constructions for relating the combinatorial invariant to Legendrian homology}
\label{sec:geom}

%*********************************************************************
\subsection{Main result and overview of Sections~\ref{sec:geom}
  and~\ref{Sec:openclosed}}

As mentioned in Section \ref{Sec:intro}, the Legendrian contact homology (Legendrian homology, for short) of a
Legendrian link $\Lambda$ is a part of SFT and is in particular
defined using moduli spaces of holomorphic disks. We will discuss this
theory in the setting relevant to this paper in Section
\ref{Sec:openclosed} below. Here we just describe its basic structure
for Legendrian links $\Lambda\subset Y_k=\#^{k}(S^{1}\times S^{2})$.
Legendrian homology associates a DGA $\A_{\mathrm{H}}(\Lambda)$ with
differential $\d_{\mathrm{H}}$ to $\Lambda$. The algebra $\A_{\mathrm{H}}(\Lambda)$ is freely generated by the Reeb chords of $\Lambda$ over $\Z[\mathbf{t},\mathbf{t}^{-1}]$ and the
differential $\d_{\mathrm{H}}$ is defined through a count of holomorphic disks in the
symplectization $\R\times Y_k$ with Lagrangian boundary condition
$\R\times\Lambda$. The main result of the paper can then be stated as
follows:

\begin{theorem}\label{thm:main}
The Legendrian homology DGA $\A_{\mathrm{H}}(\Lambda)$ of a link $\Lambda\subset Y_k$ in Gompf normal form is canonically isomorphic to the combinatorially defined algebra $\A$ associated to $\Lambda$ in Section \ref{ssec:dga}. The canonical isomorphism is a one-to-one map that takes Reeb chords to generators of $\A$ and that intertwines the differentials $\d_{\mathrm{H}}$ on $\A_{\mathrm{H}}$ and $\d$ on $\A$.
\end{theorem}

The proof of Theorem \ref{thm:main} is rather involved and occupies
Sections \ref{sec:geom} and~\ref{Sec:openclosed}. Before we outline its steps we comment on the definition of $\A_{\mathrm{H}}(\Lambda)$. The algebra $\A_{\mathrm{H}}(\Lambda)$ depends on the choice of contact form on $Y_k$, and below we will equip $Y_k$ with contact forms that depend on positive parameters $(\epsilon,\delta)$. The set of Reeb chords of $\Lambda$ will in the present setup (unlike in the case of ambient contact manifold $\R^{3}$) not be finite. In general Legendrian homology algebras come equipped with natural action filtrations and are defined as corresponding direct limits. For the contact form on $Y_k$ that we use below the action is related to the grading of the algebra, which simplifies the situation somewhat. Theorem \ref{thm:main} should be interpreted as follows. For any given degree there exists $(\epsilon_0,\delta_{0})$ such that for $Y_k$ equipped with the contact form corresponding to $(\epsilon,\delta)$ with $\epsilon<\epsilon_0$ and $\delta<\delta_0$, the natural map in the formulation has the properties stated on the part of the algebra of action less than the given degree, and furthermore, the homology of $\A$ is canonically isomorphic to the homology of $\A_{\mathrm{H}}$ defined via the action filtration.

The proof of Theorem \ref{thm:main} involves a study of holomorphic
disks in specific contact models of $Y_k$ and associated contact and
symplectic manifolds. We study aspects of the geometric constructions
and properties of these models in this section, followed by
holomorphic disks in Section~\ref{Sec:openclosed}.

%*********************************************************************
\subsection{Contact and symplectic pairs}\label{Sec:symplmfds}
In this section we consider constructions of pairs $(Y,\Lambda)$ where $Y$ is a contact $3$-manifold and $\Lambda$ is a Legendrian link, as well as pairs $(X,L)$, where $X$ is a Weinstein $4$-manifold and $L$ an exact Lagrangian submanifold in $X$ with cylindrical ends, i.e., outside a compact subset $(X,L)$ looks like a disjoint union of half symplectizations of contact pairs $(Y,\Lambda)$.

\subsubsection{One-handles}\label{sec:1handle}
Consider $\C^{2}$ with coordinates $(z_1,z_2)=(x_1+iy_1,x_2+iy_2)$. For $\delta>0$, consider the region
\[
H_{\delta}=\left\{(z_1,z_2)\colon
-\delta^{2}\le\tfrac12\left(x_{1}^{2}+y_{1}^{2}\right)+2x^{2}_{2}-y_{2}^{2}\le\delta^{2}
\right\}
\]
with boundary given by the hypersurfaces
\[
V_{\pm\delta}=\left\{(z_1,z_2)\colon
\tfrac12\left(x_{1}^{2}+y_{1}^{2}\right)+2x^{2}_{2}-y_{2}^{2}=\pm\delta^{2}
\right\}.
\]
Topologically, $V_{-\delta}\approx_{\rm top} \R^{3}\times S^{0}$ and $V_{\delta}\approx_{\rm top} S^{2}\times \R$. The normal vector field
\begin{equation}\label{eq:Liouvillevf}
Z=
\tfrac12
\left(x_1\,\pa_{x_1}+y_1\,\pa_{y_1}\right)
+2x_2\,\pa_{x_2}-y_2\,\pa_{y_2}
\end{equation}
of $V_{\pm\delta}$ is a Liouville vector field of the standard symplectic form $\omega_{\rm st}$ on $\C^{2}$, where
\[
\omega_{\rm st}=dx_1\wedge dy_1+dx_2\wedge dy_2.
\]
That is, if $L$ denotes the Lie-derivative then
\[
L_{\!Z}\,\omega_{\rm st}= d(\omega(Z,\cdot))=
d\left(\tfrac12(x_1\,dy_1-y_1\,dx_1)+2x_2\,dy_2+y_2\,dx_2\right)=\omega_{\rm st}.
\]
Note that $Z$ points out of $H_{\delta}$ along $V_{\delta}$ and into $H_{\delta}$ along $V_{-\delta}$.

The Liouville vector field $Z$ determines the contact forms $\alpha_{\pm\delta}$ along $V_{\pm\delta}$ where
\[
\alpha_{\pm\delta}=\omega_{\rm st}(Z,\cdot)|_{V_{\pm\delta}}=\left.\left(\tfrac12(x_1\,dy_1-y_1\,dx_1)+2x_2\,dy_2+y_2\,dx_2\right)\right|_{V_{\pm\delta}}.
\]
Trivializations of the contact structures $\krn(\alpha_{\pm\delta})$ are given by $(v,iv)$, where
\begin{equation}\label{eq:contacttrivhandle}
v= 2x_2\,\pa_{x_1}+y_2\,\pa_{y_1}-\tfrac12(x_1\,\pa_{x_2}-y_1\,\pa_{y_2})
\end{equation}
and $i$ denotes the complex structure on $\C^{2}$. The Reeb vector field $R$ of $\alpha_{\pm\delta}$ satisfies $R=N\tilde R$, where
\begin{equation}\label{eq:unnorReebhandle}
\tilde R= \left.\left(\tfrac12(x_1\,\pa_{y_1}-y_1\,\pa_{x_1})+2x_2\,\pa_{y_1}+y_2\,\pa_{x_2}\right)\right|_{V_{\pm\delta}},
\end{equation}
and where the normalization factor $N$ is given by
\[
N = \left(\tfrac14\left(x_1^{2}+y_1^{2}\right)+4x_2^{2}+y_2^{2}\right)^{-1}.
\]
The flow lines of the Reeb flow are thus the solution curves of the following system of ordinary differential equations
\[
\begin{cases}
\dot x_1 = -\tfrac12 y_1,\\
\dot y_1 = \tfrac12 x_1,\\
\dot x_2 = y_2,\\
\dot y_2 = 2x_2,
\end{cases}
\]
given by
\begin{align}\label{eq:solReebx1}
x_1(t) &= x_1(0)\,\cos\left(\tfrac12 t\right) + y_1(0)\,\sin\left(\tfrac12 t\right),\\\label{eq:solReeby1}
y_1(t) &= -x_1(0)\,\sin\left(\tfrac12 t\right)+y_1(0)\,\cos\left(\tfrac12 t\right),\\\label{eq:solReebx2}
x_2(t) &= x_2(0)\,\cosh\left(\sqrt{2}\,t\right)+\tfrac{1}{\sqrt{2}}y_2(0)\,\sinh\left(\sqrt{2}\, t\right),\\\label{eq:solReeby2}
y_2(t) &= \sqrt{2}x_2(0)\,\sinh\left(\sqrt{2}\,t\right)+y_2(0)\,\cosh\left(\sqrt{2}\,t\right),
\end{align}
where $(x_1(0)+iy_1(0),x_2(0)+iy_2(0))$ is the initial position.

\begin{lma}\label{lma:orbitsinhandle}
There is exactly one geometric closed Reeb orbit $\gamma$ in $V_{\delta}$. If $\gamma^{m}$ denotes the $m^{\rm th}$ iterate of $\gamma$ then
\[
\CZ(\gamma^{m})=2m,
\]
where $\CZ$ denotes the Conley--Zehnder index measured with respect to the trivialization $(v,iv)$, see \eqref{eq:contacttrivhandle}.
\end{lma}

\begin{pf}
The statement on the uniqueness of the geometric orbit is immediate from the explicit expression for the Reeb flow lines above.
In the trivialization on the contact planes along $\gamma$ given by $(\pa_{x_2},\pa_{y_2})$ the linearized Reeb flow is given by Equations \eqref{eq:solReebx2} and \eqref{eq:solReeby2}. A straightforward calculation shows that the Conley--Zehnder index of the corresponding path of matrices equals $0$. The frame $(\pa_{x_2},\pa_{y_2})$ makes one full turn in the positive direction with respect to the frame $(v,iv)$ per iterate of $\gamma$. It follows that $\CZ(\gamma^{m})=2m$.
\end{pf}

The {\em standard Legendrian strand} in $V_{\delta}$ is the subset
\[
\Lambda_{\rm st}=
\left\{(x_1,y_1,x_2,y_2)\in V_{\delta}\colon y_1=x_2=0 \right\}.
\]

\begin{lma}\label{lem:handleReebtv}
The Reeb chords of $\Lambda_{\rm st}\subset V_{\delta}$ are exactly the Reeb orbits $\gamma^{k}$ and the image under the linearized Reeb flow of the tangent space to $\Lambda_{\rm st}$ at the initial point of the chord is transverse to the tangent space at its end point.
\end{lma}

\begin{pf}
Immediate from the expression for the Reeb flow lines above.
\end{pf}

\subsubsection{Standard contact $\R^{3}$ and standard contact balls}\label{ssec:stcntctball}
The {\em standard contact form} on $\R^{3}$ is the $1$-form
\[
\alpha_{\rm st}=dz-y\,dx,
\]
where $(x,y,z)$ are coordinates on $\R^{3}=J^{1}(\R)$ with $x$ a coordinate on $\R$, $z$ the coordinate of function values in $J^{0}(\R)$, and $y$ the fiber coordinate in $T^{\ast}\R$.

We define the {\em standard contact ball} of radius $\rho>0$ as the subset
\begin{align*}
B_{\rho}&=\left\{(u,v,z)\in\R^{3}\colon u^2+v^2+z^2\le \rho^2\right\}\\
&=\left\{(r,\theta,z)\in\R_{+}\times S^{1}\times\R\colon r^2+z^2\le \rho^2\right\},
\end{align*}
with the contact form
\[
\alpha_{\rm b}=dz+\tfrac12(u\,dv-v\,du)=
dz-\tfrac12 r^{2}\, d\theta,
\]
where $(r,\theta)$ are polar coordinates on the $uv$-plane.

Let $p=(x_0,y_0,z_0)\in\R^{3}$ and consider the embedding $F\colon B_{\rho}\to\R^{3}$
\begin{equation}\label{eq:balltospace}
F(u,v,z)= (u+x_0\,,\, v+y_0\,,\, z+z_0+y_0u+\tfrac12 uv).
\end{equation}
Note that $F$ is a contact embedding, i.e., $F^{\ast}\alpha_{\rm st}=\alpha_{\rm b}$. We call $F(B_{\rho})$ the {\em standard contact ball of radius $\rho$ centered at $p$}.

\subsubsection{The manifold $Y_{k}(\delta)$, standard $\R^{3}$ with $k$ 1-handles}\label{ssec:Y_k(delta)}
The contact manifold $Y_k(\delta)$ is topologically $\R^{3}$ with $2k$ balls removed and $k$ 1-handles (i.e.~$[-1,1]\times S^{2}$) attached along the boundaries of the removed balls.

We first describe the attaching loci for the handles. Fix $2k$ points
\[
\left\{\left(0, y^{\ell}, z^{\ell}\right)\,,\,\left(A, \tilde y^{\ell}, \tilde z^{\ell}\right)\right\}_{1\le \ell\le k}
\]
in $\R^{3}$ as in Section \ref{ssec:stdform}. Let $\sigma$ denote the minimal distance between two of these points and fix standard contact balls, see Section \ref{ssec:stcntctball}, of radii $\rho\ll \sigma$ centered at these points. These balls will be the attaching locus. Let $\R^{3}_{\rho}(\circ_k)$ denote the complement of the standard contact balls of radius $\rho/2$ centered at these points.

We then consider identifications of regions near the boundary of the standard contact balls centered at  two corresponding points
$$
p^{-}=\left(0, y^{\ell}, z^{\ell}\right)\quad\text{and}\quad p^{+}=\left(A, \tilde{y}^{\ell}, \tilde{z}^{\ell}\right)
$$
with two regions in the standard handle, see Section \ref{sec:1handle}, that will be used for the handle attachment. Fix $\delta\ll \rho$ and consider the two regions
\begin{align}\label{eq:attachhandle1}
\begin{split}
A_\rho^{+}(\delta) &=\left\{(x_1,y_1,x_2,y_2)\in V_{-\delta}\colon x_2=0, y_{2}>0, x_1^{2}+y_1^{2}\le \rho^{2}\right\},\\
A_\rho^{-}(\delta) &=\left\{(x_1,y_1,x_2,y_2)\in V_{-\delta}\colon x_2=0, y_{2}<0, x_1^{2}+y_1^{2}\le \rho^{2}\right\},
\end{split}
\end{align}
see Figure \ref{fig:Arho}, and the map $G\colon A_\rho^{\pm}(\delta)\to B_\rho$, where $B_{\rho}$ is the standard contact ball:
\begin{equation}\label{eq:attachhandle2}
G(x_1,y_1,0,y_2)=(x_1,y_1,0).
\end{equation}
\begin{figure}
\centering
\includegraphics[width=.4\linewidth]{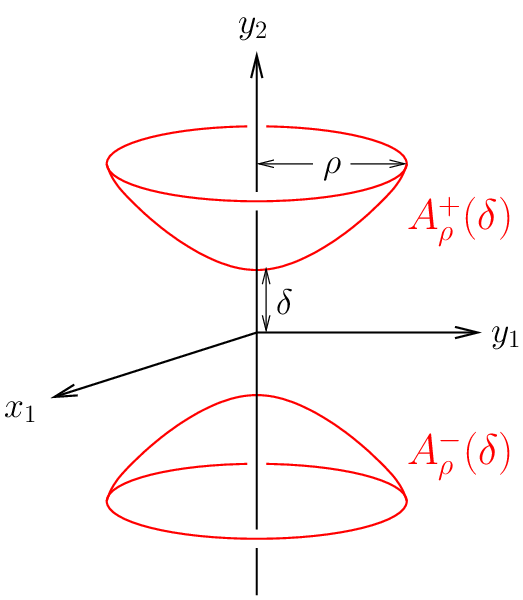}
\caption{The regions $A_{\rho}^{\pm}(\delta)\subset V_{-\delta}$.}
\label{fig:Arho}
\end{figure}
Then
\[
G^{\ast}\left(dz+\tfrac12(u\,dv-v\,du)\right)=\tfrac12(x_1\,dy_1-y_1\,dx_1).
\]
Thus $G^{\ast}\alpha_{\rm b}=\alpha_{-\delta}|_{A_\rho^{\pm}(\delta)}$. The Reeb vector fields of $\alpha_{-\delta}$ and $\alpha_{\rm b}$ are transverse to $A_{\rho}^{\pm}(\delta)$ respectively $G(A_{\rho}^{\pm}(\delta))$ and we use their flows to construct a contactomorphism from a neighborhood of $A_\rho^{\pm}(\delta)$ to $B_\rho$. We use the notation $B_\rho^{\pm}(\delta)\subset V_{-\delta}$ for this neighborhood of $A_{\rho}^{\pm}(\delta)$, see Figure \ref{fig:Brho}. Thus $B^{\pm}_{\rho}(\delta)$ is identified with a neighborhood of $p^{\pm}$.

\begin{figure}
\centering
\includegraphics[width=.7\linewidth]{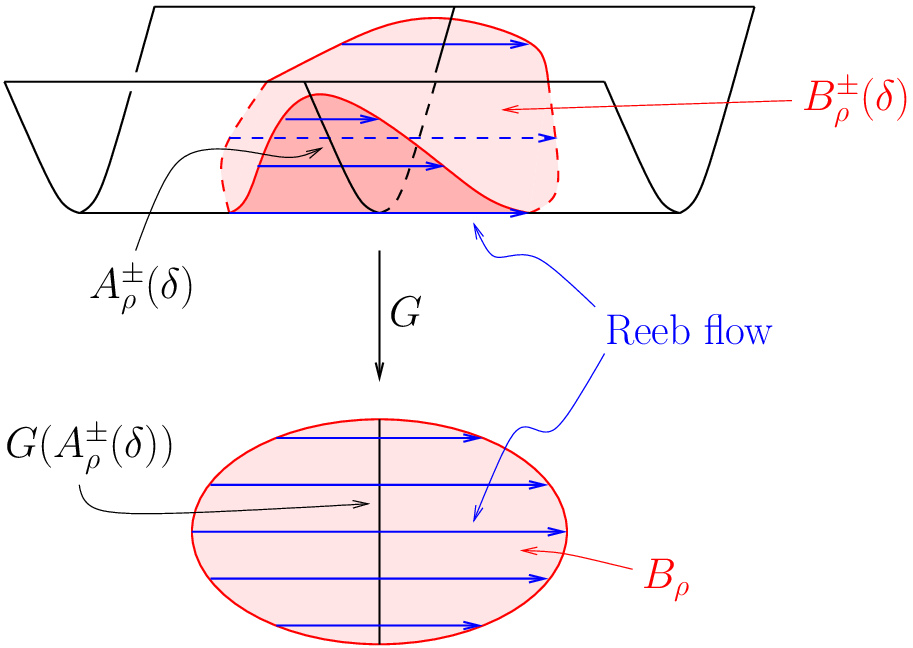}
\caption{The regions $B_{\rho}^{\pm}(\delta)\subset V_{-\delta}$.}
\label{fig:Brho}
\end{figure}

We use the flow of the Liouville vector field $Z$, see \eqref{eq:Liouvillevf}, to identify $B_{\rho}^{\pm}(\delta)-B_{\rho/2}^{\pm}(\delta)\subset V_{-\delta}$ with a region in $V_{\delta}$. More precisely, let $\Phi^{t}_{Z}\colon\R^{4}\to\R^{4}$ denote the time $t$ flow generated by the vector field $Z$. Then for each $p\in B_{\rho}^{\pm}(\delta)-B_{\rho/2}^{\pm}(\delta)$ there is a unique time $T(p)$ such that $\Phi_{Z}^{T(p)}(p)\in V_{\delta}$ and we define the map
\[
\phi\colon \bigl(B_{\rho}^+(\delta)-B_{\rho/2}^+(\delta)\bigr)
\cup \bigl(B_{\rho}^-(\delta)-B_{\rho/2}^-(\delta)\bigr)
\to V_{\delta},\quad \phi(p)=\Phi_{Z}^{T(p)}(p).
\]
Then $\phi^{\ast}\alpha_{\delta}=e^{T}\alpha_{-\delta}$. See Figure \ref{fig:liouv}.
\begin{figure}
\centering
\includegraphics[width=.8\linewidth]{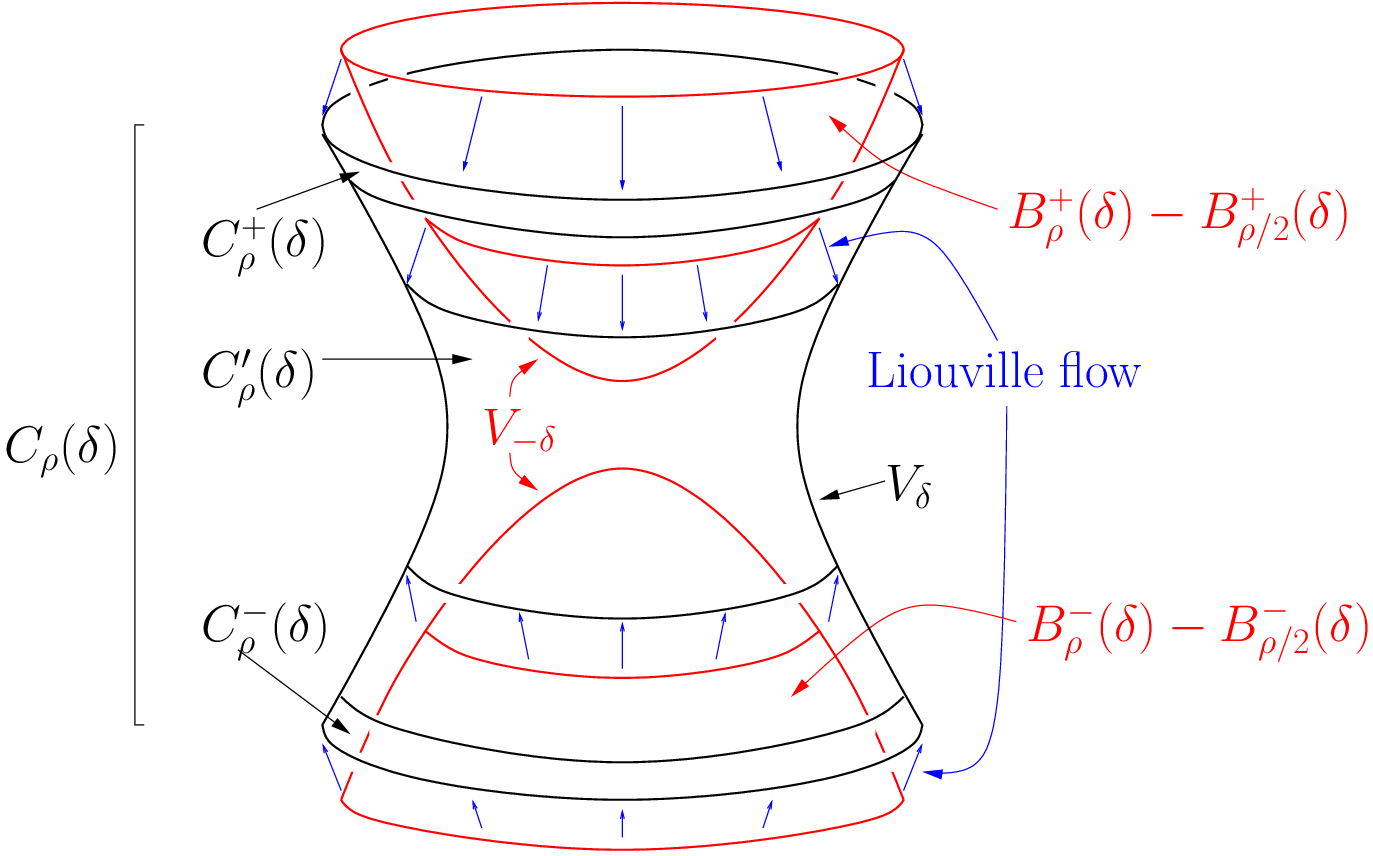}
\caption{The map $\phi$ identifies regions in $V_{\pm\delta}$ via the Liouville flow.}
\label{fig:liouv}
\end{figure}

We next estimate the function $T$ and its derivative. The time $t$ flow of $Z$ with initial condition $p=(x_1(0),y_1(0),x_2(0),y_2(0))$ is given by
\[
\Phi_{Z}^{t}(p)=
\left(
e^{\tfrac12 t}x_1(0)\,,\,e^{\tfrac12 t}y_1(0)\,,\,e^{2 t}x_2(0)\,,\,e^{-t}y_2(0)
\right).
\]
Thus the function $T=T(p)$ solves the equation
\[
\frac12e^{T}\left(x_1(0)^{2}+y_1(0)^{2}\right)+e^{4T}2x_2(0)^{2}-e^{-2T}y_2(0)^{2}=\delta^{2}.
\]
Using the fact that the initial value lies in
$B_{\rho}(\delta)-B_{\rho/2}(\delta)\subset V_{-\delta}$, we can
rewrite this as
\begin{align}
\begin{split}
\delta^2 &=
\frac12\left(e^{T}-e^{-2T}\right)\left(x_1(0)^{2}+y_1(0)^{2}\right)
+2\left(e^{4T}-e^{-2T}\right)x_2(0)^{2}
-\delta^{2}e^{-2T} \\%\notag\\
&= \frac{e^{T}}{2}\left(x_1(0)^{2}+y_1(0)^{2}\right)+2e^{4T}x_2(0)^{2}
-e^{-2T}\left(\delta^{2}+\frac12\left(x_1(0)^{2}+y_1(0)^{2}\right)
+2x_2(0)^{2}\right).
\\\label{eq:defneqnT}
\end{split}
\end{align}
Noting that all coefficient functions (functions that depend on $T$) in the final line of \eqref{eq:defneqnT} are increasing in $T$, we deduce first that
\[
T=\Ordo(\delta^{2}/\rho^{2})
\]
and second that the solution $T$ decreases monotonically with the quantity $\left(x_1(0)^{2}+y_1(0)^2+x_2(0)^{2}\right)$ and hence that the differential of $T$ satisfies
\[
|dT|=\Ordo(\delta^{2}/\rho^{3}).
\]

Let $C_{\rho}(\delta)\subset V_{\delta}$ be the region with boundary
$\phi(\pa B_\rho^{+}(\delta)\cup \pa B_{\rho}^{-}(\delta))$ which
contains
$\phi\left(B_{\rho}^{\pm}(\delta)-B_{\rho/2}^{\pm}(\delta)\right)$,
let $C_{\rho}'(\delta)\subset V_{\delta}$ be the region with boundary
$\phi\left(\pa B_{\rho/2}^{+}(\delta)\cup B_{\rho/2}^{-}(\delta)\right)$
which does not contain
$\phi\left(B_{\rho}^{\pm}(\delta)-B_{\rho/2}^{\pm}(\delta)\right)$,
and let
$C^{\pm}_{\rho}(\delta)=\phi\left(B_{\rho}^{\pm}(\delta)-B_{7\rho/8}^{\pm}(\delta)\right)$. See Figure~\ref{fig:liouv}.

Let $\alpha_{\rho;\delta}$ denote the contact form on $C_{\rho}(\delta)$ given by
\[
\alpha_{\rho;\delta}=e^{f}\alpha_{\delta},
\]
where $f\colon C_{\rho}(\delta)\to\R$ has the following properties:
\begin{itemize}
\item $f=-T\circ\phi^{-1}$ on $C^{\pm}_{\rho}(\delta)$;
\item $f=0$ on $C_{\rho}'(\delta)$;
\item $|df|=\Ordo(\delta^{2}/\rho^{3})$.
\end{itemize}
The above estimates on $T$ show that such a function $f$ exists. Then
\[
\phi\colon (B_{\rho}^{\pm}(\delta)-B_{7\rho/8}^{\pm}(\delta),\alpha_{-\delta})\to (C_\rho^{\pm}(\delta),\alpha_{\rho;\delta})
\]
is a contactomorphism.

Finally, as discussed above, we consider $B^{\pm}_{\rho}(\delta)$ as neighborhoods of $p^{\pm}$ in $\R^{3}$, and then define $Y_{k}(\delta)$ as the contact manifold
\[
Y_{k}(\delta)=\R^{3}_{\rho}(\circ_{k})\,\,\,\bigcup_{\phi}\,\,\,\left(\sqcup_{\ell=1}^{k} C_{\rho}^\ell(\delta)\right),
\]
where $C^\ell_{\rho}(\delta)$ is a copy of $C_{\rho}(\delta)$ attached via the map $\phi$ at the pair of points $\left(0,y^{\ell},z^{\ell}\right)$ and $\left(A,\tilde{y}^{\ell},\tilde{z}^{\ell}\right)$. We denote the contact form on $Y_k(\delta)$ by $\alpha_{k;\delta}$.

\begin{lma}\label{lma:orbitsinY_k}
For all sufficiently small $\delta>0$, there are exactly $k$ distinct geometric Reeb orbits $\gamma_1,\dots,\gamma_k$ in $Y_k(\delta)$, one in each handle. Furthermore if $\gamma_j^{m}$ denotes the $m^{\rm th}$ iterate of $\gamma_j$ then
\[
\CZ(\gamma^{m}_j)=2m,
\]
where $\CZ$ denotes the Conley--Zehnder index measured with respect to the trivialization $(v,iv)$ as in \eqref{eq:contacttrivhandle}.
\end{lma}

\begin{pf}
Since the $C^{1}$-distance between the contact forms $\alpha_{\delta}$ and $\alpha_{\rho;\delta}$ on $C_\rho(\delta)$ is controlled by $\delta$, the lemma is an immediate consequence of Lemma \ref{lma:orbitsinhandle}.
\end{pf}

Choose $R>0$ such that $B_{R/2}$ contains all the balls $B_{\rho}$ where the $1$-handles of $Y_k(\delta)$ are attached. (The factor of $2$ is not used here but will be used in Section~\ref{sssec:tildeY_k}.) Then $\R^{3}-B_{R}\subset Y_k(\delta)$, and we write
\begin{equation}\label{eq:olY_k}
\ol{Y}_k(\delta)=Y_k(\delta)-\inr(\R^{3}-B_R).
\end{equation}
Then the contact form $\alpha_{k;\delta}$ on $\ol{Y}_k(\delta)$ agrees with $\alpha_{\rm b}$ in a neighborhood of $\pa\ol{Y}_k(\delta)$ which is identified with a neighborhood of $\pa B_R$ in $B_R$.

\subsubsection{Legendrian links in $Y_k(\delta)$}\label{sec:leginY_k}
Let $\Lambda\subset Y_k(\delta)$ be any Legendrian link. Then there exists a contact isotopy which moves $\Lambda$ to a link in normal form, see Section \ref{ssec:resolution}. Below, all our links will be assumed to be in normal form. With notation as in Section \ref{ssec:dga} we have the following result.

\begin{lma}\label{lma:stchord}
Any Reeb chord of a link $\Lambda$ in normal form is either entirely contained in the handle and then of the form $c^{p}_{ij}$ or it lies completely in the complement of all handles.
\end{lma}

\begin{pf}
It is straightforward to check that no Reeb chord can connect a point on $\Lambda$ inside a handle to a point outside the handle (compare Section \ref{ssec:handledisks}). The last statement follows from the fact that inside the handles $\Lambda$ is the graph of the differential of a function on the standard strand in a small $1$-jet neighborhood of that strand in combination with Lemma \ref{lem:handleReebtv}.
\end{pf}

We call Reeb chords of the first type mentioned in Lemma \ref{lma:stchord} {\em handle chords} and those of the second type {\em diagram chords}.

\subsubsection{The closed manifold $\tilde Y_{k}(\epsilon;\delta)$}\label{sssec:tildeY_k}
Consider the hypersurface
\[
E(a):=\left\{(z_1,z_2)\in\C^{2}\colon |z_1|^{2}+a^{-1}|z_2|^{2}=1\right\}.
\]
The vector field
\[
W=\frac12\left(x_1\pa_{x_1}+y_1\pa_{y_1}+x_2\pa_{x_2}+y_2\pa_{y_2}\right)
\]
is a Liouville vector field for $\omega_{\rm st}$. Since $W$ is transverse to $E(a)$ it induces a contact form
\[
\alpha_{a}=\omega(W,\cdot)=
\left(x_1\,dy_1-y_1\,dx_1+x_2\,dy_2-y_2\,dx_2\right)|_{E(a)}.
\]
The corresponding contact structure is isotopic to the standard contact structure on $S^{3}=E(1)$. The Reeb vector field on $E(a)$ is
\[
R_{a}=(x_1\pa_{y_1}-y_1\pa_{x_1})+a^{-1}(x_2\pa_{y_2}-y_2\pa_{x_2}).
\]

\begin{figure}
\labellist
\small\hair 2pt
\pinlabel $E(a)$ [tr] at 57 37
\pinlabel $F$ [l] at 329 93
\pinlabel $B_\epsilon(a)$ [b] at 220 182
\pinlabel $z_2=0$ [r] at 31 134
\pinlabel $z_1=0$ [b] at 180 208
\pinlabel $B_d$ [r] at 259 31
\endlabellist
\includegraphics[width=0.8\textwidth]{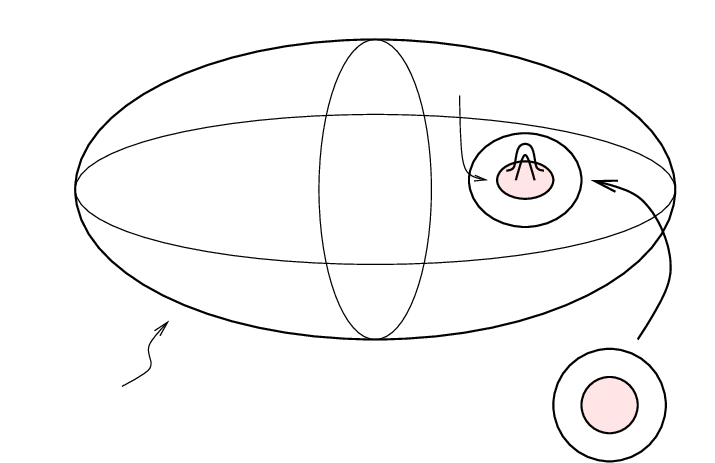}
\caption{Constructing $\tilde{Y}_k(\epsilon;\delta)$; $1$-handles are attached in a small ball that does not intersect the two closed Reeb orbits.}
\label{fig:Yk}
\end{figure}
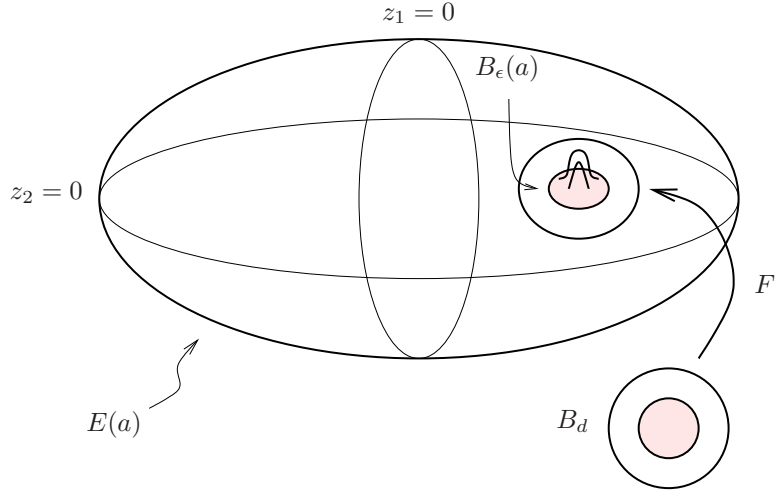

Let $F\colon B_d\to E(a)$ be a contact embedding of a standard ball of
radius $d$ such that $F(B_{d})$ does not intersect any complex
coordinate plane. For $\epsilon>0$ sufficiently small, let
$B_{\epsilon}(a)\subset E(a)$ denote the image under $F$ of the
standard ball of radius $\epsilon$ centered at $0\in B_d$.
See Figure~\ref{fig:Yk}.

\begin{lma}\label{lma:irrationalflow}
If $a$ is irrational then the closed Reeb orbits in $E(a)$ are exactly the multiples of the two circles $E(a)\cap (\C\times\{0\})$ and $E(a)\cap (\{0\}\times\C)$. Furthermore, if the Liouville--Roth exponent of $a$ equals $p>2$ then there is a constant $K>0$ such that the following holds for all sufficiently small $\epsilon>0$. If $\gamma(t)$ is a Reeb trajectory which leaves $B_{\epsilon}$ at $t=0$ and if $T>0$ is such that $\gamma(T)\in B_{\epsilon}$, then $T\ge K\epsilon^{-\frac{1}{p}}$.
\end{lma}

\begin{pf}
The lemma is a consequence of the fact that outside the periodic orbits the Reeb flow is an irrational rotation on a torus with slope $a$: an orbit which returns to $B_{\epsilon}$ gives a rational approximation of $a$ of the form $\left|a-\frac{n}{m}\right|<\epsilon$. But then, by definition of the Liouville--Roth exponent, $\epsilon>m^{-p}$, for all sufficiently small $\epsilon$, and the action of $\gamma$ is bounded below by $Km$ for some $K$.
\end{pf}

Consider the map $\psi_{\epsilon}\colon B_R\to B_\epsilon$ defined by
\[
\psi_{\epsilon}(u,v,z)=\left(\tfrac{\epsilon}{R}\, u,\tfrac{\epsilon}{R}\,v,\left(\tfrac{\epsilon}{R}\right)^{2}z\right)
\]
and note that $\psi_{\epsilon}^{\ast}\alpha_{{\rm b}}=\left(\frac{\epsilon}{R}\right)^{2}\alpha_{\rm b}$. Since $\ol{Y}_k(\delta)$ agrees with $B_R$ near its boundary, see \eqref{eq:olY_k}, we use the map $F\circ \psi_\epsilon$ in a neighborhood of the boundary to attach $\left(\ol{Y}_k(\delta),\left(\frac{\epsilon}{R}\right)^{2}\alpha_{k;\delta}\right)$ to $E(a)-B_\epsilon$. We denote the resulting contact manifold $\tilde Y_{k}(\epsilon;\delta)$ and its contact form $\tilde\alpha_{\epsilon;\delta}$.

Let $\Lambda\subset \ol{Y}_k(\delta)$  be a Legendrian link. Using the inclusion $\left(\ol{Y}_k(\delta),\left(\frac{\epsilon}{R}\right)^{2}\alpha_{k;\delta}\right)\to\left(\tilde Y_{k}(\epsilon;\delta),\tilde\alpha_{\epsilon;\delta}\right)$, we consider $\Lambda$ as a Legendrian link in $\tilde Y_{k}(\epsilon;\delta)$ and as such denote it $\Lambda(\epsilon)$.
\begin{cor}\label{cor:ReebflowtildeY}
If the Liouville--Roth exponent of $a\notin\Q$ equals $p>2$ then there is a constant $K>0$ such that the following holds for all sufficiently small $\epsilon>0$. If $c$ is a Reeb chord of $\Lambda(\epsilon)\subset\tilde Y_{k}(\epsilon;\delta)$ which is not entirely contained in $\ol{Y}_k(\delta)\subset\tilde Y_{k}(\epsilon)$ or if it is a Reeb orbit which is not a Reeb orbit in $E(a)$ then
\[
\int_{c}\tilde\alpha_{\epsilon;\delta}\ge K\epsilon^{-\frac{1}{p}}.
\]
\end{cor}

\begin{pf}
Immediate from Lemma \ref{lma:irrationalflow}.
\end{pf}

\subsubsection{Exact cobordisms $(W_k(h;\delta),L(h))$}\label{sec:excob}
Let $\Lambda\subset \ol{Y}_k(\delta)$ be a Legendrian link. As above we consider $\Lambda$ as a link in $\tilde Y_{k}(\epsilon;\delta)$ and as such we denote it $\Lambda(\epsilon)$. For suitable positive functions $h\colon\R\to\R$ which are constantly equal to $\epsilon_{\pm}$ in a neighborhood of $\pm\infty$ we construct a symplectic cobordism $(W_k(h;\delta),L_k(h))$ with positive and negative ends $\left(\tilde Y_k(\epsilon_+;\delta),\Lambda(\epsilon_+)\right)$ and $\left(\tilde Y_k(\epsilon_-;\delta),\Lambda(\epsilon_-)\right)$, respectively. Topologically these cobordisms will simply be products $\R\times \tilde Y_k(\epsilon;\delta)$, but the symplectic form will not be the symplectization of a contact form. We construct them as follows.

Consider the standard contact ball $B_d$ of radius $d$ embedded into $E(a)$ and let $0<\epsilon_\pm\ll d$. We use strictly positive smooth functions $h\colon\R\to\R$ which satisfy the following conditions:
\begin{itemize}
\item There are $T_-<0< T_+$ such that
\begin{equation}\label{eq:hinftycond}
h(t)=
\begin{cases}
\epsilon_- &\text{if }t\in(-\infty,T_-],\\
\epsilon_+ &\text{if }t\in[T_+,\infty).
\end{cases}
\end{equation}
\item The derivative of $h$ satisfies
\begin{equation}\label{eq:hC1cond}
2h'(t)+h(t)>0.
\end{equation}
\end{itemize}

Let $B_d(h)$ denote the manifold
\[
B_d(h)=\left\{(t,p)\in \R\times B_d\colon |p|\ge \frac12 h(t)\right\}\approx_{\rm top}
\R\times(B_d-B_{\epsilon_-}).
\]
Endow $B_d(h)$ with the exact symplectic form $d\left(e^{t}\alpha_{\rm b}\right)$. Recall that we consider $B_d$ as embedded in $E(a)$ and define $\tilde B_d(h)\subset\R\times E(a)$ as follows:
\[
\tilde B_d(h)= (\R\times (E(a)-B_d)) \,\,\cup\,\, B_d(h).
\]
Then the primitive $e^{t}\alpha_{\rm b}$ of the symplectic form on $B_{d}(h)$ extends as $e^{t}\alpha_a$ to $\R\times (E(a)-B_d)$. Using this extension we consider also $\tilde B_{d}(h)$ as an exact symplectic manifold and we denote the primitive of its symplectic form $e^{t}\tilde \alpha$.

Consider the region $A$ in $\ol{Y}_k(\delta)$ outside the boundary of the standard contact ball of radius $\frac{R}{2}$:
\[
A=B_R-B_{R/2}\subset \ol{Y}_k(\delta),
\]
see \eqref{eq:olY_k}. The map
\[
\Phi\colon \R\times A\to B_d(h),\quad \Phi(t,q)=\left(t,h(t)u,h(t)v,h^{2}(t)z\right)
\]
is a level preserving embedding such that
\[
\Phi^{\ast}\left(e^{t}\alpha_{\rm b}\right)=\left(e^{t}h^{2}(t)\right)\alpha_{k;\delta}.
\]
(Note that $h$ is increasing so that $|h(t)|\ll d$ for all $t\in\R$ and the image of $\Phi$ lies inside $B_{d}$.)
We define the exact symplectic cobordism $W_k(h)$ as follows:
\begin{equation}\label{eq:defW(h)}
W_k(h)= \left(\tilde B_d(h),e^{t}\alpha\right)\,\,\cup_\Phi\,\, \left(\R\times \ol{Y}_k(\delta),\left(e^{t}h^{2}(t)\right)\alpha_{k;\delta}\right),
\end{equation}
where $\Phi$ is the gluing map. We must check that the form $d\left(e^{t}h^{2}(t)\alpha_{k;\delta}\right)$ on $\R\times \ol{Y}_k(\delta)$ is symplectic. We have
\begin{equation}\label{eq:cobsymplform}
d\left(e^{t}h^{2}(t)\alpha_{k;\delta}\right)=e^{t}h^{2}(t)\, d\alpha_{k;\delta} + e^{t}h(t)\left(2h'(t)+h(t)\right)\,dt\wedge \alpha_{k;\delta}
\end{equation}
and \eqref{eq:hC1cond} implies that this is indeed symplectic.

Inside this symplectic cobordism we also have an exact Lagrangian cobordism $L(h)$ interpolating between the Legendrian submanifolds $\Lambda(\epsilon_+)$ and $\Lambda(\epsilon_-)$:
\[
L(h)=\R\times\Lambda\subset \R\times\ol{Y}_k\subset W_k(h).
\]
To see that $L(h)$ is Lagrangian, note that its tangent space is spanned by the tangent vector of $\Lambda$ and $\pa_t$. The vanishing of the restriction of the symplectic form then follows immediately from \eqref{eq:cobsymplform} in combination with $\Lambda$ being Legendrian.

Finally, we note that in the regions near $t=-\infty$ and $t=+\infty$ in \eqref{eq:hinftycond} where $h(t)$ is constant, $\left(W_k(h),L(h)\right)$ is symplectomorphic to the symplectizations of $\left(\tilde Y_k(\epsilon_-;\delta),\Lambda(\epsilon_-)\right)$ and $\left(\tilde Y_k(\epsilon_+;\delta),\Lambda(\epsilon_+)\right)$, respectively.

%*********************************************************************
%*********************************************************************
\section{Legendrian homology in closed and open manifolds}\label{Sec:openclosed}
In this section we study the analytical aspects of Theorem \ref{thm:main}, and in particular obtain rather explicit descriptions in Section~\ref{sec:holodisks} of the moduli spaces of holomorphic disks involved. In Section~\ref{Sec:compute} we then study the actual solution spaces, and in Section \ref{sec:orientaton} their orientation. However, before going into the detailed aspects of this study, we give in Section \ref{sec:generalLCH} a more general overview of Legendrian (contact) homology in order to provide a wider context of the more technical study that follows.

%*********************************************************************
\subsection{Legendrian homology in the ideal boundary of a  Weinstein manifold}\label{sec:generalLCH}
Our discussion in this section follows \cite{bib:BEE}.
Let $Y$ be a contact manifold which is the ideal boundary of a Weinstein manifold $X$ with vanishing first Chern class $c_1(X)=0$. Pick an almost complex structure $J$ on $X$ which is compatible with the symplectic form and which in the end $[0,\infty)\times Y$ of $X$ splits as a complex structure in the contact planes of $Y$ and pairs the $\R$-direction $\partial_t$ with the Reeb vector field $R$ of the contact form on $Y$, $J\partial_t=R$. In this setup holomorphic curves satisfy SFT-compactness \cite{BEHWZ}: any finite energy holomorphic curve in $X$ is punctured and is asymptotic to a Reeb orbit cylinder $\R\times\gamma\subset \R\times Y$ at infinity. Furthermore, any sequence of holomorphic curves converges to a several-level holomorphic building with one level in $X$ and several levels in $\R\times Y$, where levels are joined at Reeb orbits.

Let $\Lambda\subset Y$ be a Legendrian submanifold. The Legendrian homology algebra $\A_{\mathrm{H}}(Y,\Lambda)$ is the algebra freely generated over $\Z[H_{2}(X,\Lambda)]$ by the set of Reeb chords of $\Lambda$ graded by a Maslov index, see \cite[Section 2.1]{bib:BEE}. The differential in $\A_{\mathrm{H}}(Y,\Lambda)$ satisfies the Leibniz rule and is defined through a count of so-called anchored holomorphic disks. These are two-level holomorphic buildings of the following form. The top level is a map $u\colon (D,\partial D)\to (\R\times Y,\R\times \Lambda)$, where $D$ is a disk with the following punctures: one positive boundary puncture near which $u$ is asymptotic to $\R\times a$ for some Reeb chord $a$ of $\Lambda$, several negative boundary punctures where $u$ is asymptotic to $\R\times  b$ for some Reeb chord $b$ of $\Lambda$ (possibly different for different punctures), and interior negative punctures where $u$ is asymptotic to $\R\times \gamma$ for some Reeb orbit $\gamma$ (again possibly different for different punctures). The lower level consists of holomorphic spheres with positive punctures at the Reeb orbits of all interior negative punctures of $u$.
We write $\M_A^{(\R\times Y;X)}$
for the moduli space of such buildings, where $A$ denotes the homology class of the building. The differential acting on a Reeb chord $a$ is now defined as follows:
\[
\partial_{\mathrm{H}}a=\sum_{\dim(\M_{A}^{(\R\times Y;X)}(a;\mathbf{b}))=1}|\M_{A}^{(\R\times Y;X)}(a;\mathbf{b})|\,A\mathbf{b},
\]
where $|\M_{A}^{(\R\times Y;X)}|$ denotes the number of $\R$-components in the moduli space (recall that $J$ is $\R$-invariant in the end $\R\times Y$). Here the moduli spaces $\M_{A}^{(\R\times Y;X)}(a;\mathbf{b})$ are oriented manifolds and the number of $\R$-components refers to a signed count of components.

\begin{rmk}\label{rmk:d^2=0}
The fact that $\partial_{\mathrm{H}}^{2}=0$ follows by identifying configurations that contribute to $\partial_{\mathrm{H}}^{2} a$ with the broken curves at the boundary of the $1$-manifold of anchored holomorphic disks of dimension $1$ with positive puncture at $a$; this $1$-manifold is the quotient of the corresponding 2-dimensional moduli space by the $\R$-action.
\end{rmk}

\begin{rmk}\label{rmk:coeffs}
When considering Legendrian homology of $\Lambda\subset Y$ without
using the filling $X$ it is natural to use coefficients for the DGA in
$\Z[H_{2}(Y,\Lambda)]$, rather than in $\Z[H_{2}(X,\Lambda)]$. However, in the combinatorial section of this paper, we have used a third coefficient ring, $\Z[H_1(\Lambda)] \cong \Z[\mathbf{t},\mathbf{t}^{-1}]$. Here we explain why this suffices in our situation.

In the
case studied in this paper, $Y = \#^{k}(S^{1}\times S^{2})$ and $X$ is obtained by attaching $k$ $1$-handles to the $4$-ball. Thus $H_2(X,\Lambda) \cong \ker(H_1(\Lambda) \to H_1(X))$, while there is an exact sequence
\[
0 \to H_2(Y) \to H_2(Y,\Lambda) \to H_1(\Lambda) \to H_1(Y).
\]
Using the explicit form given below for the holomorphic disks contributing to the differential, it is easy to check that one can choose caps for the Reeb chords of $\Lambda$ so that the $H_2(Y)$ portion of the homology class of any such disk is trivial. It follows that coefficients in either $H_2(X,\Lambda)$ or $H_2(Y,\Lambda)$ reduce to coefficients in $\ker(H_1(\Lambda) \to H_1(X))$.

Note that $\ker(H_1(\Lambda) \to H_1(X))$ may in general be smaller than $H_1(\Lambda)$. One might as well use coefficients in $\Z[H_1(\Lambda)]$, as we do, rather than in $\Z[\ker(H_1(\Lambda) \to H_1(X))]$, since this clearly does not lose information. We can interpret the fact that the coefficients reduce from $H_1(\Lambda)$ to $\ker(H_1(\Lambda) \to H_1(X))$ as follows: if $\Lambda_j$ is a component of $\Lambda$ such that $[\Lambda_j] \ne  0 \in H_1(X)$, then one can choose capping paths such that $t_j^{\pm 1}$ does not appear in the differential. Indeed, this can be shown directly combinatorially: if $\Lambda_j$ passes algebraically a nonzero number of times through one of the $1$-handles, then replace the single base point $\ast_j$ on $\Lambda_j$ by multiple base points, one on each strand passing through that $1$-handle. (For the relation between the DGAs for single and multiple base points, see \cite[section 2.6]{bib:NR}.) Any holomorphic disk whose boundary passes through these base points must then pass through them in canceling pairs, and so $t_j^{\pm 1}$ does not appear in the differential.
\end{rmk}

We stay in the general setting in order to explain the functorial properties of Legendrian homology.
Consider a Weinstein $4$-manifold $W$ with an exact Lagrangian submanifold $L\subset W$. Assume that outside a compact set, $(W,L)$ consists of two ends: a negative end symplectomorphic to the negative half $((-\infty,0]\times Y_-,(-\infty,0]\times \Lambda_-)$ of a symplectization of a pair $(Y_-,\Lambda_-)$, where $\Lambda_-$ is a Legendrian submanifold of contact $Y_-$,
and a positive end symplectomorphic to the positive half $([0,\infty)\times Y_+,[0,\infty)\times \Lambda_+)$ of a symplectization of a pair $(Y_+,\Lambda_+)$. Assume that $Y_-$ is the ideal boundary of a Weinstein manifold $X$. Then $Y_+$ is the ideal boundary of the Weinstein manifold $X\circ W$ obtained by gluing $W$ to $X$ along $Y_-$. If $a$ is a Reeb chord of $\Lambda_{+}$ and $\mathbf{b}$ is a word of Reeb chords of $\Lambda_-$, we let $\M_{A}^{(W;X)}(a;\mathbf{b})$ denote the moduli space of holomorphic disks in $(W,L)$ anchored in $X$, with positive puncture at $a$ and negative punctures according to $\mathbf{b}$; the definition of $\M_{A}^{(W;X)}(a;\mathbf{b})$ precisely generalizes the previous definition of $\M_{A}^{(\R\times Y;X)}(a;\mathbf{b})$. Then the algebra map $\Phi\colon \A_{\mathrm{H}}(Y_+,\Lambda_+)\to\A_{\mathrm{H}}(Y_-,\Lambda_-)$ defined on generators as
\[
\Phi(a)=
\sum_{\dim(\M_{A}^{(W;X)}(a;\mathbf{b}))=0}|\M_{A}^{(W;X)}(a;\mathbf{b})|\,A\mathbf{b}
\]
is a chain map (i.e.~a  morphism of DGAs).
\begin{rmk}\label{rmk:chmap}
The proof of the chain map equation $\Phi\circ\partial_+-\partial_-\circ\Phi=0$ is analogous to the proof of $\partial_{\mathrm{H}}^{2} =0$: configurations contributing to $\Phi\circ\partial_+-\partial_-\circ\Phi$  are in oriented one-to-one correspondence with the broken curves at the boundary of the $1$-manifold of anchored holomorphic disks in $(W,L)$ of dimension $1$ with one positive boundary puncture.
\end{rmk}
Furthermore, a similar but more involved argument shows that a $1$-parameter family of exact symplectic cobordisms $(W_{t},L_t)$, $t\in[0,1]$ gives a chain homotopy between the chain maps $\Phi_\tau$ induced by $(W_{\tau},L_{\tau})$, $\tau=0,1$, see Appendix \ref{app:aninv} for a more detailed discussion.

In order to make sense of these definitions of differentials and chain maps, one needs the moduli spaces to be transversely cut out. For the disks in the upper level of the holomorphic buildings involved this is relatively easy: the argument in \cite[Lemma 4.5(1)]{EES-PxR} shows that it is possible to achieve transversality by varying $J$ in the contact planes near the Reeb chord end points. For the lower level, achieving transversality is more involved: because of multiple covers it is not sufficient to perturb only $J$, and a more elaborate perturbation scheme is needed, e.g.~using the polyfold framework of \cite{bib:HWZ} or Kuranishi structures as in \cite{bib:FOOO}. For an argument adapted to the case just discussed (i.e.~disks anchored in a Weinstein filling), see \cite[Section 2h]{bib:AS}. In the case under study in this paper we will show that these more elaborate perturbation schemes are not needed, see Corollary \ref{cor:noanchor} and Lemma \ref{lem:noanchor2}, and hence they will not be further discussed.

For purposes of computing Legendrian homology, the open manifold $Y_k(\delta)$ is simpler than the closed $\tilde Y_k(\epsilon;\delta)$ because of the absence of ``wandering'' chords that leave the region where the Legendrian link lies and then come back. One of the main results of this section shows that there is no input in Legendrian homology from these wandering chords. More precisely, by inclusion, a Legendrian link $\Lambda\subset \ol{Y}_k(\delta)$ can be viewed as a link in $\tilde{Y}_k(\epsilon;\delta)$ or in $Y_k(\delta)$, and we show that if $\epsilon$ and $\delta$ are sufficiently small then there is a canonical isomorphism between the Legendrian homology DGAs $\A_{\mathrm{H}}(\tilde Y_k(\epsilon;\delta),\Lambda(\epsilon))$ and $\A_{\mathrm{H}}(Y_k(\delta);\Lambda)$ below a given grading.

In Section \ref{sec:holodisks} we use the cobordisms in Section
\ref{sec:excob} interpolating between $\tilde Y_k(\epsilon;\delta)$ for different parameter values and an argument inspired by arguments of Bourgeois--van Koert \cite{bib:BvK} and Hutchings \cite{bib:H}, to show that the Legendrian homology of a link $\Lambda\subset \ol{Y}_k(\delta)$, considered as a subset of $\tilde Y_k(\epsilon;\delta)$, is isomorphic to the Legendrian homology of $\Lambda\subset \ol{Y}_k(\delta)$, considered as a subset of $Y_k(\delta)$. Finally in Section \ref{Sec:compute} we show that for certain regular complex structures on $\R\times Y_k(\delta)$, the Legendrian homology differential of a link $\Lambda$ in standard position is given by the combinatorial formula from Section \ref{ssec:dga}.

%*********************************************************************
\subsection{Holomorphic disks}\label{sec:holodisks}
Recall that the manifold $Y_{k}(\delta)$ was built by attaching small 1-handles to $\R^{3}$. We use notation as in Section \ref{ssec:Y_k(delta)}.
\subsubsection{Almost complex structures}
Consider the regions of the form $B_{\rho}-B_{\rho/2}$ where the handles are attached to $\R^{3}$. We will use an almost complex structure on $\R\times Y_k(\delta)$ which is induced from the standard complex structure on the plane outside the $B_{3\rho/4}$ and which interpolates in the region $B_{3\rho/4}-B_{\rho/2}$ to an almost complex structure on the contact planes in the handle that agrees with the standard complex structure in a $1$-jet neighborhood of the standard Legendrian strand, inside of which the part of the link that runs through the handle lies. For later convenience, we assume that the size of the $1$-jet neighborhood where the complex structure is standard equals $\delta_0\ll\delta$ and that the strands of the link in the handle lie within distance $\delta_1\ll\delta_0$ from the standard strand through the handle.

As we choose the Legendrian to lie close to straight line segments near the attaching regions, the almost complex structure outside the handles and the almost complex structure inside the neighborhood of a standard strand agree and the interpolation will be chosen trivial here. This means in particular that the almost complex structure agrees with the standard $1$-jet structure all along a neighborhood of the extended strand which goes through the handle and continues out in $\R^{3}$.

After scaling by $\epsilon$ we consider $Y_k(\delta)$ as a subset of $\tilde{Y}_{k}(\epsilon;\delta)$ concentrated in a small ball in $E(a)$. We extend the almost complex structure over the contact planes over the rest of this manifold in some fixed way.

\subsubsection{Trivial anchoring}
As discussed in Section \ref{sec:generalLCH}, Legendrian homology for boundaries of Weinstein domains is defined by counting anchored holomorphic disks, i.e.~disks with additional interior negative punctures that are filled by rigid holomorphic planes in the Weinstein manifold. Here we show that no such extra interior negative punctures are needed in the cases of $Y_{k}(\delta)$ and $\tilde Y_{k}(\epsilon;\delta)$, which we consider as  the boundaries of their natural Weinstein fillings: a half space in $\C^{2}$ and the ball in $\C^{2}$, respectively, with $k$ 1-handles attached. In fact, a similar result holds for boundaries of subcritical Weinstein manifolds in any dimension but we restrict attention to the case of dimension $3$ since that is all we need here.

\begin{lma}
If the formal dimension of an anchored holomorphic disk mapping to $(\R\times Y_{k}(\epsilon;\delta),\R\times\Lambda(\epsilon))$ (or to $(\R\times Y_{k}(\delta),\R\times \Lambda)$) is $\le 1$ then the disk has only one level and no interior punctures.
\end{lma}

\begin{proof}
It is well known that the Conley--Zehnder index in $E(a)$ is proportional to action. It follows that if $\gamma$ is a Reeb orbit in $\tilde Y_{k}(\epsilon;\delta)$ that is not contained in one of the handles, then $|\gamma|>1$. Thus the minimal grading $|\gamma|$ of an orbit in $Y_{k}(\epsilon;\delta)$ is attained at the central Reeb orbits in the handles and satisfies $|\gamma|=1$, and the same holds in $Y_{k}(\delta)$. This implies that if the formal dimension of a holomorphic building with interior negative punctures at some $\gamma$ equals $1$, then the disk in the symplectization must have dimension $<1$ and hence does not exist by transversality for disks with one positive boundary puncture, as the only such disks are trivial strips.
\end{proof}

\begin{cor}\label{cor:noanchor}
The differential of the DGA of a Legendrian link in $\tilde Y_{k}(\epsilon;\delta)$ or in $Y_{k}(\delta)$ is defined through disks without interior punctures.\qed
\end{cor}

\subsubsection{The core algebra}\label{sec:corealgebra}
Consider a Legendrian link $\Lambda(\epsilon)\subset\tilde Y_k(\epsilon;\delta,)$ using the notation above. The Reeb chords of $\Lambda(\epsilon)$ fall into two classes, {\em interior} and {\em exterior}, where the interior chords are entirely contained in $\ol{Y}_k(\delta)\subset\tilde Y_k(\epsilon;\delta)$ and the exterior chords are not. As in Section \ref{sec:leginY_k} we further subdivide the interior chords into diagram chords and handle chords.

Let $A(\Lambda)$ denote the algebra generated by all interior chords and let $A_{\text{1-h}}(\Lambda)$ denote the algebra generated by all handle chords. Since the contact form on $\ol Y_k(\delta)\subset \tilde Y_k(\epsilon;\delta)$ depends on $\epsilon$ only through scaling, it is clear that $A(\Lambda)$ is independent of $\epsilon$.

\begin{lma}\label{lma:subalgebra}
If $\epsilon>0$ and $\delta>0$  are sufficiently small then $A(\Lambda)$ and $A_{\text{1-h}}(\Lambda)$ are sub-DGAs of $\A_{\mathrm{H}}(\tilde Y_{k}(\epsilon;\delta),\Lambda(\epsilon))$. Furthermore, the differentials on $A(\Lambda)$ and $A_{\text{1-h}}(\Lambda)$ agree with the differentials on these algebras obtained by considering $\Lambda$ as a Legendrian link in $Y_{k}(\delta)$, i.e.~using the differential on $\A_{\mathrm{H}}(Y_{k}(\delta),\Lambda)$.
\end{lma}
\begin{pf}
We first show that the differential acting on an interior chord is a sum of monomials of interior chords. We start with diagram chords. Let $a$ denote a diagram chord. By definition of the contact form on $\tilde Y_k(\epsilon;\delta)$ there exists a  constant $\ell$ such that the action of any diagram chord is bounded by $\ell\epsilon^{2}$. Since by Lemma \ref{lma:orbitsinY_k} the action of any exterior chord is bounded below by $K\epsilon^{-\frac{1}{p}}$ for some constant $K$ it follows that no holomorphic disk with one positive puncture at $a$ can have any negative punctures mapping to exterior chords.

We consider second the case of a handle chord. Let $a$ denote a handle chord and let $n_a$ denote the number of times it intersects the subset $\{(x_1,y_1,x_2,y_2)\in V_{\delta}\colon y_1=0\}$. Note that there exists $\ell_0$ such that the action of any diagram chord is bounded below by $\ell_0\epsilon^{2}$ and that the action of $a$ equals $(n_a+\theta)\delta\epsilon^{2}$ for some $-1<\theta< 1$. Let $u$ be a holomorphic disk with positive puncture at $a$, $m$ negative punctures at diagram chords, and $t$ negative punctures at exterior chords $b_1,\dots,b_t$, and with boundary representing the homology class $\sum_{j=1}^{s}r_j[\Lambda_j]$. Since the disk has nonnegative $d(e^{t}\tilde\alpha)$-area, we find by Stokes' Theorem and monotonicity near base points in the diagram part of the link that
\[
\mathfrak{a}(a)>
\sum_{j=1}^{t}{\mathfrak a}(b_j) + m\ell_0\epsilon^{2} + rc_0\epsilon^{2},
\]
where ${\mathfrak a}(c)$ denotes the action of a Reeb chord $c$, and where $r=\sum_{j=1}^{s}|r_j|$. Here we use the fact that the projection of any disk passing a base point covers a small half disk near this base point.

The sum of gradings of the chords at the negative end is then bounded above by
\[
I=k_{\rm ext}\sum_{j=1}^{t}{\mathfrak a}(b_j) + mi_0 +ri_1,
\]
where $k_{\rm ext}$ is related to the constant of proportionality between Conley--Zehnder index and action in $E(a)$ and where $i_0$ and $i_1$ are grading bounds for diagram chords and for the homology classes $[\Lambda_j]$.
The area inequality then gives
\[
(n_{a}+\theta)\delta\epsilon^{2}>
\sum_{j=1}^{t}{\mathfrak a}(b_j) + m\ell_0\epsilon^{2} + rc_0\epsilon^{2},
\]
and consequently,
\[
n_a>\frac{1}{\delta}(\epsilon^{-2}\sum_{j=1}^{t}{\mathfrak a}(b_j) + m\ell_0 + rc_0)-\theta.
\]
Thus, if $M$ is the maximal difference in Maslov potential for two strands passing through a handle then the grading of $a$ satisfies
\[
|a| \geq 2n_a-1- M >\frac{2}{\delta}(\epsilon^{-2}\sum_{j=1}^{t}{\mathfrak a}(b_j) + m\ell_0 + rc_0)-2\theta-1-M>I+1
\]
provided $\delta$ is small enough and at least one of the terms in the
expression for $I$ is nonzero. Since $I$ is an upper bound for the sum
of the gradings at the negative end, it follows that the moduli space
containing the holomorphic disk $u$ is not counted in the differential.
We conclude that both $A(\Lambda)$ and $A_{\text{1-h}}(\Lambda)$ are subalgebras.

Finally, it is easy to see that the area of any disk which contributes to the differential on $A(\Lambda)$ is $\Ordo(\epsilon^{2})$. A straightforward monotonicity argument then shows that no such disk can leave $\R\times\ol{Y}_k(\delta)$ and it follows that the differential on $A(\Lambda)$ agrees with that induced by considering $\ol{Y}_k(\delta)$ as a subset of $Y_k(\delta)$.
\end{pf}

\begin{rmk}
Note that the argument in Lemma \ref{lma:subalgebra} also shows that the chords in one 1-handle generate a subalgebra on their own. By monotonicity, a disk with positive puncture in one 1-handle and a negative puncture in another has area at least $c_0\epsilon^{2}$ for some $c_0>0$. The argument now shows that such a disk cannot be rigid for grading reasons.
\end{rmk}

\begin{rmk}\label{rmk:actionleveliso}
As the exterior chords of $\Lambda(\epsilon)$ have action bounded below by $K\epsilon^{-\frac{1}{p}}$, we find that the inclusion of $A(\Lambda)$ into $\A_{\mathrm{H}}(\tilde{Y}_{k}(\epsilon;\delta),\Lambda(\epsilon))$ gives a canonical isomorphism on the parts of the DGAs below this action bound. Note that the natural action on $A(\Lambda)$ is scaled by $\epsilon^{2}$ under the inclusion so that the map is a canonical isomorphism below  action $K'\epsilon^{-(2+\frac{1}{p})}$.
\end{rmk}

Remark \ref{rmk:actionleveliso} shows that Lemma \ref{lma:subalgebra} is a rather strong result. However, since Legendrian homology algebras are defined through action filtrations with respect to fixed generic contact forms, it does not a priori give information above the action level.
As we shall see, it in fact does and the homology of $A(\Lambda)$ equals the homology of
$\A_{\mathrm{H}}(\tilde Y_k(\epsilon;\delta),\Lambda(\epsilon))$. For that reason we call $A(\Lambda)\subset \A_{\mathrm{H}}(\tilde Y_k(\epsilon;\delta),\Lambda(\epsilon))$ the {\em core algebra} of $\Lambda$.

\subsubsection{Cobordism maps}
Consider the cobordism $W(h)$ with positive end $\tilde Y(\epsilon_+;\delta)$ and negative end $\tilde Y(\epsilon_-;\delta)$, see Section \ref{sec:excob}. Such a cobordism induces a chain map of DGAs:
\[
\Phi_h\colon \A_{\mathrm{H}}(\tilde Y(\epsilon_+;\delta),\Lambda(\epsilon_+))\to
\A_{\mathrm{H}}(\tilde Y(\epsilon_-;\delta),\Lambda(\epsilon_-)).
\]
Let $W(h')$ be a cobordism with positive end $\tilde Y(\epsilon_-;\delta)$ and negative end $\tilde Y(\epsilon_+;\delta)$. Joining $W(h)$ to $W(h')$, we get a cobordism connecting $\tilde Y(\epsilon_+;\delta)$ to itself which can be deformed to a symplectization. As mentioned above (see Section \ref{app:aninv} for more detail) such a deformation induces a chain homotopy. In particular,
\[
\Phi_h\circ \Phi_{h'} = \id
\]
on the homology of $\A_{\mathrm{H}}(\tilde Y(\epsilon_+,\delta),\Lambda(\epsilon_+))$. Similarly we find that
\[
\Phi_{h''}\circ\Phi_{h} =\id
\]
on $\A_{\mathrm{H}}(\tilde Y(\epsilon_-,\delta),\Lambda(\epsilon_-))$ for suitable $h''$ and we conclude that $\Phi_h$ is an isomorphism on homology.

\begin{lma}\label{lma:idonA}
For any $\epsilon>0$, there exists $\delta_0$ and $\frac12<\theta<1$, independent of $\epsilon$, such that if $\epsilon_+=\epsilon$ and $\epsilon_-=\theta\epsilon$ then for all $\delta<\delta_0$ there exists a positive function $h\colon \R\to\R$ which satisfies \eqref{eq:hinftycond} and \eqref{eq:hC1cond} such that $\Phi_h(A(\Lambda))=A(\Lambda)$ and $\Phi_h|_{A(\Lambda)}=\id$.
\end{lma}

\begin{pf}
Repeating the proof of Lemma \ref{lma:subalgebra} we find first that $\Phi_h(A(\Lambda))\subset A(\Lambda)$ and that $\Phi_h(A_{\text{1-h}}(\Lambda))\subset A_{\text{1-h}}(\Lambda)$, and then that any disk contributing to $\Phi_h$ is contained inside $\R\times\ol{Y}(\epsilon;\delta)$. Finally, note that the $C^{2}$-norm of $h/\epsilon$ controls the distance from the cobordism $(W_k(h),L(h))$ to the trivial cobordism. Since the map induced by the trivial cobordism is the identity and since there is a uniform action bound on diagram chords it follows that for $\theta$ sufficiently close to $1$ the map equals the identity when acting on diagram chords.

Furthermore, the map $\Phi_h$ takes handle chords to sums of monomials of handle chords and it is straightforward to check that if $c$ is any handle chord, if $b_1\cdots b_k$ is a monomial of handle chords, if $\theta$ is sufficiently close to $1$, and if ${\mathfrak a}(c)\ge \sum_{j=1}^{k}{\mathfrak a}(b_j)$, then $|c|>\sum_{j=1}^{k}|b_j|$ unless $k=1$ and $b_1=c$. This implies that there are no disks of formal dimension $-1$ in the handle. It follows that for any fixed handle chord $c$ the moduli space of disks with positive puncture and negative puncture at $c$ is cobordant to the corresponding moduli space defined by the trivial cobordism and that other moduli spaces that could contribute to $\Phi_h(c)$ are empty (see also Lemma \ref{lem:noanchor2}). We conclude that $\Phi_h$ is the identity also when acting on handle chords and consequently $\Phi_h|_{A(\Lambda)}=\id$.
\end{pf}

\subsubsection{Legendrian homology and the homology of $A(\Lambda)$}
Lemma \ref{lma:subalgebra} implies that the inclusion map
\begin{equation}\label{eq:inclA}
\iota\colon A(\Lambda)\to \A_{\mathrm{H}}(\tilde Y_k(\epsilon;\delta),\Lambda(\epsilon))
\end{equation}
is a chain map for $\epsilon,\delta>0$ small enough. Furthermore, it implies that the homology of $A(\Lambda)$ is canonically isomorphic to that of  $\A_{\mathrm{H}}(Y_k(\delta);\Lambda)$.

\begin{lma}
For $\epsilon,\delta>0$ sufficiently small,
the inclusion map $\iota$ in \eqref{eq:inclA} induces an isomorphism in homology. In particular, it follows that the corresponding DGAs are quasi-isomorphic:
\[
\A_{\mathrm{H}}(\tilde Y_k(\epsilon;\delta),\Lambda(\epsilon)) \ \cong_{\mathrm{quasi}} \
\A_{\mathrm{H}}(Y_k(\delta),\Lambda).
\]
\end{lma}

\begin{pf}
For simpler notation we write
$H\A_{\mathrm{H}}$ for the homology of $\A_{\mathrm{H}}(\tilde Y_k(\epsilon;\delta),\Lambda(\epsilon))$,
and $H A$ for the homology of $A(\Lambda)$.

To see that $\iota$ is injective on homology, consider a sum of monomials $w\in A(\Lambda)$ which represents $0\in H \A_{\mathrm{H}}$. This means that there exists a sum of monomials $\gamma$ of chords of $\Lambda$ such that $d\gamma=w$. Consider now a concatenation of $N$ cobordisms $W(h_j)$, $j=1,\dots, N$ which all satisfy Lemma \ref{lma:idonA} and such that the contact manifold at the negative end $\tilde Y_k(\theta^{N}\epsilon)$, $0<\theta<1$, has the property that the action of any exterior chord which is bounded below by $K\theta^{-N/p}\epsilon^{-1/p}$ is larger than any chord appearing in a monomial in $\gamma$. Then, writing
\[
\Phi=\Phi_{h_N}\circ\dots\circ\Phi_{h_1},
\]
we must have $\Phi(\gamma)\in A(\Lambda)$ and by Lemma \ref{lma:idonA}
\[
d\Phi(\gamma)=\Phi(d\gamma)=\Phi(w)=w,
\]
and we find that $w$ represents $0\in H A$ as well and $\iota$ is injective on homology.

To see that $\iota$ is surjective on homology, for any class in $H A_{\mathrm{H}}$ represented by $\gamma$ we find as above that there is $N$ such that
$\Phi=\Phi_{h_N}\circ\dots\circ\Phi_{h_1}$ satisfies
$\Phi(\gamma)\in A(\Lambda)$. Since $\Phi$ is the identity on $A(\Lambda)$, there exist $w\in A(\Lambda)$ such that $\Phi(\gamma-w)=0$. But $\Phi$ is an isomorphism on homology and thus the class represented by $\gamma$ is also represented by the cycle $w\in A(\Lambda)$. It follows that $\iota$ is surjective on homology as well and thus a quasi-isomorphism.
\end{pf}

%*********************************************************************
\subsection{Explicit descriptions of rigid holomorphic disks}\label{Sec:compute}
In this section we describe all moduli spaces necessary for computing the differential on $A(\Lambda)$, after a slight isotopy, giving $\Lambda$ what we call a split diagram, that we describe first.

\subsubsection{Split diagrams}\label{sec:split}
Consider the front resolution of a Legendrian link in $Y_k$ as described in Section \ref{ssec:resolution}. In order to facilitate the proof that the combinatorial formula in Section \ref{ssec:dga} indeed gives a description of the Legendrian algebra, we will make additional modifications to the front resolution by adding additional ``dips'' in a neighborhood of each attaching ball. (Note that these dips do not change the DGA up to stable tame isomorphism by the analytic or combinatorial invariance proofs, see Appendices~\ref{app:aninv} and~\ref{app:combinv}.)

Consider the part of the Lagrangian diagram of a Legendrian link in a small region around the attaching $S^0$ of a one-handle, see Figure \ref{fig:stdform}. Recall that we identified a neighborhood of the core-sphere of the handle with a $\delta_0$-neighborhood of the $0$-section in a $1$-jet space of an interval and that all strands passing through the handle lie in an $\epsilon$-neighborhood of the $0$-section, $\epsilon\ll\delta_0$.

We change the link by a Legendrian isotopy which we will call {\em an
  attaching isotopy} since it is supported near the attaching
regions. This isotopy changes the Lagrangian diagram only near the
attaching locus where it appears as in Figure \ref{fig:lagdip}.
Note that the attaching isotopy induces a set of new Reeb chords $\Nn$ of the link.

\begin{figure}
\centering
\includegraphics[width=.9\linewidth]{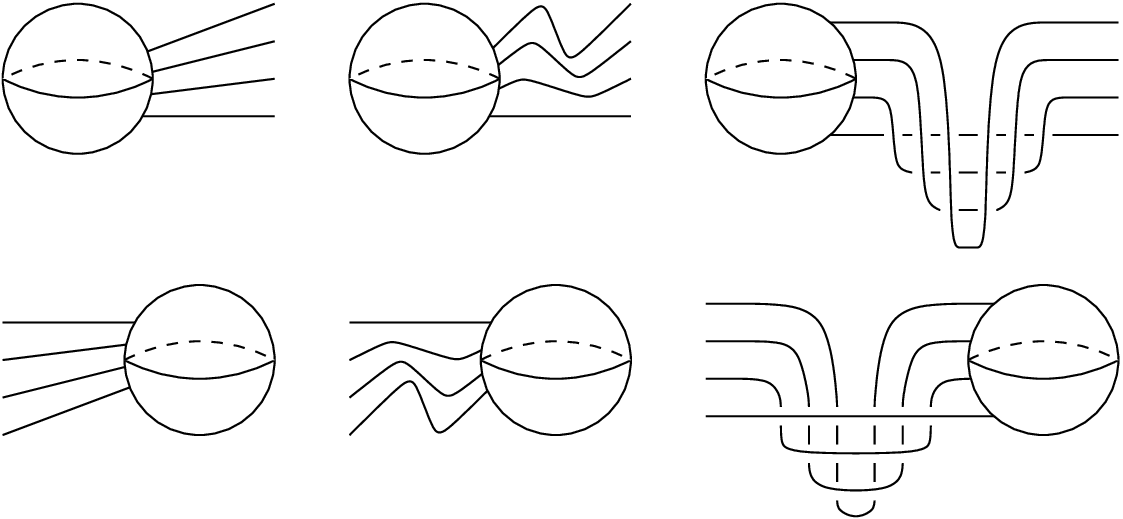}
\caption{Perturbing the $xz$ projection of a Legendrian link near the
  attaching locus (left) to obtain the $xz$ projection of a link
  (middle) whose $xy$
  projection (right) contains a dip. This produces new Reeb chords
  near the attaching locus.}
\label{fig:lagdip}
\end{figure}

\begin{lma}
There exists an isotopy such that for any Reeb chord $c\in\Nn$, the action $\act(c)$ of $c$ satisfies $\act(c)<2\delta_1$.
\end{lma}

\begin{pf}
Figure \ref{fig:frontdip} shows the front of the link before and after the isotopy. The lemma follows.
\end{pf}

\begin{figure}
\centering
\includegraphics[width=.7\linewidth]{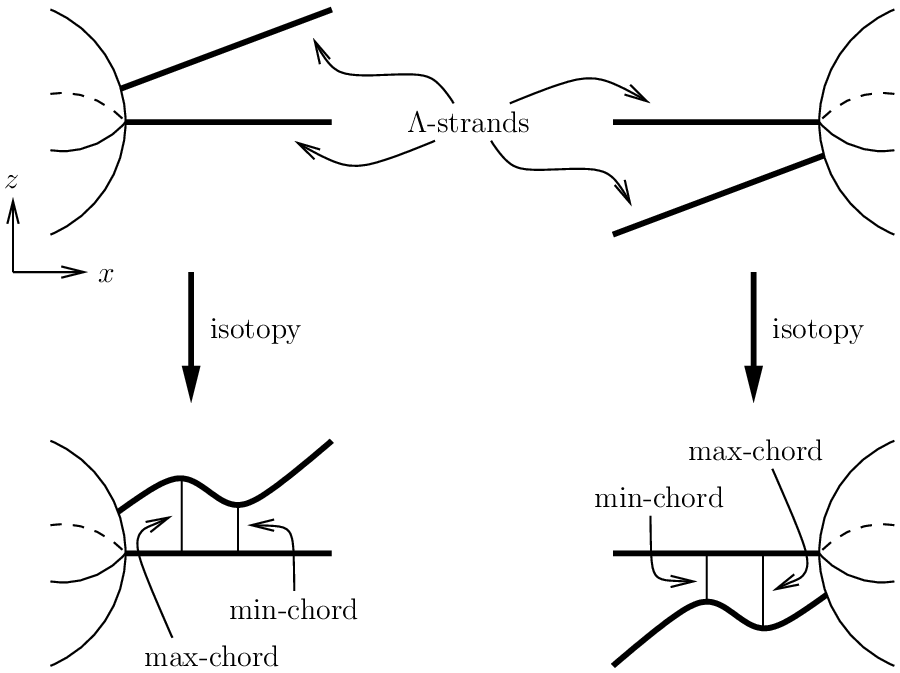}
\caption{Short chords introduced by the isotopy. In each diagram, the
  height of the vertical direction is less than $\epsilon$.}
\label{fig:frontdip}
\end{figure}

We call a Legendrian link obtained by front resolution followed by an attaching isotopy introducing short chords a Legendrian link with {\em split diagram}. We show below that the holomorphic disks which need to be considered for calculation of the differential of the Legendrian homology a Legendrian link with split diagram are of three types: diagram disks, handle disks, and interpolating disks.

\subsubsection{Diagram disks}
Let $\Lambda$ be a Legendrian link with split diagram. Let $R(r_1)\subset R(r_2)\subset \R^{2}$, where we think of $\R^{2}$ as the $xy$-plane, be regions such that $R(r_2)$ contains all of the diagram of $\Lambda$ except standard strands entering the attaching locus and such that $R(r_2)-R(r_1)$ contains all the diagram double points corresponding to the short maximum Reeb chords near the attaching locus, see Figure \ref{fig:frontdip}.

\begin{lma}\label{lma:diagramdisks}
Let $u\colon D\to\R\times Y_k(\delta)$ be a holomorphic disk with one positive puncture and boundary on $\R\times\Lambda$. Assume that the Reeb chord of the positive puncture is contained in $R(r_1)\times\R\subset Y_k(\delta)$. Then $u(D)$ is contained in $\R\times(R(r_1)\times\R)$.
\end{lma}

\begin{pf}
Recall that we use the $\R$-invariant almost complex structure pulled back from $\C$ in $\R\times(R(r_1)\times\R)$. Let $E=u^{-1} (\R\times (R(r_2)\times\R))$. Then for generic $r_2$, $E$ is a Riemann surface with boundary $\pa E$ mapping to
\[
(\R\times \Lambda)\,\,\cup\,\, (\R\times(\pa R(r_2)\times\R)),
\]
and if $\pi\colon \R\times(R(r_2)\times\R)\to R(r_2)$ is the projection then $\pi\circ u$ is holomorphic with respect to the standard complex structure on $\C$. If $u$ does not map any boundary component into $\R\times(R(r_2)-R(r_1))\times\R$ then it is an easy consequence of the maximum principle that $u(D)\subset R(r_1)$: if not $u|_E$ would not have bounded $y$-coordinate and its area would be infinite.

Assume thus that some boundary component of $D$ maps into this region. Consider the outermost strand before the dip. As above, the maximum principle implies that the disk cannot lie on the outside of this strand. Hence it lies on the inside. If the disk has no more corners then it eventually ends up on the outside and the maximum principle again is violated. It follows that the disk stays in $\R\times (R(r_1)\times\R)$.
\end{pf}

We call disks in Lemma \ref{lma:diagramdisks} {\em diagram disks}.

\subsubsection{Handle disks}\label{ssec:handledisks}
Consider a standard Legendrian strand $\Lambda$ running through a one-handle, see Section \ref{sec:1handle}. Recall that the Reeb chords of $\Lambda$ are then $c^{m}$, $m=1,2,\dots$ where $m$ is the number of times the chord wraps around the central Reeb orbit and that although all chords are orbits, they are nondegenerate in the sense that the linearized Reeb flow takes the tangent space of $\Lambda$ at the start point to a Lagrangian subspace in the contact plane in the tangent space at the endpoint which is transverse to the tangent space of $\Lambda$ at the end point. Consider a small perturbation of the standard strand so that no Reeb chord is also an orbit. Then the Reeb chords are in natural one-to-one correspondence with the Reeb chords of the unperturbed strand and we then have the following.

\begin{lma}\label{lma:constterm}
The moduli space $\M^{\R\times Y_{k}(\delta)}(c^{1};\varnothing)$ of once punctured holomorphic disks with boundary on $\Lambda\times\R$ contains algebraically one $\R$-family (i.e., one disk, rigid up to translation).
\end{lma}

\begin{pf}
Consider the Weinstein domain obtained by attaching one $1$-handle to the ball. Let $\Lambda$ denote the standard Legendrian knot which runs through the handle once. Then the Weinstein manifold which results from attaching a $2$-handle along $\Lambda$ is the $4$-ball. According to \cite[Corollary 5.7]{bib:BEE} there is a quasi-isomorphism between the symplectic homology of the ball and the homology of a complex $A^{\rm Ho}(\Lambda)$ associated to $A(\Lambda)$ that has two generators for each monomial $w$ in $A(\Lambda)$, $\hat w$ and $\check w$:
\[
SH(B^{4})\to A^{\rm Ho}(\Lambda).
\]
Now the complex on the left hand side is acyclic and the complex of the right hand side has no generators in negative degrees, one generator $\tau$ in degree $0$, and one generator $\check{c}^{1}$ in degree $1$. Since the homology must vanish we conclude that
\[
d_{A^{\rm Ho}}\check{c}^{1}=\tau.
\]
On the other hand $\la d_{A^{\rm Ho}}\check{c}^{1},\tau\ra$ can be computed in terms of moduli spaces of holomorphic disks:
\[
\la d_{A^{\rm Ho}}\check{c}^{1},\tau\ra=|\M^{Y_{k}(\delta)}(c^{1};\varnothing)|,
\]
where the right hand side denotes the number of $\R$-components in the moduli space. The lemma follows.
\end{pf}

\begin{rmk}
Alternatively, one can prove Lemma \ref{lma:constterm} using Remark \ref{rmk:loose} as follows. Since the closed up standard strand is isotopic to a stabilization it is straightforward to show that its DGA must be trivial. This implies that $1$ is a boundary which gives the desired result as explained above.
\end{rmk}

We next construct a somewhat nongeneric situation where many strands pass through the handle but where all the relevant moduli spaces admit concrete descriptions. Fix a $1$-jet neighborhood of $\Lambda_{\rm st}$ and fix a very small $\epsilon>0$. Consider $n$ strands $\Lambda_1',\dots,\Lambda_n'$ parallel to $\Lambda_{\rm st}$. More precisely, let
\[
\Lambda_{j}'=\Phi^{j\epsilon}_{R}(\Lambda_{\rm st}),
\]
where $\Phi_{R}^{t}$ is the time $t$ Reeb flow. Then $\Lambda_j'$ is Legendrian for $j=1,\dots,n$.
\begin{lma}
The Reeb chords connecting $\Lambda_j'$ to $\Lambda_k'$ are as follows:
\begin{itemize}
\item[{\rm (a)}] If $j<k$ then: there is one Reeb chord of length $(k-j)\epsilon$ starting at any point on $\Lambda_j'$, for each integer $w>0$ there is a Reeb chord along the unique Reeb orbit in the handle of length $2\pi w\delta+(k-j)\epsilon$, and there are no other Reeb chords.
\item[{\rm (b)}] If $j\geq k$  then: for each integer $w>0$ there is a Reeb chord along the unique Reeb orbit in the handle of length $2\pi w\delta-(j-k)\epsilon$, and there are no other Reeb chords.
\end{itemize}
\end{lma}
\begin{pf}
Consider ${\rm (a)}$. Note that any Reeb chord starting on $\Lambda_j'$ and ending on $\Lambda_k'$ gives rise to a Reeb chord starting and ending on $\Lambda_k'$ by subtracting one of the short Reeb chords mentioned. Similarly, in case ${\rm (b)}$, any Reeb chord stating on $\Lambda_j$ and ending on $\Lambda_k'$ gives rise to a Reeb chord starting and ending on $\Lambda_{j}'$ by adding one of the short chords.

Since $\Lambda_j'=\Phi^{j\epsilon}_R(\Lambda_{\rm st})$ the lemma follows from Lemma \ref{lem:handleReebtv}.
\end{pf}

In order to get to a situation where the moduli spaces can be described we perturb $\Lambda_j'$ further as follows. In the $1$-jet neighborhood $\Lambda_j'$ corresponds to the graph of the constant function $f(q)=j\epsilon$. We denote the perturbation of $\Lambda_j'$ by $\Lambda_j$ and define it as the $1$-jet graph of the function $f(q)=j\epsilon + j\kappa q^{2}$, where $\kappa\ll\epsilon$ and where $q$ is a coordinate along the standard strand. (Note that $q=0$ corresponds to the central Reeb orbit.)

\begin{lma}
The Reeb chords connecting $\Lambda_j$ to $\Lambda_k$ are as follows:
\begin{itemize}
\item[{\rm (a)}] If $j<k$ then for each integer $w\geq 0$ there is a Reeb chord along the unique Reeb orbit in the handle of length $2\pi w\delta+(k-j)\epsilon$, and there are no other Reeb chords.
\item[{\rm (b)}] If $j\geq k$ then for each integer $w>0$ there is a Reeb chord along the unique Reeb orbit in the handle of length $2\pi w\delta-(j-k)\epsilon$, and there are no other Reeb chords.
\end{itemize}
\end{lma}

\begin{pf}
To control short Reeb chords one may use the $1$-jet neighborhood. Uniqueness of short chords is then an immediate consequence of the fact that functions of the form $f(q)=c_0+c_2 q^{2}$, $c_0,c_2$ constants, have only one critical point at the origin. For the long chords simply note that no strand was perturbed at its intersection with the central Reeb orbit and that for all long chords, the linearized Reeb flow mapped the tangent line at the initial point to a tangent line transverse to the tangent line in the contact plane at the final point already for $\Lambda_0',\dots,\Lambda_n'$.
\end{pf}

We will use the following notation for Reeb chords between the strands $\Lambda_1,\dots,\Lambda_n$: let $c_{ij}^{p}$ be the Reeb chord connecting $\Lambda_{j}$ to $\Lambda_{i}$ of length closest to $2\pi\delta p$. Using the definition of Reeb chord grading from \cite[Section 2.1]{bib:BEE}, it is a straightforward consequence of Lemma \ref{lma:orbitsinhandle} that the grading of $c^{p}_{ij}$ is
\[
|c^{p}_{ij}|=2p-1 + m(i)-m(j),
\]
where $m(i)$ denotes the integer valued Maslov potential of the strand $\Lambda_{i}$, see Section \ref{ssec:intdga}.

\begin{lma}\label{lma:handlediff}
For sufficiently small perturbations $\Lambda_{1},\dots,\Lambda_n$ of $\Lambda_{\rm st}$ as described above the following holds.  If $i,j,l\in\{1,\dots,n\}$ and if $m,p,q\ge 0$ with $m=p+q$, then the moduli space $\M(c^{m}_{ij};c^{p}_{lj},c^{q}_{il})$ is transversely cut out and  $C^{1}$-diffeomorphic to $\R$. Furthermore, moduli spaces of holomorphic disks with all punctures inside the handle and with more than two negative punctures have formal dimension $>1$.
\end{lma}

\begin{pf}
To see the last statement note that if the positive puncture maps to $c^{p}_{ij}$ and the negative punctures map to chords $c^{q_l}_{ij}$ then $p\ge \sum q_{l}$ for action reasons and the grading formula shows that the dimension of the moduli space is $>1$ if there are more than two negative punctures.

Now consider the $1$-parameter family (including $\R$-translation) of disks shown in Figure \ref{fig:handledisk1}. We claim first that it is transversely cut out. To this end we use the trivialization $x_2+iy_2$ for the normal bundle of the disk. We trivialize the bundle over the Reeb chords which leads to a Lagrangian boundary condition given by small rotations which when closed up by negative rotations at corners has Maslov index equal to $-1$. The operator is Fredholm of index $0$ with values in a $1$-dimensional bundle. It follows that it is an isomorphism.

We next show that this is the only disk (up to $\R$-translation) in the moduli space. To see this we note that if there were another disk in the moduli space then it could not be contained in the Reeb orbit cylinder over $c^{m}$. This however contradicts its action being $0$ and we conclude there is exactly one disk.
\end{pf}

\begin{figure}
\centering
\includegraphics[width=.4\linewidth]{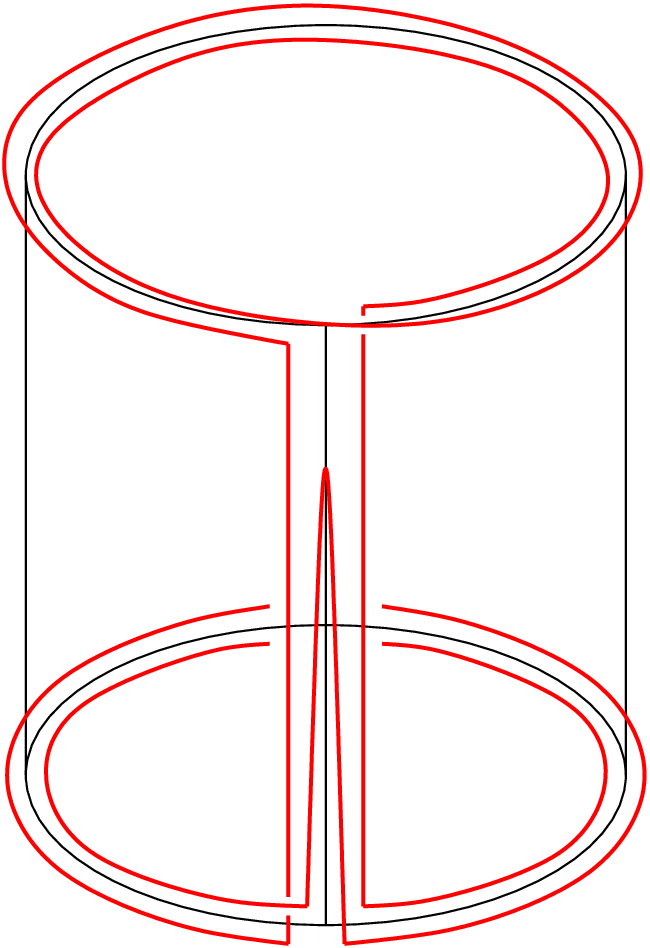}
\caption{A rigid disk in a handle covering a Reeb orbit cylinder.}
\label{fig:handledisk1}
\end{figure}

By Lemma \ref{lma:subalgebra}, disks that contribute to the differential and that start inside the handle cannot have negative punctures outside the handle.

\subsubsection{Interpolating curves}\label{sec:interpolatingdisks}
Finally, we consider curves that start outside the handle and go inside. An action argument shows that they cannot leave a small neighborhood of the standard strand. It follows that they lie entirely in the $1$-jet neighborhood and hence are given by a count of holomorphic polygons on the diagram in the cotangent bundle in a neighborhood of the $0$-section, just like in the $\R^{3}$-case identifying the $0$-section with the $x$-axis and the cotangent fibers with lines in $y$-direction.

%*********************************************************************
\subsection{Orientations and signs}\label{sec:orientaton}
In this section we present an orientation scheme and compute the corresponding signs of the rigid disks described in Section \ref{Sec:compute}. The material presented here is a  generalization of \cite[Section 4.5]{EES-orientation} and we will refer there for many details.

\subsubsection{A trivialization of the bundle of contact planes in $Y_{k}(\delta)$}
Consider first the standard contact $\R^{3}$ with contact form $\alpha_{\rm st}=dz-y\,dx$. As above we use the complex structure on the contact planes induced by the standard complex structure on $T^{\ast}\R=\C$ via the Lagrangian projection $\Pi\colon \R^{3}=J^{1}(\R)\to T^{\ast}\R$. Furthermore we trivialize the contact plane field using the frame $\pa_{y}, J\pa_{y}$ where $J$ is the complex structure on the contact plane field. We next transport this trivializing vector field to the standard contact balls. Recall the contactomorphism in \eqref{eq:balltospace}. Its inverse is
\[
F^{-1}(x,y,z)=(x-x_0\,,\,y-y_0\,,\,z-z_0-\tfrac12(x-x_0)(y+y_0)).
\]
In particular, the vector field corresponding to $\pa_y$ is then
\begin{equation}\label{eq:trivball}
d F^{-1}\,\pa_y=\pa_v-\tfrac12u\,\pa_z.
\end{equation}

We next study the behavior of trivializations in the handles of $Y_k(\delta)$. Recall the trivialization of the contact plane field along $V_{-\delta}$ given by $v,i\cdot v$, where $i$ is the complex structure of $\C^{2}$ and where $v$ is given by \eqref{eq:contacttrivhandle}. In our construction of $Y_{k}(\delta)$ we attached the handles using standard contact balls identified with balls in $V_{-\delta}$. The identifications were constructed using the map of the disks in \eqref{eq:attachhandle1} and \eqref{eq:attachhandle2}. In order to determine the vector field in the standard ball that corresponds to $v$, we first express $v$ as
\[
v=w_0+w_R,
\]
where $w\in TA_\rho$ and where $w_R$ is parallel to the Reeb vector field. Using notation as in \eqref{eq:unnorReebhandle}, we find
\[
w_0=y_2\pa_{y_1}+\tfrac12 y_1\pa_{y_2},\qquad w_R=-\tfrac{1}{2}x_1y_2^{-1}\,\tilde R.
\]
The vector field $v_{\rm b}$ in the standard contact ball that corresponds to $v$ is thus the sum of the image of $w_0$ under $dG$, see \eqref{eq:attachhandle2}, and a suitable multiple of the Reeb field. A straightforward computation gives
\[
v_{\rm b}=-\tfrac12u\left(\delta^{2}+\tfrac14(u^{2}+v^{2})\right)\pa_{z}
+\left(\delta^{2}+\tfrac34 u^{2}+\tfrac12 v^{2}\right)\pa_{v}+\tfrac12 uv\,\pa_{u}.
\]
Continuously killing all quadratic terms, we find a homotopy to the trivialization given by
\[
-\tfrac12 u\,\pa_{z}+\pa_{v},
\]
which in turn corresponds to the vector field $\pa_{y}$ in $1$-jet coordinates. In conclusion, we have continued the trivialization of the contact plane field given by $(\pa_{y}, J\pa_{y})$ over $\R^{3}(\circ_k)$ to all of $Y_{k}(\delta)$ in such a way that inside the handles this trivialization is given by $(v,i\cdot v)$.

\subsubsection{Capping operators and the symplectization direction}
In \cite{EES-orientation} the contact manifolds under consideration are of the form $P\times\R$, where $P$ is an exact symplectic manifold, with contact form $dz-\beta$, where $\beta$ is the primitive of the symplectic form on $P$. (In particular, taking $P=T^{\ast}\R$ and $\beta=y\,dx$, we get standard contact $\R^{3}$ as a special case.) If we use an almost complex structure on the symplectization of such a manifold given by the product of the standard complex structure on $\C\approx\R\times\R$ and the pullback to the contact planes of an almost complex structure on $P$, holomorphic curves in the symplectization can be described entirely in terms of holomorphic curves in $P$, and it was moduli spaces of holomorphic curves in $P$ that were oriented in \cite{EES-orientation}. In the present setup we do not have any projection and hence we must study holomorphic curves in the symplectization. As a starting point we will revisit the construction in \cite{EES-orientation}, in a sense adding the symplectization direction back.

\subsubsection{A brief description of the orientation scheme}
We give the description as close as possible to the application we have in mind. The starting point is a number of properties of $\bar\pa$-operators with Lagrangian boundary conditions on the disk that we discuss next. These properties originate from \cite{bib:FOOO}; we will however refer to \cite{EES-orientation} since the treatment there is closer to the case under study here. Let $D$ denote the unit disk in $\C$ and let $L\colon \pa D\to\Lag(\C^{2})$ denote a smooth map where $\Lag(\C^{2})$ denotes the space of Lagrangian planes in $\C^{2}$. Consider the Riemann--Hilbert problem with trivialized Lagrangian boundary condition: let
$W_2(L)$ denote the space of functions $u\colon D\to\C^{2}$ with two derivatives in $L^{2}$ such that
\[
u(e^{i\theta})\in L(e^{i\theta}),\quad\quad L\bar\pa_{J} u|_{\pa D}=0,
\]
and let $W_1$ denote the space of $J$-complex-antilinear maps $TD\to T\C^{2}$ with one derivative in $L^{2}$ that vanish when restricted to the boundary, or equivalently, after fixing a trivialization of $TD$, the corresponding space of maps $D\to\C^{2}$. Then the linearized operator $L\bar\pa_{J}\colon W_2(L)\to W_1$ is a Fredholm operator of index
\[
2+\mu(L),
\]
where $\mu(L)$ denotes the Maslov index of $\Lambda$. If $Y'$ denotes the space of Fredholm operators as just described, then the determinant bundle over $Y'$ with fiber over $L\in Y$ given by \[
\Lambda^{\rm max}\left(\krn(L\bar\pa_{J})\right)\otimes\Lambda^{\rm max}\left(\cokrn(L\bar\pa_{J})\right)
\]
is nonorientable and remains nonorientable if $Y'$ is replaced by the space of oriented Lagrangian boundary conditions. However, adding the requirement that the bundle of Lagrangian subspaces over $\pa D$ is trivialized leads to an orientable bundle, see \cite[Lemma 3.8]{EES-orientation}. Note that the space of trivialized Lagrangian boundary conditions is homotopy equivalent to $\Omega(U(2))$, the loop space of the unitary group of $\C^{2}$: if $A\colon\pa D\to U(2)$ then $A\cdot \R^{2}$ is a trivialized Lagrangian boundary condition and an orthonormal frame in a Lagrangian plane gives an element in $U(2)$. Furthermore, as explained in \cite[Lemma 3.8]{EES-orientation}, orientations on $\R^{2}\subset \C^{2}$ and on $\C$ induce an orientation on the determinant bundle over $Y$. We use this orientation of the determinant bundle over $Y$ to orient all moduli spaces.

Let $\Lambda\subset Y_{k}(\delta)$ be a Legendrian link and consider a moduli space\linebreak $\M^{\R\times Y_{k}(\delta)}(a;b_1,\dots,b_k)$. In order to orient it we first fix a trivialization $\tau$ of the tangent bundle of $\Lambda$. Then $(\pa_t, \tau)$ gives a trivialization of the tangent bundle of $\R\times\Lambda\subset \R\times Y_k(\delta)$, where $\pa_{t}$ is the vector field along the $\R$-factor of $\R\times Y_{k}(\delta)$. Note that if $u\colon D_{k+1}\to \R\times Y_{k}(\delta)$ is a holomorphic map of a $(k+1)$-punctured disk with boundary on $\R\times\Lambda$, then $u^{\ast}T(\R\times Y_{k}(\delta))$ splits as $u^{\ast}(\krn(\alpha))\oplus\C$, and the Lagrangian boundary condition splits as $u^{\ast}(T\Lambda)\oplus\R$. Second we fix positive and negative oriented capping operators at all Reeb chords. These are $1$-punctured disks with trivialized Lagrangian boundary conditions, which when glued to the punctures give a disk with trivialized boundary condition.

In \cite[Section 3.2]{EES-orientation}, (linear) gluing analysis for joining operators as just described is carried out. In particular, if $k+2$ operators are glued into an operator then orientations on the determinants of all operators except one induce an orientation on the last one. (The actual rule for how to induce such an orientation depends on some choices that are discussed in detail in \cite{EES-orientation}.) In particular, capping operators allow us to orient the kernel/cokernel of the linearized operator at $u$, which together with a chosen fixed orientation on the space of automorphisms/conformal structure on $D_{k+1}$ gives an orientation of the moduli space.

\subsubsection{Signs of diagram disks}
The main objective of this section is to reduce the computation of the sign of a diagram disk to the computation of similar disks carried out in \cite[Section 4.5]{EES-orientation}. Consider the four quadrants at a crossing of the link diagram as in the right diagram in Figure \ref{fig:signs}. In the terminology of \cite{EES-orientation}, quadrants where the sign is $-1$ are called $A$-shaded. We then have the following:
\begin{lma}\label{lma:signdiagramdisk}
Let $T\Lambda$ be equipped with the Lie group spin structure. Then there is a choice of orientation of $\C$ and $\R$ such that the sign of a rigid diagram disk in $\M(a;b_1,\dots,b_k)$ is given by $(-1)^{s(u)}$ where $s(u)$ is the number of $A$-shaded corners of $u$.
\end{lma}
We prove Lemma \ref{lma:signdiagramdisk} after discussing capping operators. Here we use the observation above that the boundary condition of the linearized operator splits and define split capping operators. For the component of the boundary condition at a Reeb chord in the contact plane we use exactly the same capping operators as in \cite[Section 4.5]{EES-orientation}. (As there, this requires stabilization in order to have room to rotate the trivialization so that it satisfies compatibility conditions at Reeb chord endpoints.) In the symplectization direction, the tangent space to the configuration space of maps around a solution is a Sobolev space with small positive weight near each puncture, augmented by a $1$-dimensional space of explicit solutions near each puncture corresponding to translations in the $\R$-direction. The kernel and cokernel of the corresponding linearized problem are canonically identified with the kernel and cokernel of the same (linearized) operator acting on the Sobolev space with small negative exponential weight and without auxiliary space of solutions. We thus use as capping operators in the symplectization directions the once-punctured disk with constant boundary condition $\R$ and with small negative exponential weight. This operator has index $1$ and has $1$-dimensional kernel. The gluing sequence for the capping disks with positive (respectively, negative) punctures, which glue to the constant boundary condition on the closed disk, again of index $1$ with $1$-dimensional kernel, is:
\[
\begin{CD}
0 @>>> \ker(\bar\pa_{D}) @>>> \ker(\bar\pa_{D_+})\oplus\ker(\bar\pa_{D_-}) @>>> \R @>>> 0.
\end{CD}
\]
Here $\bar\pa_{D_{+}}$ and $\bar\pa_{D_-}$ denote the capping operators in the $\R$-direction with positive (respectively, negative) puncture, $\bar\pa_D$ denotes the operator on the closed disk with constant $\R$ boundary condition, and the last $\R$ is identified with the cokernel of the $\bar\pa$-operator on the strip with $\R$-boundary conditions and positive exponential weights at both ends. (This cokernel consists of constant functions with values in $i\R$ and the operator is the operator on the middle part of the glued disk as the gluing parameter goes to $\infty$.) The maps in the sequence are given by the matrices
\[
\left(
\begin{matrix}
1\\
1
\end{matrix}
\right)
\quad\text{ and }\quad
\left(
\begin{matrix}
-1 & 1
\end{matrix}
\right).
\]
Consider next the gluing sequence for the $\C$-component of a linearized boundary condition. Note that the result of gluing capping operators to the linearized boundary condition is again a closed disk with constant boundary condition $\R$ and $1$-dimensional kernel, and the gluing sequence is
\[
\begin{CD}
0 @>>> \ker(\bar\pa_{D}) @>>> \ker(\bar\pa_{D_+})\oplus_{j=1}^{m}\ker(\bar\pa_{D_-;j})\oplus\ker(\bar\pa_{\C}) @>>>\\
{ }@>>>\oplus_{j=1}^{m} \R @>>> 0,
\end{CD}
\]
where the summands are gluing parameters near the punctures. Thus the orientation induced on $\ker(\bar\pa_{\C})$ is simply the orientation induced by the $\R$-factor.

\begin{pf}[Proof of Lemma \ref{lma:signdiagramdisk}]
By our choice of capping operators and the linear gluing that this gives rise to, as discussed above, it follows that the orientation of the reduced moduli space is determined entirely in terms of the operator in the contact planes which is pulled back from $\C$ for diagram disks. The lemma is then an immediate corollary of \cite[Theorem 4.32]{EES-orientation}.
\end{pf}
\subsubsection{Signs of interpolating disks}
After the signs of diagram disks have been computed, the sign of interpolating disks can be determined by a simple homotopy argument. More precisely, we have the following:
\begin{lma}\label{lma:signinterpoldisk}
Let $T\Lambda$ be equipped with the Lie group spin structure. Then there is a choice of orientation of $\C$ and $\R$ such that the sign of a rigid interpolating disk in $\M(a;b_1,\dots,b_k)$ is given by $(-1)^{s(u)}$ where $s(u)$ is the number of $A$-shaded corners of $u$. Here the diagram disk should be drawn in a $1$-jet neighborhood of $\Lambda_{\rm st}$.
\end{lma}

\begin{pf}
By monotonicity any rigid holomorphic connecting disk lies in an arbitrarily small neighborhood of $\Lambda_{\rm st}$ and its Lagrangian projection has convex corners. Furthermore, it can be homotoped through disks with convex corner to a small such disk in $\R^{2}$. The result then follows from Lemma \ref{lma:signdiagramdisk} once we choose capping operators as there for the short Reeb chords in the handle.
\end{pf}

\subsubsection{Signs of handle disks}
We now turn to determining the signs of rigid handle disks. In order to determine the signs of interpolating disks, we made choices for the orientations of the capping disks with positive puncture at the short Reeb chords in the middle of the handle. We choose here also the capping disks with negative puncture as for the diagram disks, and the orientation on the determinants of the corresponding operators are then fixed by the condition that these two orientations should glue to the canonical orientation. Note next that the linearized boundary condition of a rigid handle disk with positive puncture at a short Reeb chord in the middle of the handle can be deformed to that of a small rigid triangle lying in the $1$-jet neighborhood of $\Lambda_{\rm st}$. It follows from Lemma \ref{lma:signinterpoldisk} that the sign of any such rigid disk in $\M(c^{0}_{ij}; c^{0}_{il}, c^{0}_{lj})$ is
\[
(-1)^{|c^{0}_{il}|+1}.
\]

Before we complete the discussion of combinatorial signs we need to recall an orientation gluing result: with capping operators fixed and oriented in such a way that orientations on positive and corresponding negative caps glue to the canonical orientation on the determinant of the operator on the closed disk, the following holds. If $I\subset\hat\M(a;b_1,\dots,b_m)$ is a component of a reduced moduli space of dimension $1$ with nonempty boundary which is oriented as above, and if $u_0\# v_0$ and $u_1\# v_1$ are the broken disks at its two boundary points, then
\begin{equation}\label{eq:boundaryorient}
\sigma(u_1)\epsilon(u_1,v_1)\sigma(v_1)+\sigma(u_0)\epsilon(u_0,v_0)\sigma(v_0)=0,
\end{equation}
where $\sigma(u)$ is the sign of the rigid disk $u$ and where
\[
\epsilon(u,v)=(-1)^{\sum_{j=1}^{l}|b_j|},
\]
if $u\in\M(a;b_1,\dots,b_l,b_{l+1},\dots, b_m)$ and $v\in\M(b_{l+1};c_1,\dots,c_k)$. This follows from \cite[Lemma 4.11]{EES-orientation} in combination with the above discussion of capping operators in the symplectization direction. Using this property we can establish a combinatorial sign rule inductively.

\begin{lma}\label{lma:signhandledisk}
There is a choice of orientations for capping operators at the Reeb chords inside the handle for which the sign of the rigid disk in $\M(c^{p}_{ij};c^{q}_{il},c^{r}_{lj})$ equals
\[
(-1)^{|c^{q}_{il}|+1}.
\]
\end{lma}

\begin{pf}
After the above discussion of short chords, we may assume that the sign rule holds for all rigid triangles with positive puncture below a given action. Consider the next Reeb chord $a$ in the action ordering. We claim that if we choose the orientation on the capping operator of $a$ as a positive puncture so that the sign rule holds for some triangle with positive puncture at $a$, then it holds for all such triangles. To see this, we represent the link as $\Lambda_1\cup\dots\cup\Lambda_n$ inside the handle exactly as in Lemma \ref{lma:handlediff}. Then any rigid triangle lies inside the trivial holomorphic cylinder over the central Reeb chord and looks like a Reeb chord strip with a slit, see Figure \ref{fig:handledisk1}. For any two such triangles we get a transversely cut out reduced $1$-dimensional moduli space by considering the corresponding strip with two slits. Assume that this moduli space is $\M(a;b_1,b_2,b_3)$. Then its boundary points correspond to
\begin{align*}
(u_0,v_0)&\in\M(a;b_1,b_{23})\times\M(b_{23};b_2,b_3),\\
(u_1,v_1)&\in\M(a;b_{12},b_{3})\times\M(b_{12};b_1,b_2),
\end{align*}
and our inductive assumption implies that
\[
\sigma(v_0)=(-1)^{|b_2|+1}\quad\text{ and }\quad
\sigma(v_1)=(-1)^{|b_1|+1}.
\]
Equation \eqref{eq:boundaryorient} then gives
\[
(-1)^{|b_1|+1}\left(\sigma(u_0)(-1)^{|b_2|}+\sigma(u_1)\right)=0.
\]
Since $|b_{12}|=|b_1|+|b_2|+1$, it follows that $\sigma(u_0)=(-1)^{|b_1|+1}$ if and only if $\sigma(u_1)=(-1)^{|b_{12}|+1}$. By induction, we conclude that the lemma holds.
\end{pf}

%*********************************************************************
%*********************************************************************
\appendix
\section{Invariance from an analytical perspective}\label{app:aninv}
In this section we adapt the standard SFT proof of invariance to show
that the homotopy type of the DGA $A(\Lambda)$ for a Legendrian link
$\Lambda\subset Y_k(\delta)$ depends only on the isotopy class of
$\Lambda$. As mentioned in the Introduction, this is a proof strategy
rather than a proof since it depends on a perturbation scheme for
M-polyfolds that has not yet been fully worked out.

Consider an isotopy $\Lambda_t$, $t\in[0,1]$, between Legendrian links $\Lambda_{0}$ and $\Lambda_1$. The isotopy then gives an exact Lagrangian $L_{01}\subset Y_{k}(\delta)$ with positive end $\Lambda_{0}$ and negative end $\Lambda_{1}$. As mentioned in Section~\ref{sec:generalLCH}, this cobordism induces a DGA-morphism $\Phi_{01}\colon A(\Lambda_0)\to A(\Lambda_1)$. On the other hand, reversing the isotopy we get a corresponding map $\Phi_{10}\colon  A(\Lambda_{1})\to A(\Lambda_0)$ and by shortening the composite isotopy we get a  $1$-parameter family of cobordisms $L_{00}^{s}$, $s\in [0,1]$, with induced map at $s=0$ given by $\Phi_{10}\circ\Phi_{01}$ and at $s=1$ given by the identity. Thus, once we prove that a $1$-parameter family of cobordisms induces a chain homotopy, we can conclude that $\Phi_{10}\circ\Phi_{01}$ is chain homotopic to the identity. Similarly, $\Phi_{01}\circ\Phi_{10}$ is chain homotopic to identity and invariance will follow.

Below we will discuss some of the steps in this program, following \cite{E-ratSFT}.

%*********************************************************************
\subsection{Anchored disks}
Consider a  generic $1$-parameter family of cobordisms $L_s\subset Y(\delta)\times\R$, $s\in[0,1]$. In order to define the chain homotopy $K$ induced by $L_s$ we will consider the boundary of $1$-dimensional parameterized moduli spaces. That in turn involves looking at holomorphic disks of formal dimension $-1$. Using the argument from \cite[Lemma 4.5(1)]{EES-PxR} as in Section \ref{sec:generalLCH}, it is straightforward to show that disks with one positive boundary puncture and several negative boundary and interior punctures are transversely cut out after a small perturbation of the almost complex structure $J$. Recall that the Conley--Zehnder indices of the Reeb orbits in $Y_{k}(\delta)$ are $2m$, $m>0$. Thus, possible dimensions of spaces of holomorphic disks are $2m-1$. Let $\gamma$ be a simple central Reeb orbit in the middle of a handle, and let $\M^{X}(\gamma)$ denote the  moduli space of holomorphic disks with positive puncture at $\gamma$.

\begin{lma}\label{lem:noanchor2}
For generic $J$, the space $\M^{X}(\gamma)$ is a transversely cut out $1$-manifold with boundary points corresponding to $\R$-families in the moduli space $\M^{\R\times Y_k(\delta)}(\gamma)$.
\end{lma}

\begin{proof}
Since the Reeb orbit is simple the moduli space cannot contain multiple covers. Hence any curve in it is somewhere injective and the transversality follows from a standard argument. Since the Reeb orbit has minimal action there can be no bubbling and the boundary is as claimed.
\end{proof}

\begin{rmk}
It follows from the Bott--Morse description of symplectic homology that the moduli spaces $\M^{X}(\gamma)$ in fact are nonempty, since curves in them constrained by the unstable manifold of the critical point in the handle are needed for the symplectic homology differential to hit the corresponding Morse generator, which must be hit since the symplectic homology of a subcritical Weinstein manifold vanishes \cite{Kai}.
\end{rmk}

%*********************************************************************
\subsection{Producing chain homotopies}
Consider a $1$-parameter family of exact Lagrangian cobordisms $(Y_{k}(\delta),L_s)$, $s\in[0,1]$, together with a $1$-parameter family of almost complex structures $J_{s}$. We assume that $J_s$ is such that parameterized moduli spaces of disks with one positive boundary puncture are transversely cut out, and that this holds also for the moduli spaces at the endpoints $s=0,1$. Write $\Lambda_\pm$ for the Legendrian submanifolds at the positive and negative ends of $L_s$. Then $L_0$ and $L_1$ determine chain maps
\[
\Phi_{0},\Phi_1\colon A(\Lambda_+)\to A(\Lambda_-).
\]

\begin{lemma}\label{lem:chainhomotopy}
The DGA maps $\Phi_{0}$ and $\Phi_{1}$ are chain homotopic, i.e., there exists a finite sequence of chain maps $\Phi^{j}\colon A(\Lambda_+)\to A(\Lambda_-)$, $j=1,\dots, m$, with $\Phi^{1}=\Phi_0$ and $\Phi^{m}=\Phi_{1}$, and degree $+1$ maps $K^{j}\colon A(\Lambda_{+})\to A(\Lambda_-)$, $j=1,\dots,m-1$, such that %$\Phi^{1}=\Phi_0$ and $\Phi^{m}=\Phi_{1}$ and such that
\begin{equation}\label{eq:chhomtpy}
\Phi^{j+1}-\Phi^{j} =\Omega_{K^{j}}\circ\pa_+ + \pa_-\circ\Omega_{K^{j}},
\end{equation}
where $\pa_\pm$ is the differential on $A(\Lambda_{\pm})$ and where $\Omega_{K^{j}}$ is linear and is defined as follows on monomials:
\[
\Omega_{K^{j}}(c_1\dots c_r)=\sum_{l=1}^{r}(-1)^{|c_1|+\cdots+|c_{l-1}|}\Phi^{j+1}(c_1\cdots c_{l-1})K^{j}(c_l)\Phi^{j}(c_{l+1}\cdots c_r).
\]
\end{lemma}

\begin{rmk}
The lemma follows from (a slightly extended version of) \cite[Lemma B.15]{E-ratSFT} (which involves orientations on the moduli spaces), which is stated in somewhat different terminology. In the proof below, we will adapt the terminology used there to the current setup. Also, it should be mentioned that \cite[Lemma B.15]{E-ratSFT} depends on a perturbation scheme for so-called M-polyfolds (the most basic level of polyfolds), the details of which were not yet worked out, and hence it should be viewed as a proof strategy rather than a proof in the strict sense.
\end{rmk}

\begin{proof}
Write $\M^{s}(a;\mathbf{b})$ for the moduli space of (anchored) holomorphic disks in $Y_k(\delta)\times\R$ with boundary on $L_s$, positive puncture at the Reeb chord $a$, and negative punctures at the Reeb chords in the word $\mathbf{b}$.
Consider a parameterized moduli space
\[
\M(a;\mathbf{b})=\bigsqcup_{s\in[0,1]}\M^{s}(a;\mathbf{b}),
\]
where the $(L_s,J_s)$, $s\in[0,1]$, is generic.

We say that a holomorphic disk of formal dimension $d$ is a $(d)$-disk. Recalling that the Legendrian homology algebra is defined as a direct limit using the action filtration, we work below a fixed energy level. Since $(L_s,J_s)$ is generic, the $(-1)$-disks form a transversely cut-out $0$-manifold in the parameterized moduli space that is generic with respect to the projection to the parameter space. In particular this implies that $(-1)$-disks occur only at a finite number of isolated instances $0<s_1<s_2<\dots<s_m<1$ at which the cobordism contains exactly one $(-1)$-disk, and that the projection from the parameterized $1$-dimensional moduli space is a Morse function. It then follows that the chain maps $\Phi_s$ remain unchanged except possibly when $s$  crosses a $(-1)$-disk moment. We consider crossing such a $(-1)$-disk moment. For simpler notation we still take $s\in[0,1]$ and assume that there is a single $(-1)$-disk moment in the interval.

We first note that if the $(-1)$-disk has an interior puncture and gives an anchored $(0)$-disk, then it follows from Lemma \ref{lem:noanchor2} that $\Phi_{s}$ does not change.

In order to express the change in the chain map at other $(-1)$-disk moments, we use the $(-1)$-disk to construct a chain homotopy. To do so one must however overcome a transversality problem that arises when several copies of the $(-1)$-disk are attached to a disk in the symplectization of the upper end, resulting in a non-transverse broken disk. To solve this problem we invoke so-called abstract perturbations. More precisely, a perturbation is constructed that orders the negative punctures of any disk in the upper end in time so that only one negative puncture can be attached to the $(-1)$-disk at the time. Note that since we start the perturbation from the degenerate situation where all negative punctures lie at the same time, new $(-1)$-disks might arise when the perturbation is turned on, where perturbations of a moduli space $\M^{\Lambda_+\times\R}(a;\mathbf{b})$ in the positive end might give new $(-1)$-disks with positive puncture at $a$.

The actual perturbation scheme is organized energy level by energy level in such a way that the
the size of the time separation is determined by the action of the Reeb chord at the positive puncture. In particular the time distances between the positive punctures in the newly created $(-1)$-disk are of the size of this time separation. As we move to the next energy level, the time separation is a magnitude larger, so that one of the negative punctures of a disk on the new energy level passes all the positive punctures of the $(-1)$-disks created on lower energy levels before the puncture following it enters the region where $(-1)$-disks may be attached to it.

Consider now the parameterized $1$-dimensional moduli space $\M(a;\mathbf{b})$ of $(0)$-disks defined using the perturbation scheme just described. The boundary of $\M(a;\mathbf{b})$ then consists of the $0$-manifolds $\M^{0}(a,\mathbf{b})$ and $\M^{1}(a;\mathbf{b})$ as well as broken disks that consist of one $(-1)$-disk at $s$ and several $(0)$-disks with a $(1)$-disk in the upper or lower end attached. Thus, if for a Reeb chord $c$, $K(c)$ denotes the count of $(-1)$-disks \emph{after} the ordering perturbation scheme is turned on:
\[
K(c)=\sum_{|c|-|\mathbf{b}|=-1}\left|\M(c,\mathbf{b})\right|\mathbf{b},
\]
then, by counting the boundary points of oriented $1$-manifolds, we conclude that \eqref{eq:chhomtpy} holds.
\end{proof}

\begin{rmk}
The proof of invariance just given is in a sense less involved than the combinatorial proof given in Section \ref{app:combinv}. In fact this comparison is a little misleading, the combinatorial proof should be compared to computing the chain homotopies associated to all relevant $(-1)$-moments rather than just proving existence, see \cite{EHK}. For example, Gompf move 6 below involves an infinite sequence of birth-deaths of Reeb chords and the corresponding stable tame isomorphism involves a countably infinite stabilization, as does the chain map of the corresponding cobordism defined as a direct limit with respect to action.
\end{rmk}

%*********************************************************************
%*********************************************************************
\section{Invariance from a combinatorial perspective}\label{app:combinv}

In this section, we present proofs of various combinatorial results from Section~\ref{sec:comb} that have been deferred until now. We begin with the short proofs
of Proposition~\ref{prop:resolution} (about resolutions of fronts) and Proposition~\ref{prop:d2} ($\d^2=0$)
in Appendix~\ref{app:d2proof}. The bulk of this section, Appendix~\ref{sec:combpf}, consists of a proof of the main combinatorial invariance result, Theorem~\ref{thm:invariance}.

%*********************************************************************
\subsection{Proofs of Propositions~\ref{prop:resolution} and~\ref{prop:d2}}
\label{app:d2proof}

\begin{figure}
\centerline{\includegraphics[width=3.5in]{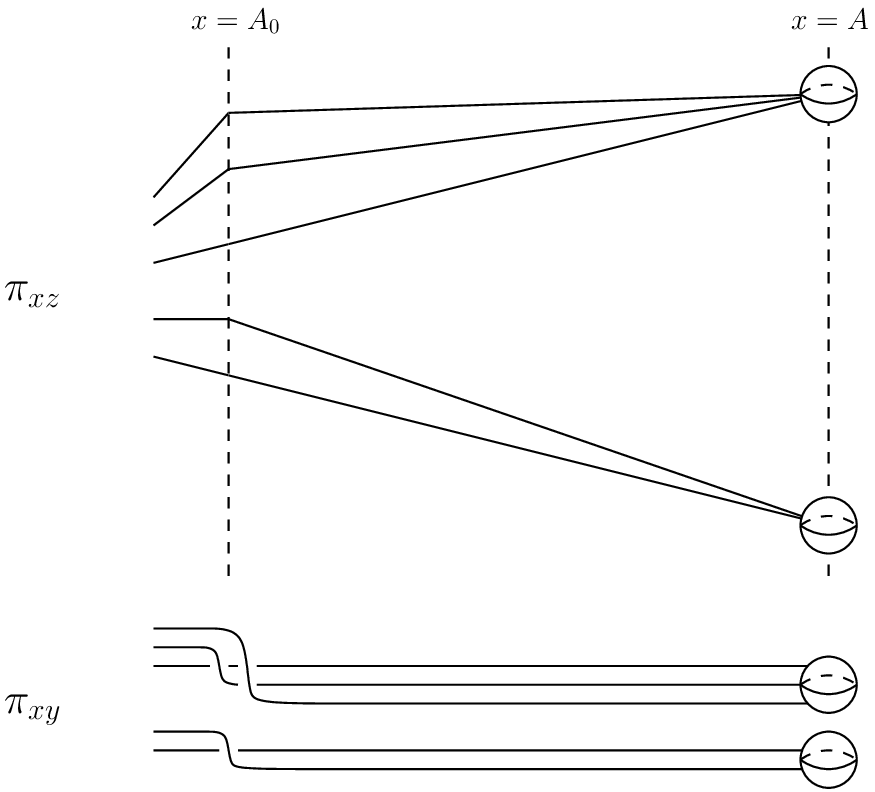}}
\caption{
Setting up the rightmost portion of a front ($\pi_{xz}$) projection so that the corresponding $xy$ projection is the resolution of the front.
}
\label{fig:resolve-right}
\end{figure}

\begin{proof}[Proof of Proposition~\ref{prop:resolution}]
This follows the proof of Proposition~2.2 in \cite{bib:NgCLI}. Perturb a front in Gompf standard form to be the $xz$ projection of a tangle in normal form, as in Figure~\ref{fig:gompf}. Now beginning from the left, change the front by planar isotopy so that the corresponding $xy$ projection is given by the resolution of the front, as in \cite{bib:NgCLI}. The only novel feature in the current setup is that the right end of the front has to be arranged to pass into the $1$-handles (that is, strands passing into the same $1$-handle must become arbitrarily close in $z$ coordinate) while remaining in normal form.

Thus, suppose that we have constructed the front from left to right, including all crossings and cusps, and the only remaining task is to pass into the $1$-handles. Then the right-hand end of the front, where the front intersects say $x=A_0$ for some $A_0$, consists of a number of straight line segments whose slopes are increasing from bottom to top. Now for each $1$-handle, choose one strand of the Legendrian link that passes through that $1$-handle, and extend the corresponding straight line segment (with slope unchanged) rightward from $x=A_0$ to $x=A$ for some $A \gg A_0$. Then connect the $x=A_0$ endpoints of all strands passing through that $1$-handle to the endpoint of the chosen strand at $x=A$. See Figure~\ref{fig:resolve-right} for an illustration.

For $A$ sufficiently large, the slopes of the line segments between $x=A_0$ to $x=A$ are, from bottom to top, increasing from handle to handle, and decreasing within the subset of strands passing into a single handle. Thus the corresponding $xy$ projection for $x\geq A_0$ looks precisely as desired: at $x=A_0$ there is a half twist involving the strands passing into each $1$-handle, and otherwise the $xy$ projection consists of horizontal lines. (To make the projection of the half twist generic, just perturb the $x=A_0$ endpoints slightly in the $x$ direction.)
This completes the proof.
\end{proof}

\begin{figure}
\centerline{
\includegraphics[width=\textwidth]{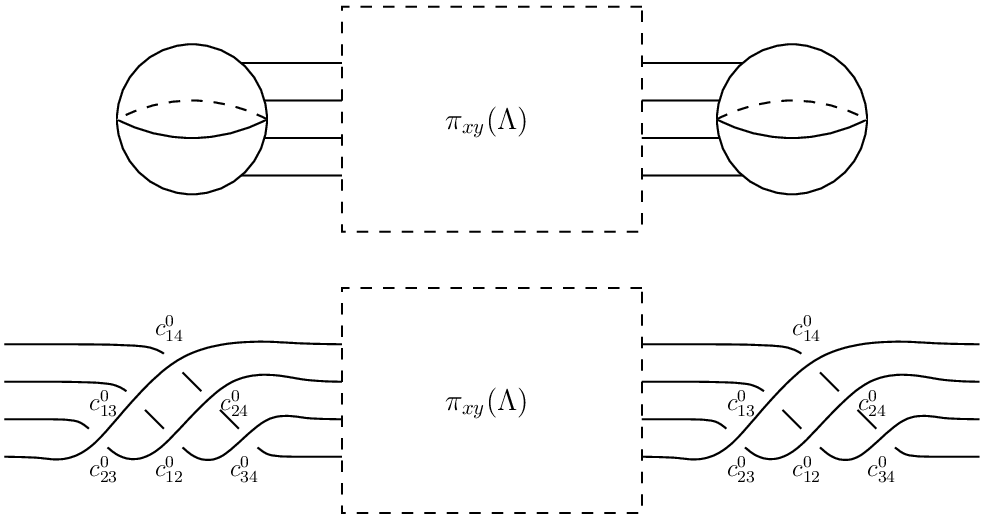}
}
\caption{
Turning a Legendrian link $\Lambda$ in $S^1\times S^2$ into a Legendrian tangle $\tilde{\Lambda}$ in $\R^3$.
}
\label{fig:d2}
\end{figure}

\begin{proof}[Proof of Proposition~\ref{prop:d2}]
The fact that this is true for generators of the internal DGA was noted in Section~\ref{ssec:intdga}. Now suppose that $a_i$ is a crossing of $\pi_{xy}(\Lambda)$. We wish to show that $\d^2(a_i)=0$ and that $|\d(a_i)| = |a_i|-1$. We will assume for notational reasons that $\Lambda \subset S^1\times S^2$, but the same proof works for any number of $1$-handles.

Replace the $1$-handle in the $xy$ projection of $\Lambda$ by
half-twists on the left and right, as in Figure~\ref{fig:d2}. The
result is the $xy$ projection of a Legendrian tangle $\tilde{\Lambda}$
in $\R^3$. To see this, add ``dips'' (cf.\ \cite{bib:Sabloff}) to
$\Lambda$ on the left and right, and then the half-twists are half of
each dip; see Figure~\ref{fig:lagdip} from Section~\ref{sec:split}. Let $n$ be the number of strands of $\Lambda$ passing through the $1$-handle. Label the crossings in the new half-twists in the tangle by $c_{ij}^0$, $1\leq i<j\leq n$, as shown in Figure~\ref{fig:d2} (note that these labels are repeated on left and right).

To the $xy$ projection of the Legendrian tangle $\tilde{\Lambda}$, we can associate a differential graded algebra over $\Z[\mathbf{t},\mathbf{t}^{-1}]$ just as in the standard Chekanov setup \cite{bib:Chekanov,bib:ENS}, where the differential counts bounded disks (no unbounded regions) with boundary in $\pi_{xy}(\tilde{\Lambda})$. Here crossings are graded as usual by rotation numbers of paths from overcrossings to undercrossings, where the $2n$ ends of the tangle are identified pairwise; note in particular that this agrees with the grading of $c_{ij}^0$ from Section~\ref{ssec:dga}. For definiteness of orientation signs, at each of the crossings labeled $c_{ij}^0$, if $|c_{ij}^0|$ is even, then we take the south and west corners to be the $-1$ corners for orientation signs.

With these conventions, the Chekanov differential on the algebra for $\tilde{\Lambda}$ agrees with the differential $\d$ on both the crossings $a_i$ in $\pi_{xy}(\Lambda)$ and on the $c_{ij}^0$. (In particular, this shows that $\d(a_i)$ is a finite sum of terms.) The usual proof that $\d^2=0$ and $\d$ lowers degree by $1$ from \cite{bib:Chekanov,bib:ENS} now gives the desired result.
\end{proof}

%*********************************************************************
%*********************************************************************
\subsection{Combinatorial proof of invariance}
\label{sec:combpf}

%*********************************************************************
\subsubsection{Outline of the proof}
\label{ssec:outlinepf}
In this section, we prove the main combinatorial invariance result, Theorem~\ref{thm:invariance}. The proof relies on establishing invariance under several elementary moves.
\begin{figure}
\centerline{
\includegraphics[width=\textwidth]{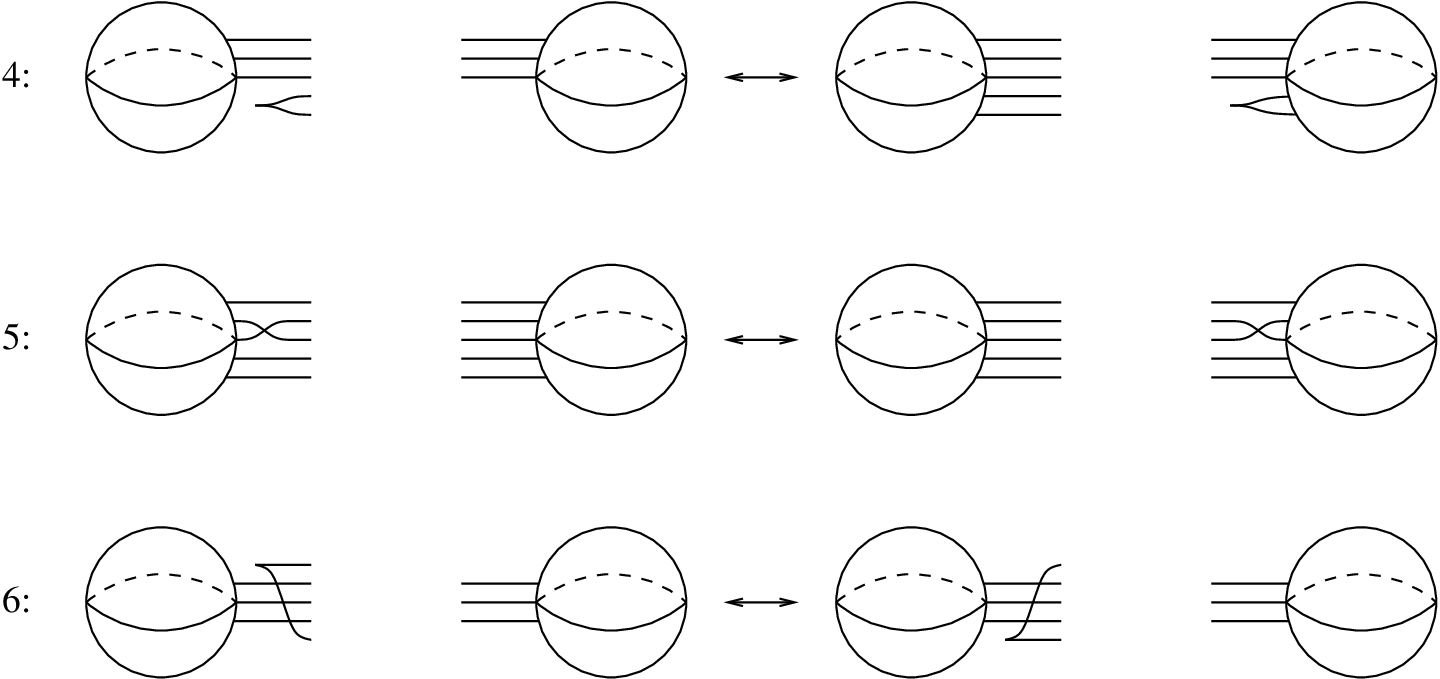}
}
\caption{
Gompf moves 4, 5, and 6.
}
\label{fig:gompfmoves}
\end{figure}
More precisely, by Gompf
\cite{bib:Gompf}, any Legendrian link in $\#^k(S^1\times
S^2)$ can be represented by a front
diagram in Gompf standard form (see Section~\ref{ssec:stdform}). Two
such front diagrams represent links that are Legendrian isotopic if
and only if they are related by a sequence of
Legendrian isotopy of the front tangle inside the box $[0,A] \times
[-M,M]$ (which can in turn be decomposed into a sequence of Legendrian
Reidemeister moves as in $\R^3$), and three other moves, which we call
Gompf moves 4, 5, and 6. See Figure~\ref{fig:gompfmoves}. As mentioned
in Section~\ref{ssec:resolution}, we can perturb any diagram in Gompf
standard form to a diagram in normal form, and the Gompf moves can be
similarly perturbed to involve diagrams in normal form. We will
henceforth assume that all front diagrams are in normal form.

We will prove invariance under each of Gompf's moves; this is the
content of Propositions~\ref{prop:isotopyinv}, \ref{prop:gompf4inv},
\ref{prop:gompf5inv}, and \ref{prop:gompf6inv} below. These results are then combined to prove Theorem~\ref{thm:invariance} in Section~\ref{sssec:combpf}.

\subsubsection{Reducing the number of possible moves}
Before turning to the actual study of the DGA under moves, we consider two technical
simplifications. First, Gompf's version of move 4 in
\cite{bib:Gompf} has the cusped strands above the others, rather than
below. The version we use, as shown in Figure~\ref{fig:gompfmoves}, is
equivalent to Gompf's version, either by symmetry or by additional
application of Gompf move 5.

Second, Gompf moves 4 and 6 have variants obtained from our diagrams
by a $180^\circ$ rotation, which also need to be considered
for a general Legendrian isotopy. However, we do not need
to separately consider these alternate moves, because the action of
$180^\circ$ rotation affects the DGA in a simple way, which we describe next.

\begin{definition}[cf.\ \cite{bib:Ngtransverse}]
Let $(\A,\d)$ be a semifree DGA over $R$ with generators $\{a_i\}$. There is
another differential $\d_{\op}$ on $\A$ defined on generators by
reversing the order of every word in $\d(a_i)$ and introducing
appropriate signs. More precisely, define an $R$-module involution
$\op :\thinspace \A\to\A$ by
\[
\op(a_{i_1}a_{i_2}\cdots a_{i_n}) = (-1)^{\sum_{j<k}
  |a_{i_j}||a_{i_k}|}
a_{i_n}\cdots a_{i_2}a_{i_1},
\]
and define $\d_{\op} = \op \circ \d \circ \op$. Then $(\A,\d_{\op})$
is another semifree DGA satisfying the Leibniz rule.
\label{def:op}
\end{definition}

\begin{lemma}
Let $\Lambda$ be a Legendrian link in $\#^k(S^1\times S^2)$ in
normal form, with front projection $\pi_{xz}(\Lambda)$. Then the
result of rotating $\pi_{xz}(\Lambda)$ by $180^\circ$ is the
front projection of a Legendrian link $\Lambda'$ in normal form, and
if $(\A,\d)$, $(\A',\d')$ are the DGAs associated to
$\pi_{xy}(\Lambda)$ and $\pi_{xy}(\Lambda')$, then $(\A',\d')$ is
(tamely) isomorphic to $(\A,\d_{\op})$.
\label{lem:rotation}
\end{lemma}

\begin{proof}
It is clear that the front of a normal-form link $\Lambda$, rotated $180^\circ$,
is again the front of a normal-form link $\Lambda'$.
Furthermore, the map $(x,y,z) \mapsto (A-x,y,-z)$ is a
contactomorphism of $\R^3$ sending $\Lambda$ to $\Lambda'$. The
resulting $xy$ projections $\pi_{xy}(\Lambda)$ and
$\pi_{xy}(\Lambda')$ are related by reflection in the vertical ($y$)
axis. There is an obvious one-to-one correspondence between generators
of $\A$ and generators of $\A'$, and it is easily checked that this
correspondence preserves grading. Since reflection reverses the
orientation of $\R^2$, terms appearing in the differential $\d$ have
analogous terms in $\d'$ but with the order of letters reversed. We
leave it as an exercise to check that the signs in $\d'$ and the signs
in $\d$ are related as in Definition~\ref{def:op}.
\end{proof}

%*********************************************************************
\subsubsection{Invariance under isotopy within $[0,A] \times [-M,M]$}
\label{ssec:invpf1}

%Here we prove the following.

\begin{proposition}
If $\Lambda$ and $\Lambda'$ are Legendrian links in $\#^k(S^1\times
S^2)$ in normal form, isotopic through Legendrian links in normal
form, then the DGAs associated to $\pi_{xy}(\Lambda)$ and
$\pi_{xy}(\Lambda')$ are stable tame isomorphic.
\label{prop:isotopyinv}
\end{proposition}

Note that the isotopy in Proposition~\ref{prop:isotopyinv} occurs
entirely within the box $[0,A] \times [-M,M]$ in the $xz$ projection,
and does not involve the portions of the Legendrian link in the $1$-handles.

Before presenting the proof of Proposition~\ref{prop:isotopyinv}, we
note an area estimate that is very similar to the corresponding
estimate for Legendrian links in $\R^3$. First, we say that a
Legendrian link $\Lambda$ is in \textit{normal-plus form} if:
\begin{itemize}
\item
$\Lambda$ is in normal form;
\item
the points
$(0,y_i^{\ell},z_i^{\ell})$ and
$(A,\tilde{y}_i^{\ell},\tilde{z}_i^{\ell})$ where $\Lambda$
intersects $x=0$ and $x=A$ satisfy
\begin{align*}
z_i^{\ell} &= z_1^{\ell} + (i-1)\epsilon, \hspace{5ex}
\ell=1,\ldots,k,~
i=1,\ldots,n_{\ell}, \\
\tilde{z}_i^{\ell} &= \tilde{z}_1^{\ell} + (i-1)\epsilon, \hspace{5ex}
\ell=1,\ldots,k,~
i=1,\ldots,n_{\ell},
\end{align*}
for all $t$ and arbitrary fixed small $\epsilon>0$.
\end{itemize}
That is, the points in the $xz$
projection where $\Lambda$ enters the
$1$-handles are evenly
spaced, with $\epsilon$ between the $z$-coordinates of adjacent points.

Now for $\Lambda$ in normal-plus form,
let $a_1,\ldots,a_n$ denote the crossings in $\pi_{xy}(\Lambda)$
(the Reeb chords of the portion of $\Lambda$ outside of the
$1$-handles), and define a height function $h$ on Reeb chords by
$h(a_i) = z_i^+-z_i^-$, where $z_i^{\pm}$ are the $z$-coordinates of
$\Lambda$ at crossing $a_i$, with $z_i^+ > z_i^-$. Extend this
height function to internal Reeb chords $c_{j_1j_2;\ell}^0$ by
$h(c_{j_1j_2;\ell}^0) = (j_2-j_1)\epsilon$; note that this is the difference
in $z$-coordinate between strand $i$ and strand $j$ as they enter the
$\ell^{\rm th}$ $1$-handle.

\begin{lemma}
Suppose that $\Lambda$ is in normal-plus form, $a_i$ is a crossing of
$\pi_{xy}(\Lambda)$, and $b_1,\ldots,b_r$ are each either
a crossing $a_j$ or a internal Reeb chord $c_{j_1j_2;\ell}^0$.
\label{lem:action}
If
$\Delta(a_i;b_1,\ldots,b_r) \neq \emptyset$, then
\[
h(a_i) > h(b_1) + \cdots + h(b_r).
\]
\end{lemma}

\begin{proof}
The argument is the same as for the analogous estimate in
\cite{bib:Chekanov}: by Stokes' Theorem, the difference
$h(a_i)-h(b_1)-\cdots-h(b_r)$ is the area of an immersed disk with
boundary on $\pi_{xy}(\Lambda_t)$ that has positive corner at $a_i$
and negative corners at $b_1,\ldots,b_r$, since they both measure
the integral of $dz = y\,dx$ around the boundary of the disk. The one
new feature is that
a ``negative corner'' at a $c_{j_1j_2;\ell}^0$ is not a corner, but
rather a strip bounded by strands $j_1$ and $j_2$ at handle $\ell$;
see Figure~\ref{fig:corners}. Since the $z$ coordinate jumps by
$(j_2-j_1)\epsilon = h(c_{j_1j_2;\ell}^0)$ at such a corner, the
result follows.
\end{proof}

\begin{proof}[Proof of Proposition~\ref{prop:isotopyinv}]
Suppose that we have an isotopy $\Lambda_t$ of Legendrian tangles in
normal form. By perturbing, we may assume that each $\Lambda_t$ is in
normal-plus form, and that the isotopy has generic
singularities in the $xy$ projection. We can further choose $\epsilon$
in the definition of normal-plus form to be sufficiently small that
each $h(c_{j_1j_2;\ell}^0)$ (which is bounded above by
$(n_\ell-1)\epsilon$)  is smaller than $h(a_i)$ for any crossing $a_i$
in any $\pi_{xy}(\Lambda_t)$.

We now follow Chekanov's proof of invariance of the DGA from the
$\R^3$ case \cite{bib:Chekanov}.
As $t$ goes from $0$ to $1$, the tangle diagram $\pi_{xy}(\Lambda_t)$
changes by planar isotopy and a sequence of Reidemeister moves, each
of which is either a triple-point move (Reidemeister III) or a
double-point move (Reidemeister II).

Invariance under a triple-point
move is precisely as in \cite{bib:Chekanov}. Invariance under a
double-point move also follows the proof from
\cite{bib:Chekanov}. Note in this case that the stable tame
isomorphism between the DGAs for two diagrams related by a
Reidemeister II move leaves all internal generators
$c_{j_1j_2;\ell}^p$ unchanged, and is defined on crossings
$a_i$ inductively, beginning with the crossing of smallest height and
working the (finite) way up to the crossing of largest height. The
fact that we can induct in this manner follows from Lemma~\ref{lem:action}.
\end{proof}

We will implicitly use Proposition~\ref{prop:isotopyinv} in the proofs of invariance under Gompf moves 4, 5, and 6 below. In each case, to simplify the proof, it is convenient to create new crossings in the $xy$ projection that localize certain holomorphic disks. These new crossings are created through Legendrian isotopy contained in $[0,A] \times [-M,M]$ in the $xy$ projection, which does not affect the stable tame isomorphism type of the DGA by Proposition~\ref{prop:isotopyinv}.

%*********************************************************************
\subsubsection{Invariance under Gompf move 4}
\label{ssec:invpf4}

%Here we prove invariance under Gompf move 4.

\begin{proposition}
If $\Lambda$ and $\Lambda'$ are Legendrian links in $\#^k(S^1\times
S^2)$ in normal form, related by Gompf move 4, then the DGAs
associated to $\pi_{xy}(\Lambda)$ and $\pi_{xy}(\Lambda')$ are stable
tame isomorphic.
\label{prop:gompf4inv}
\end{proposition}

\begin{figure}
\labellist
\small\hair 2pt
\pinlabel $\Lambda$ at 190 66
\pinlabel $\Lambda'$ at 559 66
\endlabellist
\centerline{
\includegraphics[width=\textwidth]{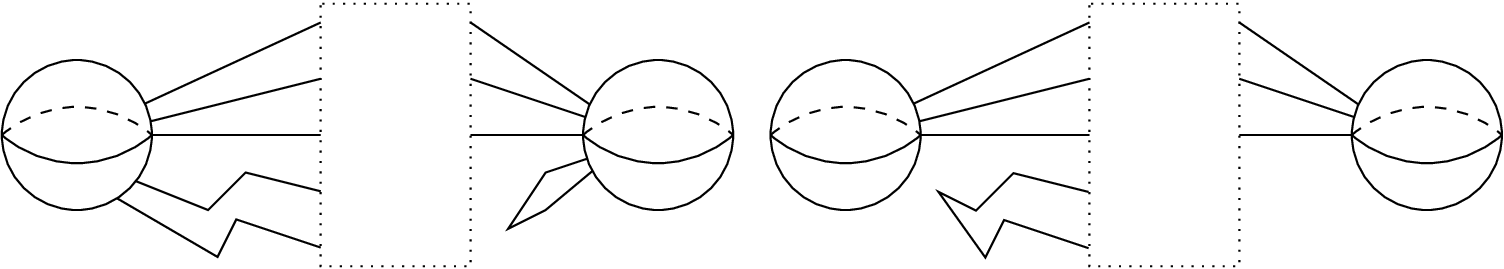}
}
\caption{
Front ($xz$) diagrams for Legendrian links $\Lambda,\Lambda'$ in normal
form related by Gompf
move 4. The corners in these fronts are understood to be smoothed
out.
}
\label{fig:gompf4a}
\end{figure}
\begin{figure}
\centerline{
\includegraphics[width=\textwidth]{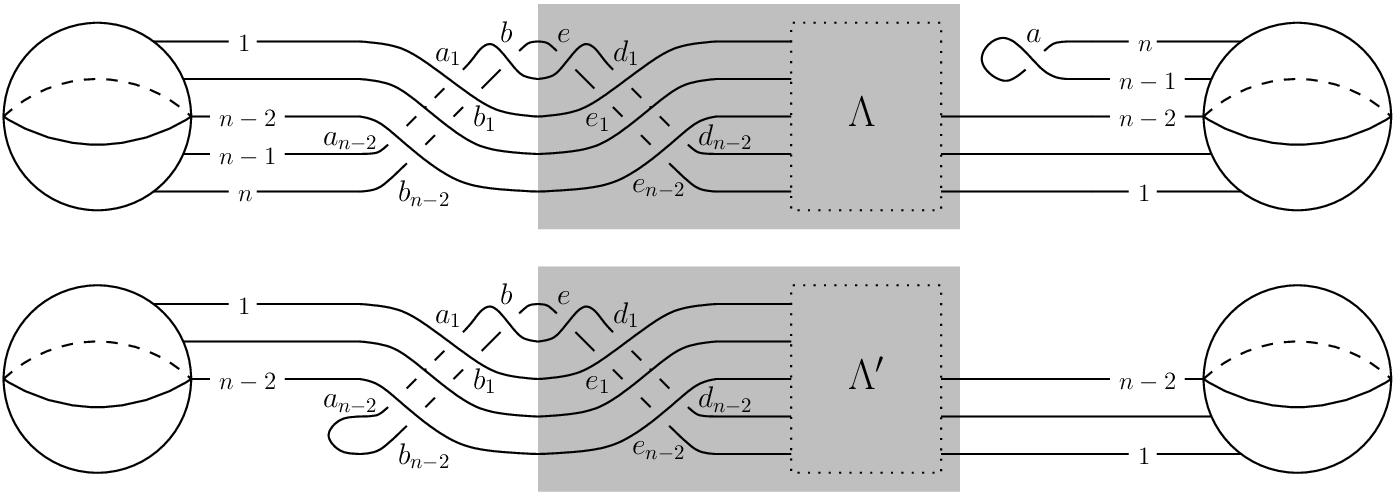}
}
\caption{
The $xy$ projections for
$\Lambda$ and $\Lambda'$ given in
Figure~\ref{fig:gompf4a}, with some crossings labeled. The dotted boxes are identical (and correspond to the dotted boxes in Figure~\ref{fig:gompf4a}), as
are the two shaded regions.
}
\label{fig:gompf4b}
\end{figure}

Let $\Lambda$ and $\Lambda'$ be Legendrian links in $\#^k(S^1\times
S^2)$ related by Gompf move 4. Perturb the fronts of $\Lambda$
and $\Lambda'$ to be in normal form, and then further perturb them to
take the form shown in Figure~\ref{fig:gompf4a}. Then the
corresponding $xy$ projections are
in $xy$-normal form as shown in
Figure~\ref{fig:gompf4b}.

By Proposition~\ref{prop:isotopyinv}, it
suffices to show that the DGAs for $\Lambda$ and $\Lambda'$ associated
to the $xy$ projections shown in Figure~\ref{fig:gompf5b} are stable
tame isomorphic. Denote these DGAs by $(\A,\d)$ and $(\A',\d')$
respectively. Note that we have labeled particular crossings in the
$xy$ projections of $\Lambda$ and $\Lambda'$; then $\A'$ is generated
by
\begin{itemize}
\item
$c_{ij}^p$, $1\leq i,j\leq n-2$, and $p\geq 1$ if $i\geq j$
\item
$b,e,a_1,\ldots,a_{n-2},b_1,\ldots,b_{n-2},d_1,\ldots,d_{n-2},e_1,\ldots,e_{n-2}$
\item
the other ``external'' crossings contained in the dashed box in
Figure~\ref{fig:gompf4b},
\end{itemize}
while $\A$ is generated by the same generators along with
\begin{itemize}
\item
$c_{ij}^p$, $1\leq i,j\leq n$, where at least one of $i,j$ is $n-1$ or
$n$, and $p\geq 1$ if $i\geq j$
\item
$a$.
\end{itemize}
For the purpose of keeping track of signs, let $(m(1),\ldots,m(n))$ be
the Maslov potential associated to strands $1,\ldots,n$ passing
through the handle, and write $\sigma_i = (-1)^{m(i)}$ for each
$i$. Note that $\sigma_{n-1} = -\sigma_n$, and if we write $\sgn(x) =
(-1)^{|x|}$, then $\sgn(a) = -1$ and $\sgn(c_{ij}^p) = -\sigma_i\sigma_j$.

Extend $\d'$ from $\A'$ to $\A$ by setting
\begin{alignat*}{2}
\d'(a) &= 1-c_{n-1,n}^0, & \hspace{6ex} & \\
\d'(c_{n-1,n}^0) &= 0, && \\
\d'(c_{in}^p) &= c_{i,n-1}^p, && 1\leq i\leq n,~p\geq 1,
\text{ if } i\geq n-1, \\
\d'(c_{i,n-1}^p) &= 0, && 1\leq i\leq n,~p\geq 1
\text{ if } i\geq n-1, \\
\d'(c_{n-1,j}^p) &= c_{nj}^p, && 1\leq j\leq n-2,~p\geq 1, \\
\d'(c_{nj}^p) &= 0, && 1\leq j\leq n-2,~p\geq 1;
\end{alignat*}
then $(\A,\d')$ is a stabilization of $(\A',\d')$. It suffices to
construct a tame automorphism of $\A$ intertwining $\d$ and $\d'$.
To this end, we will construct elementary automorphisms
$\Phi,\Psi,\Omega$ of $\A$ and establish that
\begin{equation}
\d \circ \Phi \circ \Psi \circ \Omega(x) = \Phi \circ \Psi \circ
\Omega \circ \d'(x)
\label{eq:gompf4}
\end{equation}
on generators $x$ of $\A$; then $\Phi\circ\Psi\circ\Omega$ is the desired
tame automorphism.

We break down the generators of $\A$ into groups, and establish
\eqref{eq:gompf4} for each group in turn:
\begin{itemize}
\item
Group 1: $a$, $c_{n-1,n}^0$,
$c_{ij}^0$ for $i,j\leq n-2$, and any crossing in the shaded regions in
Figure~\ref{fig:gompf4b} (including $e,e_1,\ldots,e_{n-2},d_1,\ldots,d_{n-2}$)
\item
Group 2: $c_{ij}^p$ for $p\geq 1$ and
$i,j\leq n-2$
\item
Group 3: $c_{i,n-1}^0$ and $c_{in}^0$ where $i\leq n-2$, and
$c_{ij}^p$ where $p\geq 1$ and either $i\geq n-1$ or $j\geq n-1$
\item
Group 4:
$a_1,\ldots,a_{n-2},b_1,\ldots,b_{n-2},b$.
\end{itemize}
We will also present the definitions of $\Phi,\Psi,\Omega$ in turn as
they are used. For now, we note that $\Phi,\Psi,\Omega$ will act as the
identity on all generators of $\A$ except:
\begin{itemize}
\item
for $\Phi$: Group 2
\item
for $\Psi$: Group 3
\item
for $\Omega$: Group 4.
\end{itemize}

If $x$ is in Group 1, then $\d(x) = \d'(x)$; note in particular that
this holds for $x=a$ and $x=c_{n-1,n}^0$ by the construction of $\d'$
and the fact (calculable from Figure~\ref{fig:gompf4b}) that $\d(a) =
1-c_{n-1,n}^0$. In addition, any generator
appearing in this differential is also in Group 1. It follows that
$\Phi\circ\Psi\circ\Omega$ acts as the identity on both $x$ and
$\d'(x)$, and thus that \eqref{eq:gompf4} holds for $x$ in Group 1.

We next construct $\Phi$ so that \eqref{eq:gompf4} holds for $x$ in
Group 2. First note that if $i,j\leq n-2$, then
\begin{align*}
\d(c_{ij}^p) &= \delta_{ij}\delta_{1p} + \sum_{m=1}^n \sum_{l=0}^p
\sigma_i\sigma_m c_{im}^l c_{mj}^{p-l}, \\
\d'(c_{ij}^p) &= \delta_{ij}\delta_{1p} + \sum_{m=1}^{n-2} \sum_{l=0}^p \sigma_i\sigma_m c_{im}^l c_{mj}^{p-l}.
\end{align*}
To define $\Phi$, we introduce some auxiliary elements of $\A$. We
extend our definition of $c_{ij}^p$ to allow one or both of the
indices $i,j$ to take the value $n-\half$, as follows:
\begin{alignat*}{2}
c_{n-\half,n}^0 &= c_{n-1,n-\half}^0 = 1, & \hspace{2ex}&\\
c_{n-\half,\alpha}^p &= c_{n-1,\alpha}^p - a c_{n\alpha}^p, && p \geq 1,~
\alpha\in\{1,\ldots,n-1,n-\half,n\}, \\
c_{\alpha,n-\half}^p &= c_{\alpha n}^p - c_{\alpha,n-1}^p a, &&
p\geq 0 \text{ and } \alpha\in\{1,\ldots,n-2\},\\
& && \text{or }
p\geq 1 \text{ and } \alpha\in\{n-1,n-\half,n\}, \\
c_{\alpha\beta}^0 &= 0, && \alpha\geq\beta.
\end{alignat*}
Note that the second and third lines overlap and agree in the definition of
$c_{n-\half,n-\half}^p$ for $p\geq 1$; in both cases, we find that
\[
c_{n-\half,n-\half}^p = c_{n-1,n}^p-c_{n-1,n-1}^p a-a c_{nn}^p+a c_{n,n-1}^p a.
\]
Also define
\[
\sigma_{n-\half} = \sigma_{n-1} = -\sigma_n.
\]

We collect a few facts about these new $c_{\alpha\beta}^p$ that are
easy to verify from the definition.

\begin{lemma} \label{lem:n-.5}
\begin{enumerate}
\item
For $\alpha,\beta\in \{1,\ldots,n-1,n-\half,n\}$, $p_2>0$, and either
$p_1>0$ or $p_1=0$ and $\alpha\geq n-2$,
we have
\[
c_{\alpha,n-1}^{p_1}c_{n-1,n-\half}^0c_{n-\half,\beta}^{p_2} -
c_{\alpha,n-\half}^{p_1}c_{n-\half,n}^0 c_{n\beta}^{p_2} =
c_{\alpha,n-1}^{p_1}c_{n-1,\beta}^{p_2} -
c_{\alpha n}^{p_1}c_{n\beta}^{p_2}.
\]
\item
For $p\geq 0$ and $\alpha,\beta\in\{1,\ldots,n-1,n-\half,n\}$ with at least one of $\alpha,\beta = n-\half$, we have
\[
\d(c_{\alpha\beta}^p) = \sum_{m=1}^n \sum_{l=0}^p \sigma_\alpha\sigma_m c_{\alpha m}^l c_{m\beta}^{p-l}
\]
where the sum on $m$ is only over integers.
\end{enumerate}
\end{lemma}

Now define $\Phi :\thinspace \A\to\A$ as
follows: $\Phi$ acts as the identity on all generators outside of
Group 2, and
if $p\geq 1$ and $i,j\leq n-2$, then
\[
\Phi(c_{ij}^p) = c_{ij}^p + \sum_{q\geq 1} (-1)^q \sum_{\lambda_0+\cdots+\lambda_q = p+1} c_{i,n-\half}^{\lambda_0-1} c_{n-\half,n-\half}^{\lambda_1} \cdots c_{n-\half,n-\half}^{\lambda_{q-1}} c_{n-\half,j}^{\lambda_q},
\]
where the second sum is over all ways to write $p+1 = \lambda_0+\cdots+\lambda_q$ for any $\lambda_0,\ldots,\lambda_q \geq 1$. For example, for $i,j\leq n-2$,
\begin{align*}
\Phi(c_{ij}^1) &= c_{ij}^1 - c_{i,n-\half}^0 c_{n-\half,j}^1, \\
\Phi(c_{ij}^2) &= c_{ij}^2 - c_{i,n-\half}^0 c_{n-\half,j}^2 - c_{i,n-\half}^1 c_{n-\half,j}^1 + c_{i,n-\half}^0 c_{n-\half,n-\half}^1 c_{n-\half,j}^1.
\end{align*}

Choose any ordering of $\A$ for which the $c_{ij}^p$ generators are ordered so that $c_{ij}^p<c_{i'j'}^p$ whenever $np+j-i<np'+j'-i'$. With respect to the induced filtration on $\A$, the differentials $\d$ and $\d'$ are filtered, sending $c_{ij}^p$ to a sum of terms lower in the filtration. It is then easy to check that $\Phi$ is elementary with respect to this ordering.

We can now establish \eqref{eq:gompf4} for $x$ in Group 2. This is a
result of the following lemma, along with the fact that the
differential $\d'$, applied to a generator in Group 2, only
involves generators in Groups 1 and 2.

\begin{lemma} \label{lem:dd'}
For all $p\geq 0$ and all $i,j\leq n-2$,
\[
\d \circ \Phi(c_{ij}^p) = \Phi \circ \d' (c_{ij}^p).
\]
\end{lemma}

\begin{proof}
We have
\[
\Phi \circ \d' (c_{ij}^p) =
\sum_{m\leq n-2} \sum_l \sigma_i\sigma_m\Phi(c_{im}^l)\Phi(c_{mj}^{p-l})+\delta_{ij}\delta_{1p}.
\]
When we expand this out using the definition of $\Phi$, besides $\delta_{ij}\delta_{1p}$, we obtain a sum of terms of the form
\[
\sigma_i\sigma_m (-1)^q c_{i,n-\half}\cdots c_{n-\half,m} c_{m,n-\half} \cdots c_{n-\half,j}
\]
with $m\leq n-2$, where we have suppressed the superscripts on the $c$'s, and $q$ is the length of the word $c_{i,n-\half}\cdots c_{n-\half,m} c_{m,n-\half} \cdots c_{n-\half,j}$. (Some terms are shorter than this, e.g., $\sigma_i\sigma_m (-1)^q c_{im} c_{m,n-\half}\cdots c_{n-\half,j}$, but the following argument still applies.)
Note that this term also appears, with the same sign, in
$\d((-1)^q c_{i,n-\half}\cdots c_{n-\half,n-\half} \cdots c_{n-\half,j})$, where
$q-1$ is the length of the word $c_{i,n-\half}\cdots c_{n-\half,n-\half} \cdots c_{n-\half,j})$,
by Lemma~\ref{lem:n-.5}(2) and the fact that $\sgn(c_{i,n-\half}) = -\sigma_i\sigma_n$ and $\sgn(c_{n-\half,n-\half}) = 1$. Thus every term in the expansion of $\Phi \circ \d' (c_{ij}^p)$ has a corresponding term in the expansion of $\d \circ \Phi(c_{ij}^p)$ (and this also holds for the extra term $\delta_{ij}\delta_{1p}$).

It suffices to check that the sum $S$ of the remaining terms in the expansion of $\d\circ\Phi(c_{ij}^p) = \d(\sum (\pm c_{i,n-\half}\cdots c_{n-\half,j}))$ is $0$. These are the terms for which, besides $i$ at the beginning and $j$ at the end, no subscripts $\leq n-2$ appear. We can write $S$ as a finite sum $S_1 + S_2 + \cdots$, where
$S_1$ is the contribution of $\partial(c_{ij}^p)$, $S_2$ is the contribution of $\partial(-\sum c_{i,n-\half}c_{n-\half,j})$, and so forth.
More precisely, we have
\begin{align*}
S_1 &= \sigma_i\sigma_{n-1} \sum_{\lambda_0+\lambda_1=p+1} (
c_{i,n-1}^{\lambda_0-1} c_{n-1,j}^{\lambda_1} - c_{in}^{\lambda_0-1} c_{nj}^{\lambda_1}), \\
S_2 &= -\sigma_i\sigma_{n-1} \sum_{\lambda_0+\lambda_1+\lambda_2=p+1}
\left( c_{i,n-1}^{\lambda_0-1}c_{n-1,n-\half}^{\lambda_1}
c_{n-\half,j}^{\lambda_2} - c_{in}^{\lambda_0-1}c_{n,n-\half}^{\lambda_1}
c_{n-\half,j}^{\lambda_2} \right. \\
& \qquad\qquad\qquad\qquad\qquad \left. + c_{i,n-\half}^{\lambda_0-1} c_{n-\half,n-1}^{\lambda_1} c_{n-1,j}^{\lambda_2} -
c_{i,n-\half}^{\lambda_0-1}c_{n-\half,n}^{\lambda_1}c_{nj}^{\lambda_2} \right)
\end{align*}
and so forth, where all sums are over $\lambda_0\geq 1$ and $\lambda_i \geq
0$ for $i\geq 1$. Each nonzero term in these sums either has all $\lambda_q \geq 1$,
or it has exactly one $\lambda_q = 0$ (where $q>0$) and incorporates either $c_{n-1,n-\half}^0$ or $c_{n-\half,n}^0$. We can then write $S_r = S_r^0 + S_r^1$, where $S_r^0$ sums the terms with some $\lambda_q=0$, and $S_r^1$ sums the terms with all $\lambda_q \geq 1$.

Note that $S_1^0 = 0$ and $S_1^1 = S_1$, while
\begin{align*}
S_2^0 &= -\sigma_i\sigma_{n-1} \sum_{\lambda_0+\lambda_2 = p+1} \left(c_{i,n-1}^{\lambda_0-1}c_{n-1,n-\half}^{0} c_{n-\half,j}^{\lambda_2}
-
c_{i,n-\half}^{\lambda_0-1}c_{n-\half,n}^{0}c_{nj}^{\lambda_2} \right) \\
&= -\sigma_i\sigma_{n-1} \sum_{\lambda_0+\lambda_2 = p+1}
\left(c_{i,n-1}^{\lambda_0-1}c_{n-1,j}^{\lambda_2} -
c_{in}^{\lambda_0-1} c_{nj}^{\lambda_2}\right) \\
&= -S_1^1,
\end{align*}
where the middle equality follows from Lemma~\ref{lem:n-.5}(1). Similarly $S_r^0 = -S_{r-1}^1$ for all $r \geq 2$, whence $S = \sum_r (S_r^0+S_r^1) = 0$.
\end{proof}

We next turn our attention to the definition of $\Psi$ and establishing \eqref{eq:gompf4} for Group 3 generators. First define the differential $\d'' :\thinspace \A\to\A$ by $\d'' = \Phi^{-1} \circ \d \circ \Phi$. For any $1\leq i\leq n$ and $p\geq 0$, we have
\begin{align*}
\d''(c_{in}^p) &= \Phi^{-1}\d\Phi(c_{in}^p) =
\Phi^{-1}\d(c_{in}^p)\\
& = \Phi^{-1}(\sigma_i\sigma_{n-1}c_{i,n-1}^pc_{n-1,n}^0+\cdots)
= \sigma_i\sigma_{n-1}c_{i,n-1}^pc_{n-1,n}^0+\cdots,
\end{align*}
where the ellipses involve only terms lower than $c_{i,n-1}^p$ in the filtration chosen earlier. Denote the final ellipsis by
\[
\d''_0(c_{in}^p) := \d''(c_{in}^p) - \sigma_i\sigma_{n-1}c_{i,n-1}^pc_{n-1,n}^0.
\]
Similarly, for any $1\leq j\leq n-2$ and $p\geq 1$, we can define $\d''_0(c_{n-1,j}^p)$ by
\[
\d''_0(c_{n-1,j}^p) := \d''(c_{n-1,j}^p) + c_{n-1,n}^0 c_{nj}^p,
\]
and $\d''_0(c_{n-1,j}^p)$ is lower in the filtration than $c_{nj}^p$.

Now define the algebra map $\Psi :\thinspace \A\to\A$ by
\begin{alignat*}{2}
\Psi(c_{i,n-1}^p) &= \sigma_i\sigma_{n-1}c_{i,n-1}^p+\d''_0(c_{in}^p)+\d''(c_{i,n-1}^p)a & \hspace{2ex} &
p\geq 1 \text{ and } 1\leq i\leq n,\\
& && \text{or } p=0 \text{ and } i\leq n-2\\
\Psi(c_{in}^p) &= c_{in}^p - c_{i,n-1}^p a && p\geq 1 \text{ and } 1\leq i\leq n,\\
& && \text{or } p=0 \text{ and } i\neq n-2, \\
\Psi(c_{nj}^p) &= -c_{nj}^p+\d''_0(c_{n-1,j}^p)+a\d''(c_{nj}^p) &&
p\geq 1 \text{ and } 1\leq j\leq n-2 \\
\Psi(c_{n-1,j}^p) &= c_{n-1,j}^p - a c_{nj}^p &&
p \geq 1 \text{ and } 1\leq j\leq n-2
\end{alignat*}
and $\Psi$ acts as the identity on all other generators of $\A$.
From the above discussion, $\Psi$ is elementary with respect to the chosen ordering on $\A$.

\begin{lemma}
If $x$ is a generator in Group 3, then \label{lem:gompf4psi}
\[
\d''\circ \Psi (x) = \Psi\circ\d'(x).
\]
\end{lemma}

\begin{proof}
If either $p\geq 1$ or $p=0$ and $i\leq n-2$, we have
\begin{align*}
\d''\Psi(c_{in}^p) &= \d''(c_{in}^p-c_{i,n-1}^p a) \\
&= (\sigma_i\sigma_{n-1}c_{i,n-1}^pc_{n-1,n}^0 + \d''_0(c_{in}^p))\\
& +
(\d''(c_{i,n-1}^p)a-\sigma_i\sigma_{n-1} c_{i,n-1}^p(1+c_{n-1,n}^0)) \\
&= \Psi(c_{i,n-1}^p) \\
&= \Psi\d'(c_{in}^p)
\end{align*}
whence $\d''\Psi(c_{i,n-1}^p) = (\d'')^2 \Psi(c_{in}^p) = 0 = \Psi\d'(c_{i,n-1}^p)$, and the lemma holds for $x=c_{in}^p$ and $x=c_{i,n-1}^p$.
If $p\geq 1$ and $j\leq n-2$, then a similar calculation yields $\d''\Psi(c_{n-1,j}^p) = \Psi\d'(c_{n-1,j}^p)$ and
$\d''\Psi(c_{nj}^p) = 0 = \Psi\d'(c_{nj}^p)$, and the lemma holds for $x=c_{n-1,j}^p$ and $x=c_{nj}^p$.
\end{proof}

It follows immediately from Lemma~\ref{lem:gompf4psi} that \eqref{eq:gompf4} holds when $x$ is a Group 3 generator, since then
$\d\Phi\Psi\Omega(x) = \d\Phi\Psi(x) = \Phi\d''\Psi(x) = \Phi\Psi\d'(x)
= \Phi\Psi\Omega\d'(x)$.

Finally, it remains to define $\Omega$ and establish \eqref{eq:gompf4} for Group 4 generators. Define the algebra map $\Omega :\thinspace \A\to\A$ by
\begin{align*}
\Omega(a_i) &= a_i-\sigma_i\sigma_{n-1}c_{in}^0, &\hspace{3ex} & 1\leq i\leq n-2, \\
\Omega(b_i) &= b_i+a_i a, && 1\leq i\leq n-2, \\
\Omega(b) &= b+a, &&
\end{align*}
and $\Omega$ is the identity on all other generators of $\A$. We can choose an ordering of $\A$ (independently from before) such that $a<b$ and $c_{in}^0 a < a_i < b_i$ for all $1\leq i\leq n-2$; then $\Omega$ is elementary with respect to this ordering.

An inspection of Figure~\ref{fig:gompf4b} yields, for $i\leq n-2$:
\begin{align*}
\d(a_i) &= \sum_{j=i+1}^{n-2} c_{ij}^0 a_j + c_{i,n-1}^0 - \sigma_i\sigma_{n-1} d_i, \\
\d'(a_i) &= \sum_{j=i+1}^{n-2} c_{ij}^0 a_j - \sigma_i\sigma_{n-1} d_i, \\
\d(b_i) &= \sum_{j=i+1}^{n-2} c_{ij}^0 b_j + c_{in}^0 - \sigma_i\sigma_{n-1} a_i c_{n-1,n}^0 - \sigma_i\sigma_n e_i - \sigma_i\sigma_n d_i b, \\
\d'(b_i) &= \sum_{j=i+1}^{n-2} c_{ij}^0 b_j - \sigma_i\sigma_{n-1} a_i  - \sigma_i\sigma_n e_i - \sigma_i\sigma_n d_i b, \\
\d(b) &= e+c_{n-1,n}^0, \\
\d'(b) &= e+1.
\end{align*}
Note in particular that there are no holomorphic disks whose only positive puncture is at $a_i$ or $b_i$ and extend into the dotted boxes in Figure~\ref{fig:gompf4b}. This can be shown in general via an area estimate, or more simply by assuming that the figures in the dotted boxes are the result of resolution, in which case the rightmost point in such a holomorphic disk would have to be a positive puncture (cf.\ the differential for simple fronts in \cite{bib:NgCLI}).

One can now check directly that \eqref{eq:gompf4} holds for generators in Group 4, i.e., $a_i$, $b_i$, and $b$. For instance, from the definition of $\Psi$, we have $\Psi(c_{in}^0) = c_{in}^0-c_{i,n-1}^0 a$, and thus
\begin{align*}
\d\Phi\Psi\Omega(a_i) &= \d\Phi\Psi(a_i-\sigma_i\sigma_{n-1}c_{in}^0) \\
&= \d(a_i-\sigma_i\sigma_{n-1}(c_{in}^0-c_{i,n-1}^0 a)) \\
&= -\sigma_i\sigma_{n-1} d_i + \sum_{j=i+1}^{n-2} c_{ij}(a_j-\sigma_j\sigma_{n-1}(c_{jn}^0-c_{j,n-1}^0a)) \\
&= \Phi\Psi(-\sigma_i\sigma_{n-1} d_i + \sum_{j=i+1}^{n-2} c_{ij}(a_j-\sigma_j\sigma_{n-1}c_{jn}^0)) \\
&= \Phi\Psi\Omega(-\sigma_i\sigma_{n-1} d_i + \sum_{j=i+1}^{n-2} c_{ij}^0 a_j) \\
&= \Phi\Psi\Omega\d'(a_i).
\end{align*}
Similar computations hold for $b_i$ and $b$.

This completes the proof of \eqref{eq:gompf4} for all generators of $\A$, and the proof of invariance under Gompf move 4.

%*********************************************************************
\subsubsection{Invariance under Gompf move 5}
\label{ssec:invpf5}

%Here we prove invariance under Gompf move 5.

\begin{proposition}
If $\Lambda$ and $\Lambda'$ are Legendrian links in $\#^k(S^1\times
S^2)$ in normal form, related by Gompf move 5, then the DGAs
associated to $\pi_{xy}(\Lambda)$ and $\pi_{xy}(\Lambda')$ are stable
tame isomorphic.
\label{prop:gompf5inv}
\end{proposition}

\begin{figure}
\centerline{
\includegraphics[width=\textwidth]{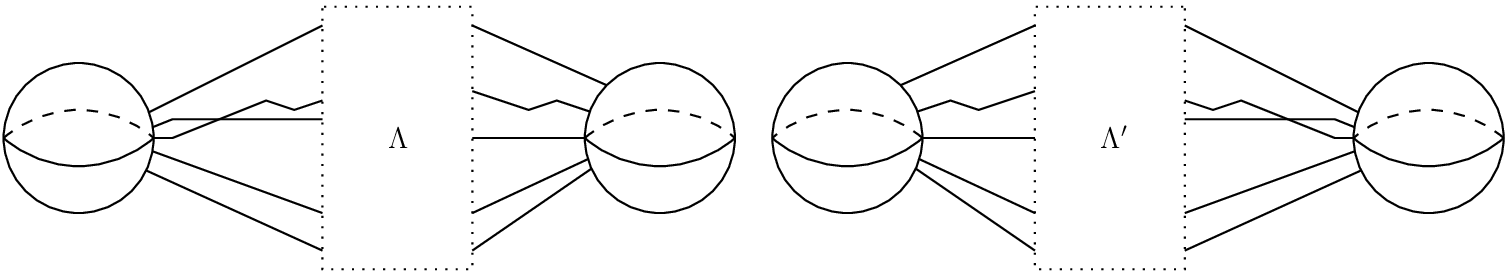}
}
\caption{
Front ($xz$) diagrams for Legendrian links $\Lambda,\Lambda'$ in normal
form related by Gompf
move 5. The corners in these fronts are understood to be smoothed
out.
}
\label{fig:gompf5a}
\end{figure}
\begin{figure}
\centerline{
\includegraphics[width=4.2in]{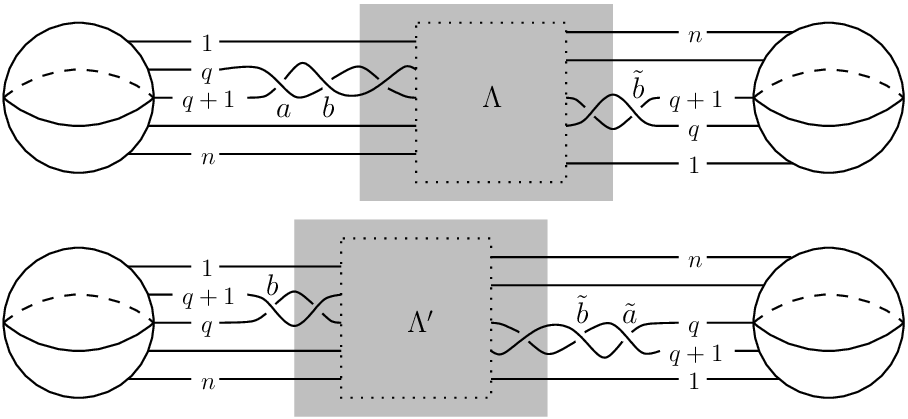}
}
\caption{
The $xy$ projections for
$\Lambda$ and $\Lambda'$ given in
Figure~\ref{fig:gompf5a}. The dotted boxes are identical (and correspond to the dotted boxes in Figure~\ref{fig:gompf5a}), as
are the two shaded regions. Note the unusual labeling of strands in
the $xy$ projection of $\Lambda'$: $q$ and $q+1$ are switched.
}
\label{fig:gompf5b}
\end{figure}

Let $\Lambda$ and $\Lambda'$ be Legendrian links in $\#^k(S^1\times
S^2)$ related by Gompf move 5. Perturb the fronts of $\Lambda$
and $\Lambda'$ to be in normal form, and then further perturb them to
take the form shown in Figure~\ref{fig:gompf5a}. Then the
corresponding $xy$ projections are
in $xy$-normal form as shown in
Figure~\ref{fig:gompf5b}.

By Proposition~\ref{prop:isotopyinv}, it
suffices to show that the DGAs for $\Lambda$ and $\Lambda'$ associated
to the $xy$ projections shown in Figure~\ref{fig:gompf5b} are stable
tame isomorphic. Denote these DGAs by $(\A,\d)$ and $(\A',\d')$
respectively. (Note that we have labeled particular
crossings $a,b,\tilde{b}$ in $\pi_{xy}(\Lambda)$ and
$\tilde{a},b,\tilde{b}$ in $\pi_{xy}(\Lambda')$.)

The generator sets for $\A$ and $\A'$ are
identical, except that $a$ and $c_{q,q+1}^0$ are generators for $\A$ but
not $\A'$, and $\tilde{a}$ and $c_{q+1,q}^0$ are generators for $\A'$
but not $\A$. Note that the presence of $c_{q+1,q}^0$ rather than
$c_{q,q+1}^0$ in $\A'$ is due to the switching of labels $q$ and $q+1$
in the bottom diagram of Figure~\ref{fig:gompf5b}. As for grading, the
strands passing through the $1$-handle in $\pi_{xy}(\Lambda)$ and
$\pi_{xy}(\Lambda')$ share a common Maslov potential
$(m(1),\ldots,m(n))$, with $m(i)$ being the potential for the strand
labeled $i$, and it follows that all common
generators of $\A$ and $\A'$ have the same grading in both. For future
reference, write $\sigma_i = (-1)^{m(i)}$ and note that
\begin{equation}
\sgn a = \sgn \tilde{a} = \sgn b = \sgn \tilde{b} =
-\sgn c_{q,q+1}^0 = -\sgn c_{q+1,q}^0 = \sigma_q\sigma_{q+1}
\label{eq:gompf5signs}
\end{equation}
where $\sgn a = (-1)^{|a|}$, etc.

Define $\tilde{\A}$ to be the graded tensor algebra generated by the union of
the sets of generators of $\A$ and $\A'$. We can extend the
differential $\d$ from $\A$ to $\tilde{\A}$, and $\d'$ from $\A'$ to $\tilde{\A}$, by setting
\begin{align*}
\d(\tilde{a}) &= -c_{q+1,q}^0, &
\d(c_{q+1,q}^0) &= 0, \\
\d'(a) &= c_{q,q+1}^0, &
\d'(c_{q,q+1}^0) &= 0.
\end{align*}
Then $(\tilde{\A},\d)$ is a stabilization of $(\A,\d)$,
$(\tilde{\A},\d')$ is a stabilization of $(\A',\d')$, and it suffices
to show that $(\tilde{\A},\d)$ and $(\tilde{\A},\d')$ are tamely isomorphic.

Define an algebra automorphism $\Phi$ of $\tilde{\A}$ as follows:
\begin{alignat*}{2}
\Phi(b) &= b-\tilde{a} & \hspace{3ex} & \\
%\Phi(c_{p,p+1}^0) &= c_{p,p+1}^0 && \\
\Phi(c_{qj}^p) &= c_{qj}^p - \sigma_q \sigma_{q+1} a c_{q+1,j}^p, &&
j\neq q+1;~p\geq 1 \text{ if } j\leq q, \\
\Phi(c_{i,q+1}^p) &= c_{i,q+1}^p + c_{iq}^p a, && i\neq q;~
p\geq 1 \text{ if } i\geq q+1, \\
\Phi(c_{q,q+1}^p) &= c_{q,q+1}^p + c_{qq}^p a && \\
& - \sigma_q\sigma_{q+1} (a c_{q+1,q+1}^p + a c_{q+1,q}^p a), && p\geq 1,
%\Phi(c_{ij}^k) &= c_{ij}^k && i\neq p \text{ and } j\neq p+1
\end{alignat*}
and $\Phi$ acts as the identity on all other generators of
$\tilde{\A}$. (In particular, $\Phi(c_{q,q+1}^0) = c_{q,q+1}^0$ and
$\Phi(c_{q+1,q}^0)=c_{q+1,q}^0$.) It is easy to check that $\Phi$ is
elementary: choose any ordering such that $a,\tilde{a}$ are at the
bottom and for fixed $p$, the $c_{ij}^p$ are ordered so that
$c_{ij}^p < c_{i'j'}^p$ if $j-i<j'-i'$.
Similarly, define an elementary automorphism $\Phi'$ of $\tilde{\A}$
by
\begin{alignat*}{2}
\Phi'(\tilde{b}) &= \tilde{b}-a & \hspace{3ex} & \\
\Phi'(c_{q+1,j}^p) &= c_{q+1,j}^p+\sigma_q\sigma_{q+1}\tilde{a}c_{qj}^p
&&
j\neq q;~p\geq 1 \text{ if } j\leq q+1 \\
\Phi'(c_{iq}^p) &= c_{iq}^p - c_{i,q+1}^p \tilde{a} &&
i\neq q+1;~p\geq 1 \text{ if } i\geq q \\
\Phi'(c_{q+1,q}^p) &= c_{q+1,q}^p-c_{q+1,q+1}^p\tilde{a} &&\\
&+
\sigma_q\sigma_{q+1}(\tilde{a}c_{qq}^p-\tilde{a}c_{q,q+1}^p\tilde{a})
&& p\geq 1,
\end{alignat*}
and $\Phi'$ acts as the identity on all other generators of $\tilde{\A}$.

To complete the proof of invariance for Gompf move 5, we will
establish that
\begin{equation}
\Phi^{-1}\circ\d\circ\Phi = (\Phi')^{-1}\circ\d'\circ\Phi'
\label{eq:gompf5}
\end{equation}
on $\tilde{\A}$. It suffices to establish \eqref{eq:gompf5} on
generators of $\tilde{\A}$, which we divide into groups:
$a,b,\tilde{a},\tilde{b}$; internal
generators $c_{ij}^p$, including $c_{q,q+1}^0$ and $c_{q+1,q}^0$; and
external generators, which correspond to crossings in the shaded
regions of Figure~\ref{fig:gompf5b}. For $a,b,\tilde{a},\tilde{b}$,
note that $\partial(a) = c_{q,q+1}^0$, $\partial(b) = x$, and
$\partial'(b) = c_{q+1,q}^0+x$, where $x$ only involves external
generators (in fact, up to sign, $x$ is just the external generator
corresponding to the leftmost crossing in the shaded region), and so
$\Phi^{-1}\partial\Phi a = c_{q,q+1}^0 = (\Phi')^{-1}\partial'\Phi' a$
and
\begin{align*}
\Phi^{-1}\partial\Phi b &= \Phi^{-1}\partial(b-\tilde{a}) =
\Phi^{-1}(x+c_{q+1,q}^0)\\
& =x+c_{q+1,q}^0=
(\Phi')^{-1}(x+c_{q+1,q}^0)=(\Phi')^{-1}\partial'\Phi'b.
\end{align*}
Similarly, \eqref{eq:gompf5} holds for $\tilde{a}$ and $\tilde{b}$.

Next, we check \eqref{eq:gompf5} for internal generators
$c_{ij}^p$. Let $\pi :\thinspace \tilde{\A} \to \tilde{\A}$ be the
algebra map that sends $c_{q,q+1}^0$ and $c_{q+1,q}^0$ to $0$ and acts
as the identity on all other generators of $\tilde{\A}$. Since the
internal differentials $\d$ and $\d'$ only differ in terms that
involve either $c_{q,q+1}^0$ and $c_{q+1,q}^0$, we have
$\pi\d c_{ij}^p = \pi\d' c_{ij}^p$ for all $i,j,p$. Equation
\eqref{eq:gompf5} for internal generators is the content of the following result.

\begin{lemma}
For all $i,j,p$, we have
\[
\Phi^{-1}\d\Phi(c_{ij}^p) = \pi\d c_{ij}^p=\pi\d' c_{ij}^p
= (\Phi')^{-1}\d'\Phi'(c_{ij}^p).
\]
\end{lemma}

\begin{proof}
We will show that $\Phi^{-1}\d\Phi(c_{ij}^p) = \pi\d c_{ij}^p$; the
proof that $(\Phi')^{-1}\d'\Phi'(c_{ij}^p)=\pi\d' c_{ij}^p$ is
essentially identical. We first establish this when $i\neq q$ and $j \neq q+1$. In this case, from the definition of $\Phi$, we have
\[
\Phi(\sigma_i\sigma_q c_{iq}^l c_{qj}^{p-l} +
\sigma_i \sigma_{q+1} c_{i,q+1}^l c_{q+1,j}^{p-l}) =
\sigma_i\sigma_q c_{iq}^l c_{qj}^{p-l} +
\sigma_i \sigma_{q+1} c_{i,q+1}^l c_{q+1,j}^{p-l},
\]
whence $\Phi(\d c_{ij}^p) = \d c_{ij}^p$. Thus
\[
\Phi^{-1}\d\Phi(c_{ij}^p) =
\Phi^{-1}\d c_{ij}^p = \d c_{ij}^p = \pi \d c_{ij}^p,
\]
where the last equality follows from the fact that $c_{q,q+1}^0$ does not appear in $\d c_{ij}^p$.

It remains to prove the lemma when $i=q$ or $j=q+1$. This breaks into
three cases: $i=q$ and $j\neq q+1$, $i\neq q$ and $j=q+1$, and $i=q$
and $j=q+1$. We will establish the first case and leave the similar
calculations for the other two cases to the reader.
For $j\neq q+1$, we have
\begin{align*}
\Phi\pi\d(c_{qj}^p) &=
\delta_{qj}\delta_{1p} + \sum_{(m,l) \neq (q+1,0)} \sigma_q\sigma_m\Phi(c_{qm}^l) \Phi(c_{mj}^{p-l}) \\
&= \delta_{qj}\delta_{1p} + \sum_{(m,l) \neq (q+1,0)} \sigma_q\sigma_m(c_{qm}^l-\sigma_q\sigma_{q+1} a
c_{q+1,m}^l)c_{mj}^{p-l} \\
& \quad + \sum_{l>0} (c_{qq}^l-\sigma_q\sigma_{q+1} a
c_{q+1,q}^l)(-\sigma_q\sigma_{q+1}a c_{q+1,j}^{p-l}) \\
& \quad + \sum_{l>0} \sigma_q\sigma_{q+1} (c_{qq}^l a - \sigma_q
\sigma_{q+1} a c_{q+1,q}^{p-l} a) c_{q+1,j}^{p-l} \\
& = \delta_{qj}\delta_{1p} + \sum_{(m,l) \neq (q+1,0)} \sigma_q\sigma_m(c_{qm}^l-\sigma_q\sigma_{q+1} a
c_{q+1,m}^l)c_{mj}^{p-l} \\
&= \delta_{qj}\delta_{1p} + \sum_{m,l} \sigma_q\sigma_m c_{qm}^l
c_{mj}^{p-l} - \sigma_q\sigma_{q+1} c_{q,q+1}^0 c_{q+1,j}^p\\
&\quad - \sum_{m,l}
\sigma_{q+1}\sigma_m a c_{q+1,m}^l c_{mj}^{p-l} \\
&= \d(c_{qj}^p - \sigma_q\sigma_{q+1} a c_{q+1,j}^p) \\
&= \d\Phi(c_{qj}^p),
\end{align*}
where the second and third sums after the second equality arise from
the fact that $\Phi(c_{mj}^{p-l})$ and $\Phi(c_{qm}^l)$ have
additional terms when $m=q$ and $m=q+1$, respectively. Thus
$\Phi^{-1}\d\Phi(c_{qj}^p) = \pi\d c_{qj}^p$, as desired.
\end{proof}

\begin{figure}
\centerline{
\includegraphics[width=5.5in]{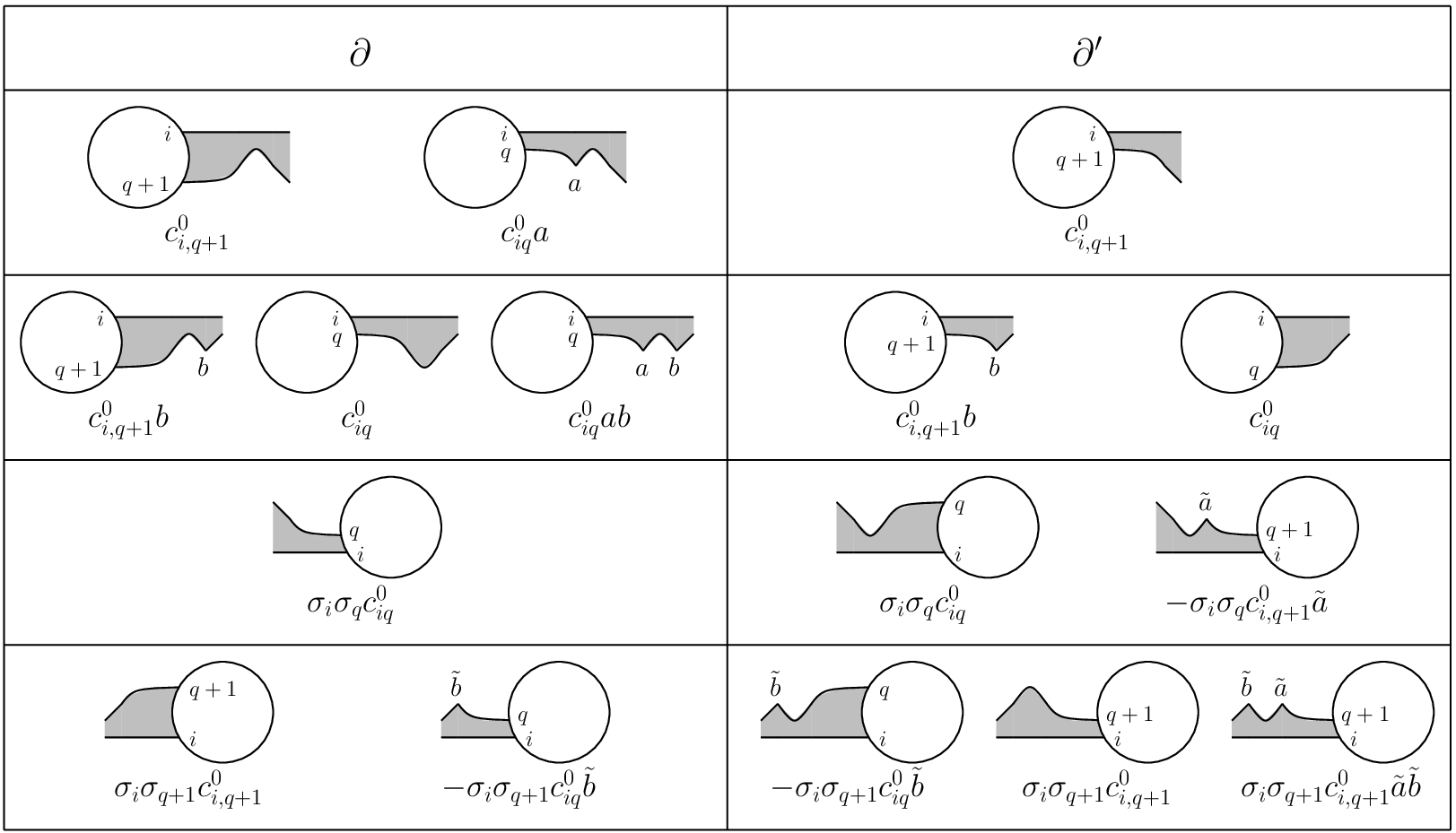}
}
\caption{
Disks contributing to the differentials $\d x$
or $\d'x$ and involving
$c_{iq}^0,c_{i,q+1}^0$ for $i<q$, arranged into families; the depicted
portions are the parts outside the shaded region in Figure~\ref{fig:gompf5b}. Some of
these disks also have negative corners at $a$, $\tilde{a}$, $b$, or
$\tilde{b}$, as shown. Circles represent the handle $2$-spheres. The
positive corner $x$ is not shown but is to the right of each diagram
where the $2$-sphere is on the left, and vice versa.
The words underneath each disk are the contributions of the depicted
portion of the disk to the corresponding word in $\d$ or $\d'$.
Orientation signs are calculated from the rules in
Section~\ref{ssec:dga}, with some simplification due to \eqref{eq:gompf5signs}.
}
\label{fig:gompf5disks}
\end{figure}

Finally, we need to check \eqref{eq:gompf5} for external generators,
corresponding to crossings in the shaded region of
Figure~\ref{fig:gompf5b}. Let $x$ be such a crossing; we need to show
that $\Phi^{-1}\d x = (\Phi')^{-1}\d' x$. The only generators
appearing in $\d x$ or $\d' x$ that are affected by $\Phi$ or $\Phi'$
are $b$, $\tilde{b}$, and generators of the form
$c_{iq}^0,c_{i,q+1}^0$ for $i<q$ and $c_{qj}^0,c_{q+1,j}^0$. The terms
in $\d x$ and $\d' x$ that do not include any of these generators are
identical, while the terms that do include these generators can be
grouped as shown in Figure~\ref{fig:gompf5disks} (which shows terms involving
$c_{iq}^0$ or $c_{i,q+1}^0$ for $i<q$; there is an entirely analogous
group of diagrams for terms involving $c_{qj}^0$ or $c_{q+1,j}^0$ for
$j>q+1$). For instance, the top left set of diagrams corresponds to a
collection of terms $y(c_{i,q+1}^0+c_{iq}^0 a)z$ in $\d x$ and a term
$y c_{i,q+1}^0 z$ in $\d' x$. The fact that $\Phi^{-1}\d x =
(\Phi')^{-1}\d' x$ now follows from the following identities, easily
checked from the definitions of $\Phi$ and $\Phi'$:
\begin{align*}
\Phi^{-1}(c_{i,q+1}^0+c_{iq}^0 a) &=
(\Phi')^{-1}(c_{i,q+1}^0) = c_{i,q+1}^0 \\
\Phi^{-1}(c_{iq}^0+c_{i,q+1}^0 b+c_{iq}^0 ab) &=
(\Phi')^{-1}(c_{iq}^0+c_{i,q+1}^0 b) = c_{iq}^0+c_{i,q+1}^0(b+\tilde{a})\\
\Phi^{-1}(c_{iq}^0) &= (\Phi')^{-1}(c_{iq}^0-c_{i,q+1}^0\tilde{a}) = c_{iq}^0\\
\Phi^{-1}(c_{i,q+1}^0-c_{iq}^0\tilde{b}) &=
(\Phi')^{-1}(c_{i,q+1}^0-c_{iq}^0\tilde{b}+c_{i,q+1}^0\tilde{a}\tilde{b})\\
& = c_{i,q+1}^0-c_{iq}^0(a+\tilde{b}),
\end{align*}
along with the analogous identities involving $c_{qj}^0$ and $c_{q+1,j}^0$.

%*********************************************************************
\subsubsection{Invariance under Gompf move 6, part 1}
\label{ssec:invpf6}

%Here we prove invariance under Gompf move 6.

\begin{proposition}
If $\Lambda$ and $\Lambda'$ are Legendrian links in $\#^k(S^1\times
S^2)$ in normal form, related by Gompf move 6, then the DGAs
associated to $\pi_{xy}(\Lambda)$ and $\pi_{xy}(\Lambda')$ are stable
tame isomorphic.
\label{prop:gompf6inv}
\end{proposition}

The proof of Proposition~\ref{prop:gompf6inv} occupies this section and the following section \ref{ssec:invpf6more}.

\begin{figure}
\centerline{
\includegraphics[width=\textwidth]{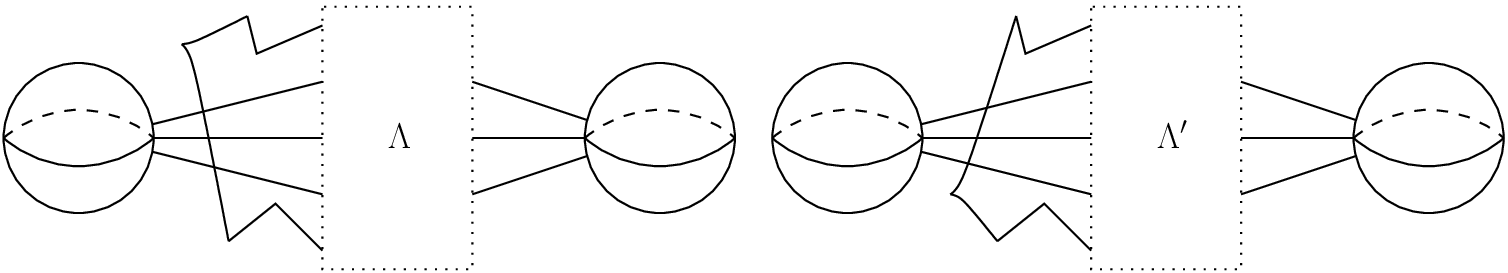}
}
\caption{
Front ($xz$) diagrams for Legendrian links $\Lambda,\Lambda'$ in normal
form related by Gompf
move 6. The corners in these fronts are understood to be smoothed
out.
}
\label{fig:gompf6a}
\end{figure}
\begin{figure}
\centerline{
\includegraphics[width=\textwidth]{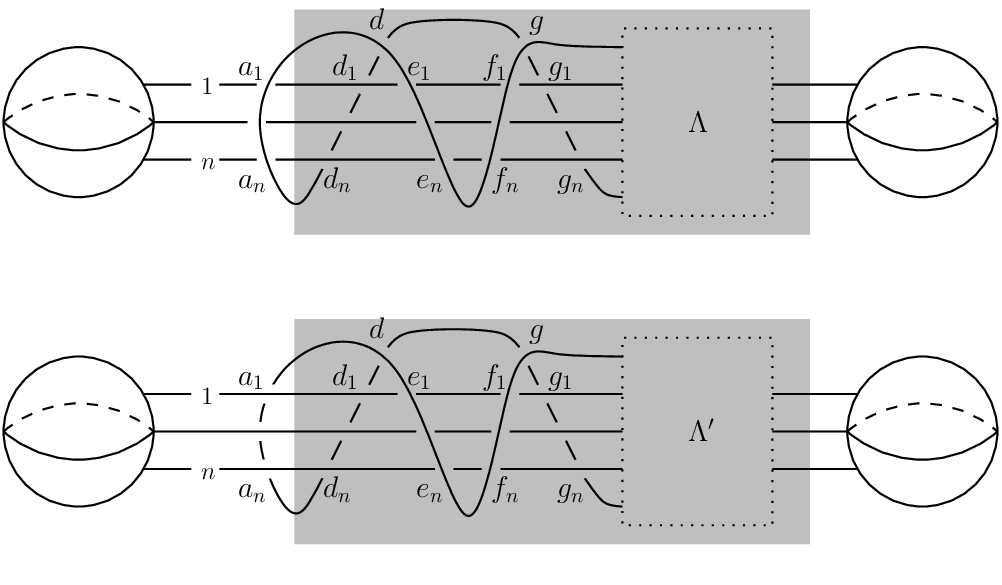}
}
\caption{
The $xy$ projections for
$\Lambda$ and $\Lambda'$ given in
Figure~\ref{fig:gompf6a}, with some crossings labeled. The dotted boxes are identical (and correspond to the dotted boxes in Figure~\ref{fig:gompf6a}), as
are the two shaded regions.
}
\label{fig:gompf6b}
\end{figure}

Let $\Lambda$ and $\Lambda'$ be Legendrian links in $\#^k(S^1\times
S^2)$ related by Gompf move 6. Perturb the fronts of $\Lambda$
and $\Lambda'$ to be in normal form, and then further perturb them to
take the form shown in Figure~\ref{fig:gompf6a}. Then the
corresponding $xy$ projections are
in $xy$-normal form as shown in
Figure~\ref{fig:gompf6b}.

By Proposition~\ref{prop:isotopyinv}, it
suffices to show that the DGAs for $\Lambda$ and $\Lambda'$ associated
to the $xy$ projections shown in Figure~\ref{fig:gompf6b} are stable
tame isomorphic. Denote these DGAs by $(\A,\d)$ and $(\A',\d')$
respectively. We will show that these DGAs are both stable tame
isomorphic to a third DGA $(\A'',\d'')$, which we now define.

Note that $\A$ is generated by internal generators $c_{ij}^p$, where $1\leq
i,j\leq n$, $p\geq 0$, and $i<j$ if $p=0$, along with external
generators corresponding to crossings in $\pi_{xy}(\Lambda)$,
including $a_1,\ldots,a_n$ as marked in Figure~\ref{fig:gompf6b} as
well as $b_1,\ldots,b_n$ and so forth. Let $\A''$ be the algebra
generated by the internal generators $c_{ij}^p$, along with all
external generators \textit{except} $a_1,\ldots,a_n$. This algebra
inherits a grading from $\A$. There is an algebra projection map $\pi
:\thinspace \A\to\A''$ that maps $a_1,\ldots,a_n$ to $0$ and every
other generator of $\A$ to itself; that is, it projects away any word
involving any $a_i$. We now define $\d'':\thinspace\A''\to\A''$ by
\[
\d'' = \pi\circ\d.
\]

We will prove that $(\A,\d)$ and $(\A'',\d'')$ are stable tame
isomorphic. The proof that $(\A',\d')$ and $(\A'',\d'')$ are stable
tame isomorphic is entirely similar. (Note for this purpose that there is also a
projection map $\pi':\thinspace \A'\to\A''$ that sends the generators
$a_1,\ldots,a_n$ of $\A'$ to $0$ and acts as the identity on all other
generators, and that $\d'' = \pi'\circ\d'$.)

For grading purposes, let $(m(1),\ldots,m(n))$ be the Maslov
potentials associated to strands $1,\ldots,n$ passing through the
handle, and let $m(0)$ be the Maslov potential associated to the
strand in the front projection of $\Lambda$ in
Figure~\ref{fig:gompf6a} that crosses over the other strands (and
includes the bottom half of the left cusp). Then we have
\[
|a_i| = m(0)-m(i)
\]
for $1\leq i\leq n$. Write $\sigma_i = (-1)^{m(i)}$ for $1\leq i\leq
n$, and $\sigma = (-1)^{m(0)}$; then
$\sgn(c_{ij}^p) = -\sigma_i\sigma_j$ and
$\sgn(a_i) = \sigma \sigma_i$.

Let $\tilde{\A}$ be the algebra generated by the generators of $\A$,
along with infinite families of generators
\[
\{a_q\}_{q=n+1}^{\infty} \text{ and } \{b_q\}_{q=n+1}^{\infty}
\]
with grading given by $|a_{q+n}|=|a_q|+2$ for all $q\geq 1$ and
$|b_q| = |a_q|-1$ for all $q\geq n+1$. Extend $\d$ from $\A$ to
$\tilde{\A}$ by setting
\[
\d(a_q)=b_q,~~~~~~~ \d(b_q)=0 \text{ for } q\geq n+1.
\]
Then $(\tilde{\A},\d)$ is a stabilization of $(\A,\d)$.

We can also view $\tilde{\A}$ as being generated by the generators of
$\A''$, along with $\{a_q\}_{q=1}^{\infty}$ and
$\{b_q\}_{q=n+1}^{\infty}$. Extend $\d''$ from $\A''$ to $\tilde{\A}$
by setting
\[
\d''(b_{q+n}) = a_q,~~~~~~ \d''(a_q)=0 \text{ for } q\geq 1;
\]
note that this is a graded differential since
$|b_{q+n}|=|a_q|+1$. Then $(\tilde{\A},\d'')$ is a stabilization of
$(\A'',\d'')$. It thus suffices to establish the following result.

\begin{lemma}
$(\tilde{\A},\d)$ and
$(\tilde{\A},\d'')$ are tamely isomorphic.
\label{lem:gompf6}
\end{lemma}

We will construct the tame isomorphism in Lemma~\ref{lem:gompf6} in
two steps, one of which is deferred to Section~\ref{ssec:invpf6more}.
Let $\tilde{\A}_0$ denote the subalgebra of $\tilde{\A}$ generated by
internal generators $c_{ij}^p$ along with $a_i$ for $i\geq 1$ and
$b_i$ for $i\geq n+1$. This is in fact a differential subalgebra of
$\tilde{\A}$ with respect to $\d$: indeed, by inspection of
Figure~\ref{fig:gompf6b}\footnote{For the purposes of computing signs in the differentials of $a_i$ here as well as other external generators later, we choose the following crossing quadrants to have orientation signs: southwest, southeast for $a_i$, $e_i$, and $f_i$; southeast, northeast for $d_i$; northwest, southwest for $g_i$; south, east for $d$; and south, west for $g$.}
and the definition of the internal
differential, we have
\begin{align*}
\d(a_i) &= -\sum_{j=1}^{i-1} \sigma \sigma_j a_j c_{ji}^0, \hspace{3ex}
\text{ for }
i\leq n, \\
\d(c_{ij}^p) &= \delta_{ij}\delta_{1p} + \sum_{\ell=0}^p \sum_{m=1}^n
\sigma_i\sigma_m c_{im}^{\ell}c_{mj}^{p-\ell}.
\end{align*}

In Section~\ref{ssec:invpf6more}, we will establish the following
result.

\begin{lemma}
There is a tame automorphism $\Psi$ of $\tilde{\A}$ with the following
properties:
\begin{enumerate}
\item \label{gompf6more1}
$\Psi$ acts as the identity on all generators of $\A$ except for
$\{a_i\}_{i\geq 1}$ and $\{b_i\}_{i\geq n+1}$, and restricts to a tame
automorphism of $\tilde{\A}_0 \subset \tilde{\A}$;
\item \label{gompf6more2}
for $q\leq n$,
\[
\Psi(a_q) = a_q - \sum_{m=n+1}^{q+n-1} \sigma\sigma_m b_m c_{m-n,q}^0;
\]
\item \label{gompf6more3}
on $\tilde{\A}_0$,
\[
\Psi\circ\d = \d''\circ\Psi.
\]
\end{enumerate}
\label{lem:gompf6more}
\end{lemma}

We now use Lemma~\ref{lem:gompf6more} to prove Lemma~\ref{lem:gompf6} and thus invariance under Gompf move 6.

\begin{proof}[Proof of Lemma~\ref{lem:gompf6}]
Define $\Omega :\thinspace \tilde{\A} \to \tilde{\A}$ to be the algebra map given by
\begin{align*}
\Omega(e_i) &= e_i + b_{i+n}, \hspace{5ex} 1\leq i\leq n, \\
\Omega(d) &= d+\sum_{i=1}^n b_{n+i}d_i,
\end{align*}
and $\Omega(x) = x$ for all other generators of $\tilde{\A}$. Note that $\Omega$ is an elementary automorphism with respect to any ordering for which $b_{n+1},\ldots,b_{2n}$ are less than the $e_i$, $d_i$, and $d$. We will show that
\begin{equation}
\Omega\circ\Psi\circ\d(x) = \d''\circ\Omega\circ\Psi(x)
\label{eq:gompf6}
\end{equation}
for all generators $x$ of $\tilde{\A}$, and thus $\Omega\circ\Psi$ is the desired tame automorphism of $\tilde{\A}$ intertwining $\d$ and $\d''$.

For a generator $x$ of $\tilde{\A}_0 \subset \tilde{\A}$, since $\Omega$ restricts to the identity on $\tilde{\A}_0$, we have $\Omega\Psi\d(x) = \Psi\d(x) = \d''\Psi(x) = \d''\Omega\Psi(x)$ by Lemma~\ref{lem:gompf6more}(\ref{gompf6more3}). Otherwise, $x$ corresponds to a crossing in the shaded region of the top diagram in Figure~\ref{fig:gompf6b}. An inspection of Figure~\ref{fig:gompf6b} shows that unless $x=d$ or $x=e_i$ for some $1\leq i\leq n$, $\d(x)$ does not involve any of $a_1,\ldots,a_n,e_1,\ldots,e_n$; then by Lemma~\ref{lem:gompf6more}(\ref{gompf6more1}), $\d(x) = \Psi\d(x)$, and so
$\Omega\Psi\d(x) = \Omega\d(x) = \d(x) = \d''(x) = \d''\Omega\Psi(x)$.

It remains to check \eqref{eq:gompf6} for $x=d$ or $x=e_i$. To this end, another inspection of Figure~\ref{fig:gompf6b} yields
\begin{align*}
\d(d_i) &= -g_i + \sum_{j=i+1}^n c_{ij}^0 d_j, & \hspace{2ex} & 1\leq i\leq n, \\
\d(e_i) &= a_i -f_i +  \sum_{j<i} \sigma\sigma_j e_j c_{ji}^0, && 1\leq i\leq n, \\
\d(d) &= 1-g+\sum_{i=1}^n a_i d_i - \sum_{i=1}^n \sigma\sigma_i e_i g_i.
\end{align*}
Then by direct calculation and Lemma~\ref{lem:gompf6more}(\ref{gompf6more2}), we have
\begin{align*}
\Omega\Psi\d(e_i) &= \Omega\left(a_i-f_i+\sum_{j=1}^{i-1} \sigma\sigma_j(e_j-b_{j+n})c_{ji}^0 \right) \\
&= a_i-f_i+\sum_{j<i} \sigma\sigma_je_jc_{ji}^0 \\
&= \d''(e_i+b_{n+i}) \\
&= \d''\Omega\Psi(e_i)
\end{align*}
and
\begin{align*}
\Omega\Psi\d(d) &= 1-g+\sum_{i=1}^n \left((a_i-\sum_{j=1}^{i-1} \sigma\sigma_j b_{j+n}c_{ji}^0)d_i-\sigma\sigma_i (e_i+b_{i+n}) g_i \right) \\
&= \d''(d) + \sum_i (a_id_i-\sigma\sigma_i b_{i+n}g_i) - \sum_{1\leq j<i\leq n} \sigma\sigma_j b_{n+j} c_{ji}^0 d_i \\
&= \d''(d+\sum_{i=1}^n b_{i+n}d_i) \\
&= \d''\Omega\Psi(d). \qedhere
\end{align*}
\end{proof}

%*********************************************************************
\subsubsection{Invariance under Gompf move 6, part 2}
\label{ssec:invpf6more}

In this section we prove Lemma~\ref{lem:gompf6more}, completing the proof
of invariance for Gompf move 6.
When defining the tame isomorphism $\Psi$, it will be convenient
to replace the notation $c_{ij}^p$ by $c_{ab}$ for $a,b\in\Z$, where
$c_{ab}$ is defined as follows: write $a=a_1n+a_2$ and $b=b_1n+b_2$
for $1\leq a_2,b_2\leq n$; then
\[
c_{ab} = \begin{cases} 0 & a\leq b \\
c_{a_2,b_2}^{b_1-a_1} & a>b.
\end{cases}
\]
Also extend $\sigma_1,\ldots,\sigma_n$ to $\sigma_a$ for all $a\in\Z$
by setting $\sigma_a = \sigma_{a_2}$.
With this notation, the following result is trivial to check.

\begin{lemma}
\begin{enumerate}
\item
For any $a,b\in\Z$, $c_{ab} = c_{a+n,b+n}$.
\item
For any $a,b\in\Z$, $\d(c_{ab}) = \delta_{a+n,b} + \sum_{m=a+1}^{b-1}
\sigma_a\sigma_m c_{am}c_{mb}$.
\end{enumerate}
\end{lemma}

For $s\geq n$, define a map $\d_s :\thinspace \tilde{\A}_0 \to
\tilde{\A}_0$ as follows:
\begin{align*}
\d_s(a_q) &= \begin{cases} b_q, & q>s, \\
-\displaystyle{\sum_{m=s-n+1}^{q-1} \sigma\sigma_m a_m c_{mq}}, & q\leq
s, \\
0, & q\leq s-n+1,
\end{cases} \\
\d_s(b_q) &= \begin{cases} 0, & q>s \\
a_{q-n}, & (n+1\leq\,)\,q\leq s,
\end{cases} \\
\d_s(c_{ab}) &= \d(c_{ab}).
\end{align*}
(The case $q\leq s-n+1$ in the definition of $\d_s(a_q)$ is
superfluous, following from the expression for $q\leq s$, but is
included for ease of reference.) Extend $\d_s$ to all of
$\tilde{\A}_0$ by the Leibniz rule. Note that $\d_n=\d$.

We next claim that $\d_n,\d_{n+1},\ldots$ are related by tame
automorphisms of $\tilde{\A}_0$. To this end, for $s\geq n+1$, define
algebra maps $\phi_s,\psi_s :\thinspace \tilde{\A}_0 \to \tilde{\A}_0$
as follows:
\begin{align*}
\phi_s(b_s) &= b_s - \sum_{m=s-n}^{s-1} \sigma\sigma_m  a_m c_{ms}, &\\
\psi_s(a_q) &= a_q-\sigma\sigma_s b_s c_{s,q+n}, & \text{for }
(s-n+1\leq\,)\, q\leq s, \\
\end{align*}
and $\phi_s,\psi_s$ act as the identity on all generators of
$\tilde{\A}_0$ besides
$b_s$ for $\phi_s$ and $a_{s-n+1},\ldots,a_s$ for $\psi_s$.

\begin{lemma}
For all $s\geq n+1$, we have $\psi_s\phi_s\d_{s-1} = \d_s\psi_s\phi_s$.
\label{lem:gompf6psiphi}
\end{lemma}

\begin{proof}
We will show that $\psi_s\phi_s\d_{s-1}(x) = \d_s\psi_s\phi_s(x)$ for
all generators $x$ of $\tilde{\A}_0$. Since
$\psi_s,\phi_s$ do not affect generators of the form $c_{ij}^p$,
$\psi_s\phi_s\d_{s-1}(c_{ij}^p) = \d(c_{ij}^p) =
\d_s\psi_s\phi_s(c_{ij}^p)$. Similarly, it follows from the
definitions of $\d_s,\d_{s-1},\psi_s,\phi_s$ that if $q>s$, we have
$\psi_s\phi_s\d_{s-1}(a_q) = b_q = \d_s\psi_s\phi_s(a_q)$ and
$\psi_s\phi_s\d_{s-1}(b_q) = 0 = \d_s\psi_s\phi_s(a_q)$; and if $q<s$,
we have $\psi_s\phi_s\d_{s-1}(b_q) = a_{q-n} = \d_s\psi_s\phi_s(b_q)$.

It remains to check $x=a_q$ for $q\leq s$ and $x=b_s$. For $x=a_q$
with $q<s$, we have
\begin{align*}
\psi_s\phi_s\d_{s-1}(a_q) &= \psi_s\phi_s\left(-\sum_{m=s-n}^{q-1}
  \sigma\sigma_m a_m c_{mq} \right) \\
&= -\sum_{m=s-n}^{q-1} \sigma\sigma_m (a_m-\sigma\sigma_s b_s
c_{s,m+n}) c_{mq} \\
&= -\left(\sum_{m=s-n+1}^{q-1} \sigma\sigma_m a_m c_{mq}\right)
-\sigma\sigma_s a_{s-n}c_{s,q+n}+b_s \d(c_{s,q+n}) \\
&= \d_s(a_q-\sigma\sigma_s b_s c_{s,q+n}) \\
&= \d_s\psi_s\phi_s(a_q),
\end{align*}
while for $x=a_s$, we have
\begin{align*}
\psi_s\phi_s\d_{s-1}(a_s) &= \psi_s\left(b_s-\sum_{m=s-n}^{s-1}
  \sigma\sigma_m a_m c_{ms}\right) \\
&= b_s - \sum_{m=s-n}^{s-1} \sigma\sigma_ma_mc_{ms} +
\sum_{m=s-n}^{s-1} \sigma_m\sigma_s b_sc_{s,m+n}c_{m+n,s+n} \\
&=
b_s+(\d_s(a_s)-\sigma\sigma_{s-n}a_{s-n}c_{s-n,s})+b_s(\d(c_{s,s+n})-1)
\\
&= \d_s(a_s-\sigma\sigma_s b_s c_{s,s+n}) \\
&= \d_s\psi_s\phi_s(a_s).
\end{align*}
Finally, for $x=b_s$, we have
\begin{align*}
\d_s\psi_s\phi_s(b_s) &= \d_s(b_s)-\d_s\psi_s\phi_s\left(
\sum_{m=s-n}^{s-1} \sigma\sigma_m a_m c_{ms} \right) \\
&= a_{s-n}-\psi_s\phi_s \d_{s-1}\left(\sum_{m=s-n}^{s-1}
  \sigma\sigma_m a_m c_{ms}\right) \\
&= a_{s-n}-\psi_s\phi_s(a_{s-n}) \\
&= 0 \\
&= \psi_s\phi_s\d_{s-1}(b_s),
\end{align*}
where the second equality follows from $\psi_s\phi_s\d_{s-1}(x) =
\d_s\psi_s\phi_s(x)$ for $x=a_m$ and $x=c_{ms}$, and the third
equality is a direct computation:
\begin{align*}
\d_{s-1}&\left(\sum_{m=s-n}^{s-1}
  \sigma\sigma_m a_m c_{ms}\right)\\
&=
\sum_{m=s-n}^{s-1} \left(-\sum_{\ell=s-n}^{m-1} \sigma_m\sigma_\ell
  a_\ell c_{\ell m} c_{ms} +
  a_m\left(\sum_{\ell=m+1}^{s-1}\sigma_m\sigma_\ell c_{m\ell}c_{\ell
      s}+\delta_{m+n,s} \right)\right) \\
&= a_{s-n}-\sum_{s-n\leq\ell<m\leq s-1} \sigma_m\sigma_\ell a_\ell
c_{\ell m} c_{ms} +
\sum_{s-n\leq m<\ell\leq s-1} \sigma_m\sigma_\ell a_mc_{m\ell}c_{\ell
  s}\\
&= a_{s-n}.
\end{align*}
This completes the verification of the lemma for all generators of
$\tilde{\A}_0$.
\end{proof}

For $s\geq n+1$, define $\Phi_s :\thinspace \tilde{\A}_0 \to
\tilde{\A}_0$ by
\[
\Phi_s = \psi_s \circ \phi_s \circ \psi_{s-1} \circ \phi_{s-1} \circ
\cdots \circ \psi_{n+1} \circ \phi_{n+1}.
\]
By Lemma~\ref{lem:gompf6psiphi}, we have
$\Phi_s\circ\d_n\circ\Phi_s^{-1} = \d_s$ on $\tilde{\A}_0$.

Now define $\d_\infty :\thinspace \tilde{\A}_0\to\tilde{\A}_0$ by
\begin{align*}
\d_\infty(c_{ij}^p) &= \d(c_{ij}^p) && \\
\d_\infty(a_q) &= 0 && q\geq 1 \\
\d_\infty(b_q) &= a_{q-n} && q\geq n+1.
\end{align*}
Intuitively, this is the limit of $\d_s$ as $s\to\infty$. Note that on
$\tilde{\A}_0 \subset \tilde{\A}$, $\d_\infty$ agrees with $\d''$. If
we can construct the ``limit'' automorphism $\Phi = \lim_{s\to\infty}
\Phi_s$, then we should have $\Phi\circ\d_n\circ\Phi^{-1} = \d_\infty$
on $\tilde{\A}_0$, and $\Phi$ intertwines $\d=\d_n$ and
$\d''=\d_\infty$ on $\tilde{\A}_0$.

To make this rigorous, consider the following algebra maps
$\Psi_1,\Psi_2:\thinspace \tilde{\A}_0\to\tilde{\A}_0$: $\Psi_1$ is
the identity on all generators except for the $a_q$, and
\[
\Psi_1(a_q) = a_q-\sum_{m=\max(q,n+1)}^{q+n-1} \sigma\sigma_m b_m
c_{m,q+n}
\]
for $q\geq 1$; $\Psi_2$ is the identity on all generators except for
the $b_q$, and
\[
\Psi_2(b_q) = b_q - \sum_{m=q-n}^{q-1} \sigma\sigma_m
(\psi_{n+1}^{-1}\psi_{n+2}^{-1}\cdots\psi_{q-1}^{-1}(a_m))c_{mq}
\]
for $q\geq n+1$. Then $\Psi_1$ is an elementary automorphism with respect to
any ordering of $\tilde{\A}_0$ for which
\[
c_{ij}^0,c_{ij}^1 < a_1 < b_{n+1} < a_2 < b_{n+2} < a_3 < b_{n+3} <
\cdots
\]
where $c_{ij}^0,c_{ij}^1$ are ordered before all $a_q$ and $b_q$
for all $1\leq i,j\leq n$, and $\Psi_2$ is an elementary automorphism
with respect to any ordering for which
\[
c_{ij}^0,c_{ij}^1 < a_1<\cdots < a_n<b_{n+1}<a_{n+1}<b_{n+2}<\cdots.
\]
(To verify this second statement, note that
$\psi_s^{-1}(a_m) = a_m+\sigma\sigma_sb_sc_{s,m+n}$ for $m\leq s$, and
so when $m\leq q-1$,
$\psi_{n+1}^{-1}\cdots\psi_{q-1}^{-1}(a_m)$ can only involve $a_m$,
$b_s$ with $s\leq q-1$, and $c_{s,m+n}$ with $s\geq n+1$.)

We now define the map $\Psi :\thinspace \tilde{\A}_0 \to \tilde{\A}_0$
by $\Psi = \Psi_1\circ\Psi_2$. By the above discussion, $\Psi$ is a
tame automorphism of $\tilde{\A}_0$.

\begin{lemma}
\begin{enumerate}
\item \label{gompf6lim1}
$\d_s(a_q) = \d_\infty(a_q)$ for $s\geq q+n-1$, and
$\d_s(b_q) = \d_\infty(b_q)$ for $s \geq q$.
\item \label{gompf6lim2}
$\Phi_s(a_q) = \Psi(a_q)$ for $s\geq q+n-1$, and
$\Phi_s(b_q) = \Psi(b_q)$ for $s\geq q+n-2$.
\end{enumerate}
\label{lem:gompf6limit}
\end{lemma}

\begin{proof}
(\ref{gompf6lim1}) is clear from the definition of $\d_s$ and
$\d_\infty$.

To establish (\ref{gompf6lim2}), rewrite $\Phi_s$ for $s\geq n+1$ as follows:
\begin{equation}
\Phi_s = \psi_s \circ \psi_{s-1} \circ \cdots \circ \psi_{n+1} \circ
\tilde{\phi}_s \circ \tilde{\phi}_{s-1} \circ \cdots \circ
\tilde{\phi}_{n+1},
\label{eq:gompf6lim}
\end{equation}
where $\tilde{\phi}_q = (\psi_{q-1}\cdots \psi_{n+1})^{-1}
\circ\phi_q\circ(\psi_{q-1}\cdots \psi_{n+1})$ for $q\geq n+1$. Since
$\phi_q$ acts as the identity on all generators but $b_q$, and
the definitions of $\psi_{q-1},\ldots,\psi_{n+1}$ do not involve
$b_q$, we conclude that $\tilde{\phi}_q$ acts as the identity on
all generators of $\tilde{\A}_0$ except $b_q$, and
\begin{align*}
\tilde{\phi}_q(b_q) &=
\psi_{n+1}^{-1}\cdots\psi_{q-1}^{-1}(\phi_q(b_q)) =
b_q - \sum_{m=q-n}^{q-1}
\sigma\sigma_m(\psi_{n+1}^{-1}\cdots\psi_{q-1}^{-1})(a_m) c_{mq}\\
&= \Psi_2(b_q).
\end{align*}
Note that by the discussion preceding the lemma,
$\Psi_2(b_q)$ involves only $a_\ell$ with $\ell<q$, along with $b$'s
and $c$'s.

From \eqref{eq:gompf6lim}, we have $\Phi_s(a_q) =
(\psi_s\cdots\psi_{n+1})(a_q)$. Using the definitions of $\Phi_s$, $\phi_s$,
and $\psi_s$, one can easily check by
induction on $s\geq n+1$ that
\begin{align*}
\Phi_s(a_q) &= (\psi_s\cdots\psi_{n+1})(a_q)\\
& = \begin{cases} a_q, & s<q \\
a_q - \sum_{m=\max(q,n+1)}^s \sigma\sigma_m b_m c_{m,q+n}, & q\leq
s\leq q+n-2, \\
a_q - \sum_{m=\max(q,n+1)}^{q+n-1} \sigma\sigma_m b_m c_{m,q+n}, & s
\geq q+n-1,
\end{cases}
\end{align*}
and it follows that $\Phi_s(a_q) = \Psi_1(a_q) = \Psi(a_q)$ when $s\geq q+n-1$.

Finally, let $s\geq q+n-2$. We have just shown that
$(\psi_s\cdots\psi_{n+1})(x) = \Psi_1(x)$ if $x=a_\ell$ with $\ell<q$,
and the same equation holds trivially if $x$ is a $b$ or $c$
generator. It follows that
\begin{align*}
\Phi_s(b_q) &= (\psi_s\cdots\psi_{n+1})(\tilde{\phi}_q(b_q)) =
(\psi_s\cdots\psi_{n+1})(\Psi_2(b_q)) =
\Psi_1(\Psi_2(b_q))\\
&=\Psi(b_q),
\end{align*}
and this completes the proof of (\ref{gompf6lim2}).
\end{proof}

We are now in position to prove Lemma~\ref{lem:gompf6more}.

\begin{proof}[Proof of Lemma~\ref{lem:gompf6more}]
The tame automorphism $\Psi$ of $\tilde{\A}_0$ is the one defined in
the above discussion. Properties (\ref{gompf6more1}) and
(\ref{gompf6more2}) in the statement of Lemma~\ref{lem:gompf6more} are
direct consequences of the construction of $\Psi$. It remains to establish
property (\ref{gompf6more3}), which we can do by checking that
$\Psi\d(x) = \d''\Psi(x)$ for all generators $x$ of $\tilde{\A}_0$.

If $x=c_{ij}^p$, then
$\Psi\d(c_{ij}^p) = \d(c_{ij}^p) = \d''(c_{ij}^p) =
\d''\Psi(c_{ij}^p)$. If $x=a_q$ with $q\geq 1$, then choose any $s
\geq q+n-1$. By Lemma~\ref{lem:gompf6limit}(\ref{gompf6lim2}), we have
$\Phi_s(a_q) = \Psi(a_q)$;
since $\d_n(a_q)$ involves only $a_1,\ldots,a_{q-1}$ and $c$, we
also have
$\Psi\d_n(a_q) = \Phi_s\d_n(a_q)$. Furthermore, $\Psi(a_q) =
\Psi_1(a_q)$ involves only $a_\ell$ with $\ell\leq q$, $b_\ell$
with $\ell\leq n+q-1$, and $c$, and so by
Lemma~\ref{lem:gompf6limit}(\ref{gompf6lim1}) we have
$\d_\infty(\Psi(a_q))=\d_s(\Psi(a_q))$. Thus
\begin{align*}
\Psi\d(a_q) &= \Psi\d_n(a_q) = \Phi_s\d_n(a_q) = \d_s\Phi_s(a_q)
= \d_s\Psi(a_q)\\
& = \d_\infty\Psi(a_q) = \d''\Psi(a_q).
\end{align*}

Finally, consider the case $x=b_q$ with $q\geq n+1$, and choose any
$s\geq q+n-2$. From the discussion preceding
Lemma~\ref{lem:gompf6limit}, $\Psi_2(b_q)$ involves only $c$,
$a_\ell$ with $\ell\leq q-1$, and $b_\ell$ with $\ell\leq q$; then
$\Psi(b_q) = \Psi_1(\Psi_2(b_q))$ involves only $c$,
$a_\ell$ with $\ell\leq q-1$, and $b_\ell$ with $\ell\leq
n+q-2$. By Lemma~\ref{lem:gompf6limit}(\ref{gompf6lim1}), it follows
that $\d_s\Psi(b_q) = \d_\infty\Psi(b_q)$. Also, by
Lemma~\ref{lem:gompf6limit}(\ref{gompf6lim2}), $\Phi_s(b_q) =
\Psi(b_q)$. Thus
\[
\Psi\d(b_q) = 0 = \Phi_s\d_n(b_q) = \d_s\Phi_s(b_q)
= \d_s\Psi(b_q) = \d_\infty\Psi(b_q) = \d''\Psi(b_q)
\]
and the proof is complete.
\end{proof}

\subsubsection{Proof of Theorem~\ref{thm:invariance}}
\label{sssec:combpf}
Consider two front diagrams that represent
Legendrian-isotopic links in normal form in $\#^k(S^1\times S^2)$. By
\cite{bib:Gompf}, these diagrams are related by a sequence of the
basic moves mentioned above. (Since our definition of the DGA requires a choice of base point on each component of the link, we should also include moves that fix the fronts but move base points along the components. However, such moves simply change the differential by conjugation by a particular tame automorphism of the algebra, replacing each generator by itself times some product of powers of $t_i$. See \cite[section~2.5]{bib:NgCLI}.)

By Proposition~\ref{prop:isotopyinv}, two Legendrian links in
normal form that are related by Legendrian isotopy inside the
box $[0,A] \times [-M,M]$ have DGAs that are stable tame
isomorphic. Propositions~\ref{prop:gompf4inv}, \ref{prop:gompf5inv},
and
\ref{prop:gompf6inv} show invariance under Gompf moves 4, 5, and 6.
Invariance for the alternate ($180^\circ$ rotated) versions of Gompf
moves 4 and 6 follows from invariance under the usual versions, along
with Lemma~\ref{lem:rotation} and the fact that two semifree DGAs
$(\A,\d),(\A',\d')$ are
stable tame isomorphic if and only if $(\A,\d_{\op}), (\A',\d'_{\op})$
are stable tame isomorphic. \qed
%\end{proof}

%*********************************************************************
%*********************************************************************

\bibliographystyle{plain}
\bibliography{biblio}

\end{document}